\documentclass[journal,twoside]{IEEEtran}

\usepackage{cite}

\usepackage{amsmath}
\usepackage{array}

\usepackage[pdftex]{graphicx}

\usepackage{amsfonts,amssymb,amsthm,mathtools,bbm,verbatim}
\allowdisplaybreaks 
\usepackage[usenames,dvipsnames]{color}
\usepackage[]{inputenc}
\usepackage{hyperref}
\hypersetup{colorlinks = true, linkcolor = blue, citecolor = blue, urlcolor  = blue}
\usepackage[all]{hypcap} 
\usepackage{algorithm}
\usepackage{algpseudocode} 
\algnewcommand\True{\textbf{true}\space}
\algnewcommand\False{\textbf{false}\space}
\usepackage{balance}

\newcommand\Tstrut{\rule{0pt}{2.1ex}} 
\makeatletter 
\def\bbordermatrix#1{\begingroup \m@th
  \@tempdima 4.75\p@
  \setbox\z@\vbox{%
    \def\cr{\crcr\noalign{\kern2\p@\global\let\cr\endline}}%
    \ialign{$##$\hfil\kern2\p@\kern\@tempdima&\thinspace\hfil$##$\hfil
      &&\quad\hfil$##$\hfil\crcr
      \omit\strut\hfil\crcr\noalign{\kern-\baselineskip}%
      #1\crcr\omit\strut\cr}}%
  \setbox\tw@\vbox{\unvcopy\z@\global\setbox\@ne\lastbox}%
  \setbox\tw@\hbox{\unhbox\@ne\unskip\global\setbox\@ne\lastbox}%
  \setbox\tw@\hbox{$\kern\wd\@ne\kern-\@tempdima\left[\kern-\wd\@ne
    \global\setbox\@ne\vbox{\box\@ne\kern2\p@}%
    \vcenter{\kern-\ht\@ne\unvbox\z@\kern-\baselineskip}\,\right]$}%
  \null\;\vbox{\kern\ht\@ne\box\tw@}\endgroup}
\makeatother

\def\G{{\mathcal{G}}}
\def\C{{\mathcal{C}}}

\def\Y{{\mathcal{Y}}}

\def\F{{\mathcal{F}}}
\def\V{{\mathcal{V}}}
\def\W{{\mathcal{W}}}
\def\S{{\mathcal{S}}}
\def\CE{{\mathcal{E}}}
\def\Nbhd{{\mathcal{N}}}
\def\P{{\mathbb{P}}}
\def\E{{\mathbb{E}}}

\def\R{{\mathbb{R}}}
\def\N{{\mathbb{N}}}
\def\Z{{\mathbb{Z}}}
\def\Q{{\mathbb{Q}}}
\def\1{{\mathbf{1}}}
\def\0{{\mathbf{0}}}
\def\I{{\mathbbm{1}}}

\def\w{{w_{\delta}}}

\DeclareMathOperator*{\br}{br}

\DeclareMathOperator*{\rank}{rank}
\DeclareMathOperator*{\kernel}{ker}
\DeclareMathOperator*{\linspan}{span}
\DeclareMathOperator*{\divergence}{div}

\def\maj{{\mathsf{majority}}}
\def\Ber{{\mathsf{Bernoulli}}}
\def\cyceq{{\, \simeq_{\mathsf{c}} \,}}
\def\cycgeq{{\, \succeq_{\mathsf{c}} \,}}

\newcommand{\Mod}[1]{\ (\mathrm{mod}\ #1)}
\newcommand{\T}{\mathrm{T}}

\newtheorem{theorem}{Theorem}
\newtheorem{lemma}{Lemma}
\newtheorem{proposition}{Proposition}

\theoremstyle{definition}
\newtheorem{definition}{Definition}

\newtheorem{conjecture}{Conjecture}

\begin{document}

\bstctlcite{IEEEexample:BSTcontrol} 

\title{Broadcasting on Two-Dimensional Regular Grids}

\author{Anuran~Makur,~Elchanan~Mossel,~and~Yury~Polyanskiy%
\thanks{This work was supported in part by the Center for Science of Information, a National Science Foundation (NSF) Science and Technology Center, under Grant CCF-09-39370, in part by the NSF under Grants CCF-17-17842, CCF-19-18421, and DMS-17-37944, in part by the MIT-IBM Watson AI Lab, in part by the Vannevar Bush Faculty Fellowship under Grant N00014-20-1-2826, in part by the Simons Investigator Award under Grant 622132, and in part by the Army Research Office (ARO) Multidisciplinary University Research Initiative (MURI) under Grant W911NF-19-1-0217. This work was presented at the 2021 IEEE International Symposium on Information Theory \cite{MakurMosselPolyanskiy2021}.}%
\thanks{A. Makur is with the Department of Computer Science and the Elmore Family School of Electrical and Computer Engineering, Purdue University, West Lafayette, IN 47907, USA (e-mail: amakur@purdue.edu).}%
\thanks{E. Mossel is with the Department of Mathematics, Massachusetts Institute of Technology, Cambridge, MA 02139, USA (e-mail: elmos@mit.edu).}%
\thanks{Y. Polyanskiy is with the Department of Electrical Engineering and Computer Science, Massachusetts Institute of Technology, Cambridge, MA 02139, USA (e-mail: yp@mit.edu).}%
\thanks{Copyright (c) 2022 IEEE. Personal use of this material is permitted. However, permission to use this material for any other purposes must be obtained from the IEEE by sending a request to pubs-permissions@ieee.org.}}%

\maketitle

\begin{abstract}
We study an important specialization of the general problem of broadcasting on directed acyclic graphs, namely, that of broadcasting on two-dimensional (2D) regular grids. Consider an infinite directed acyclic graph with the form of a 2D regular grid, which has a single source vertex $X$ at layer $0$, and $k + 1$ vertices at layer $k \geq 1$, which are at a distance of $k$ from $X$. Every vertex of the 2D regular grid has outdegree $2$, the vertices at the boundary have indegree $1$, and all other non-source vertices have indegree $2$. At time $0$, $X$ is given a uniform random bit. At time $k \geq 1$, each vertex in layer $k$ receives transmitted bits from its parents in layer $k-1$, where the bits pass through independent binary symmetric channels with common crossover probability $\delta \in \big(0,\frac{1}{2}\big)$ during the process of transmission. Then, each vertex at layer $k$ with indegree $2$ combines its two input bits using a common deterministic Boolean processing function to produce a single output bit at the vertex. The objective is to recover $X$ with probability of error better than $\frac{1}{2}$ from all vertices at layer $k$ as $k \rightarrow \infty$. Besides their natural interpretation in the context of communication networks, such broadcasting processes can be construed as one-dimensional (1D) probabilistic cellular automata, or discrete-time statistical mechanical spin-flip systems on 1D lattices, with boundary conditions that limit the number of sites at each time $k$ to $k+1$. Inspired by the literature surrounding the ``positive rates conjecture'' for 1D probabilistic cellular automata, we conjecture that it is impossible to propagate information in a 2D regular grid regardless of the noise level $\delta$ and the choice of common Boolean processing function. In this paper, we make considerable progress towards establishing this conjecture, and prove using ideas from percolation and coding theory that recovery of $X$ is impossible for any $\delta \in \big(0,\frac{1}{2}\big)$ provided that all vertices with indegree $2$ use either AND or XOR for their processing functions. Furthermore, we propose a detailed and general martingale-based approach that establishes the impossibility of recovering $X$ for any $\delta \in \big(0,\frac{1}{2}\big)$ when all NAND processing functions are used if certain structured supermartingales can be rigorously constructed. We also provide strong numerical evidence for the existence of these supermartingales by computing several explicit examples for different values of $\delta$ via linear programming.
\end{abstract}


\tableofcontents
\hypersetup{linkcolor = red}

\section{Introduction}
\label{Introduction}

The problem of broadcasting on two-dimensional (2D) regular grids is an important specialization of the broader question of broadcasting on bounded indegree directed acyclic graphs (DAGs), which was introduced in \cite{MakurMosselPolyanskiy2019,MakurMosselPolyanskiy2020,Makur2019} to generalize the classical problem of broadcasting on trees and Ising models, cf. \cite{Evansetal2000}. In contrast to the canonical study of reliable communication through broadcast channels in network information theory \cite[Chapters 5 and 8]{ElGamalKim2011}, the broadcasting problem, as studied in this paper, analyzes whether or not the ``wavefront of information'' emitted (or broadcasted) by a single transmitter decays irrecoverably as it propagates through a large, and typically structured, Bayesian network. For example, in the broadcasting on trees setting, we are given a Bayesian network whose underlying DAG is a rooted tree $T$, vertices are Bernoulli random variables, and edges are independent binary symmetric channels (BSCs) with common crossover probability $\delta \in \big(0,\frac{1}{2}\big)$. The root contains a uniform random bit that it transmits through $T$, and our goal is to decode this bit from the received values of the vertices at an arbitrarily deep layer $k$ (i.e. at distance $k$ from the root) of $T$. It is proved in a sequence of papers \cite{KestenStigum1966,BleherRuizZagrebnov1995,Evansetal2000} that the root bit can be decoded with minimum probability of error bounded away from $\frac{1}{2}$ as $k \rightarrow \infty$ if and only if $(1-2\delta)^2 \br(T) > 1$, i.e. the noise level $\delta$ is strictly less than the critical \textit{Kesten-Stigum threshold} $\frac{1}{2}\big(1 - (1/\br(T))^{1/2}\big)$, which depends on the branching number $\br(T)$ (see \cite[Chapter 1.2]{LyonsPeres2017}). This key result and its generalizations, cf. \cite{Ioffe1996a,Ioffe1996b,Mossel1998,Mossel2001,PemantlePeres2010,Sly2009,Sly2011,JansonMossel2004,Bhatnagaretal2011}, precisely characterize when information about the root bit contained in the vertices at arbitrarily deep layers of a Bayesian network with a tree-structured topology vanishes completely.

Although broadcasting on trees is amenable to various kinds of tractable analysis, the general problem of broadcasting on bounded indegree DAGs arguably better models real-world communication or social networks, where each vertex or agent usually receives multiple noisy input signals and has to judiciously consolidate these signals using simple rules. In the broadcasting on DAGs setting, we are given a Bayesian network with a single unbiased Bernoulli source vertex such that all other Bernoulli vertices have bounded indegree, and all edges are independent BSCs with noise level $\delta \in \big(0,\frac{1}{2}\big)$. Moreover, the vertices with indegree larger than $1$ compute their values by applying Boolean processing functions to their noisy inputs. Determining the precise conditions on the graph topology (e.g. the bound on the indegrees), the noise level $\delta$, and the choices of Boolean processing functions that permit successful reconstruction of the source bit is quite challenging. As a result, we usually characterize when reconstruction is possible for specific classes of DAGs and choices of processing functions.

For instance, our earlier work \cite{MakurMosselPolyanskiy2019,MakurMosselPolyanskiy2020} studies randomly constructed DAGs with indegrees $d$ and layer sizes $L_k$, where $L_k$ denotes the number of vertices at layer $k$ (i.e. at distance $k$ from the source vertex). It is established in \cite[Theorem 1]{MakurMosselPolyanskiy2020} that if $d \geq 3$ and all majority processing functions are used, then reconstruction of the source bit is possible using the majority decoder when $\delta < \delta_{\mathsf{maj}}(d)$ and $L_k \geq C(\delta,d) \log(k)$ for all sufficiently large $k$, where $\delta_{\mathsf{maj}}(d)$ is a known critical threshold (cf. \cite[Equation (11)]{MakurMosselPolyanskiy2020}) and $C(\delta,d) > 0$ is some fixed constant. Furthermore, \cite[Theorem 2]{MakurMosselPolyanskiy2020} shows a similar result when $d = 2$ and all NAND processing functions are used. (Partial converse results to \cite[Theorems 1 and 2]{MakurMosselPolyanskiy2020} are also developed in \cite{MakurMosselPolyanskiy2020}.) The aforementioned results demonstrate, after employing the probabilistic method, the existence of deterministic DAGs with bounded indegree and $L_k = \Theta(\log(k))$ such that reconstruction is possible for sufficiently small noise levels $\delta$. In fact, \cite[Theorem 3, Proposition 2]{MakurMosselPolyanskiy2020} also presents an explicit quasi-polynomial time construction of such deterministic DAGs using regular bipartite lossless expander graphs. In a nutshell, the results in \cite{MakurMosselPolyanskiy2019,MakurMosselPolyanskiy2020} illustrate that while $L_k$ must be exponential in $k$ for reconstruction to be possible in trees, logarithmic $L_k$ suffices for reconstruction in bounded indegree DAGs, because DAGs enable information fusion (or local ``error correction'') at the vertices. 

As opposed to the randomly constructed and expander-based DAGs analyzed in \cite{MakurMosselPolyanskiy2019,MakurMosselPolyanskiy2020}, in this paper, we consider the problem of broadcasting on another simple and important class of deterministic DAGs, namely, 2D regular grids. 2D regular grids correspond to DAGs with $L_k = k+1$ such that every vertex has outdegree $2$, the vertices at the boundary have indegree $1$, and all other non-source vertices have indegree $2$. For simplicity, we study the setting where the Boolean processing functions at all vertices with indegree $2$ are the same. As we will explain shortly, the literature on one-dimensional (1D) probabilistic cellular automata suggests that reconstruction of the source bit is impossible for such 2D regular grids regardless of the noise level $\delta$ and the choice of Boolean processing function. In this vein, the main contributions of this paper include two impossibility results that partially justify this intuition. In particular, we prove that recovery of the source bit is impossible on a 2D regular grid if all intermediate vertices with indegree $2$ use logical AND as the processing function, or all use logical XOR as the processing function. These proofs leverage ideas from percolation and coding theory. This leaves only NAND as the remaining symmetric processing function where the impossibility of reconstruction is unknown. Although we do not provide a complete proof of the impossibility of broadcasting with NAND processing functions, another main contribution of this paper is a careful elaboration of a martingale-based technique that yields the desired impossibility result if a certain family of superharmonic potential functions can be rigorously constructed. While such a (theoretical) construction currently remains open, we present some strong numerical evidence for it via linear programming. Furthermore, this martingale-based approach can easily be modified to obtain potential proofs for the impossibility of broadcasting on 2D regular grids for other choices of Boolean processing functions. Thus, we believe that it is instructive for future work in this area. 

\subsection{Motivation}
\label{Motivation}

As discussed in our earlier work \cite{MakurMosselPolyanskiy2020}, the general problem of broadcasting on bounded indegree DAGs is closely related to a myriad of problems in the literature. Besides its canonical broadcast interpretation in the context of communication networks, broadcasting on DAGs is a natural model of reliable computation and storage, cf. \cite{vonNeumann1956,HajekWeller1991,EvansPippenger1998,EvansSchulman1999,EvansSchulman2003,Unger2008}. Indeed, the model can be construed as a noisy circuit that has been constructed to remember (or store) a bit, where the edges are wires that independently make errors, and the Boolean processing functions at the vertices are perfect logic gates. Special cases of the broadcasting model on certain families of DAGs also correspond to well-known models in statistical physics. For example, broadcasting on trees corresponds to the extremality of certain Gibbs measures of ferromagnetic Ising models \cite[Section 2.2]{Evansetal2000}, and broadcasting on regular grids is closely related to the ergodicity of discrete-time statistical mechanical spin-flip systems (such as probabilistic cellular automata) on lattices \cite{Gray1982,Gray1987,Gray2001}. Furthermore, other special cases of the broadcasting model, such as on trees, represent information flow in biological networks, cf. \cite{Mossel2003,Mossel2004,DaskalakisMosselRoch2006,Roch2010}, play a crucial role in random constraint satisfaction problems, cf. \cite{MezardMontanari2006,Krzakalaetal2007,GerschenfeldMontanari2007,MontanariRestrepoTetali2011}, and are useful in proving converse results for community detection in stochastic block models, cf. \cite[Section 5.1]{Abbe2018}.

The main motivation for this work, i.e. the problem of understanding whether it is possible to propagate information starting from the source specifically in regular grids (see Figure \ref{Figure: Grid} for a 2D example), stems from the theory of \emph{probabilistic cellular automata} (PCA). Our \emph{conjecture} is that such propagation is possible for sufficiently low noise levels $\delta$ in $3$ or more dimensions, and impossible for 2D regular grids regardless of the noise level and the choice of Boolean processing function (which is the same for every vertex). This conjecture resembles and is inspired by the literature on the \emph{positive rates conjecture} for 1D PCA, cf. \cite[Section 1]{Gray2001}, and the existence of non-ergodic 2D PCA such as that defined by \emph{Toom's North-East-Center (NEC) rule}, cf. \cite{Toom1980}. Notice that broadcasting processes on 2D regular grids can be perceived as 1D PCA with boundary conditions that limit the layer sizes to be $L_k = k + 1$, and the impossibility of reconstruction on 2D regular grids intuitively corresponds to the ergodicity of 1D PCA (along with sufficiently fast convergence rate to stationarity). Therefore, the existence of a 2D regular grid with a choice of Boolean processing function which remembers its initial state bit for infinite time would suggest the existence of non-ergodic infinite 1D PCA consisting of $2$-input binary-state cells. However, the positive rates conjecture states (in essence) that ``relatively simple'' 1D PCA with local interactions and strictly positive noise probabilities are ergodic, and known counter-example constructions to this conjecture either require a lot more states \cite{Gacs2001}, or are non-uniform in time and space \cite{Cirelson1978}. This intuition gives credence to our conjecture that broadcasting is impossible for 2D regular grids. Furthermore, much like 2D regular grids, broadcasting on three-dimensional (3D) regular grids can be perceived as 2D PCA with boundary conditions (as shown in subsection \ref{3D Regular Grid Model}). Hence, the existence of non-ergodic 2D PCA, such as that in \cite{Toom1980}, suggests the existence of 3D regular grids where broadcasting is possible, thereby lending further credence to our conjecture. In this paper, we take some first steps towards establishing the 2D part of our larger conjecture.

\begin{figure}[t]
\centering
\includegraphics[trim = 50mm 180mm 50mm 25mm, width=0.75\linewidth]{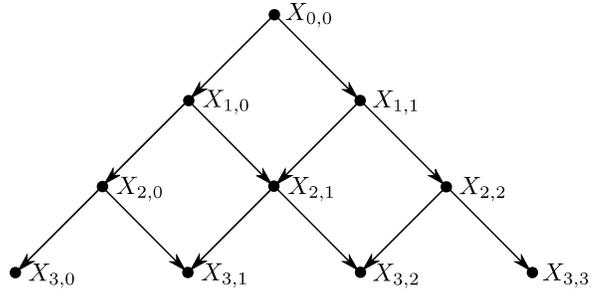} 
\caption{Illustration of a 2D regular grid where each vertex is a Bernoulli random variable and each edge is a BSC with parameter $\delta \in \big(0,\frac{1}{2}\big)$. Moreover, each vertex with indegree $2$ uses a common Boolean processing function to combine its noisy input bits.}
\label{Figure: Grid}
\end{figure}

Since the thrust of this paper partly hinges on the positive rates conjecture, we close this subsection by briefly expounding the underlying intuition behind it. Borrowing terminology from statistical physics, where a spin-flip system is said to experience a \emph{phase transition} if it has more than one invariant measure, the positive rates conjecture informally states that there are ``no phase transitions in one dimension'' \cite[Section 1]{Gray2001}. In fact, this statement is true for the special (and better understood) case of Ising models, where there is a phase transition in 2D when the temperature is sufficiently small, but not in 1D \cite[Section 2]{Gray2001}. Building intuition off of Ising models, it is explained in \cite[Section 2]{Gray2001} that when a binary-state 2D PCA is non-ergodic, its multiple invariant measures are ``close to'' the various \emph{stable} ground state configurations, e.g. ``all $0$'s'' and ``all $1$'s.'' Furthermore, for a ground state configuration such as ``all $0$'s'' to be deemed stable, we require that finite ``islands'' of $1$'s that are randomly formed by the noise process are killed by the transition (or Boolean processing) functions of the automaton, e.g. Toom's NEC rule kills such finite islands starting with their corners. However, for 1D PCA, a transition function with finite interaction neighborhood that is at the boundary of a large finite island cannot easily distinguish the island from the true ground state. Therefore, we cannot expect stable ground state configurations to exist in such simple 1D PCA, and hence, it is conjectured that such 1D PCA are ergodic.

\subsection{Outline}
\label{Outline}

We now delineate the organization of the remainder of this paper. The next subsection \ref{Deterministic 2D Grid Model} formally defines the 2D regular grid model in order to present our results in future sections. Section \ref{Main Results} presents our main impossibility results for the AND and XOR cases, describes our partial impossibility result for the NAND case and provides some accompanying numerical evidence, discusses some other related impossibility results, and elucidates the connection between 3D regular grid models and a variant of Toom's PCA. Sections \ref{Analysis of Deterministic And Grid} and \ref{Analysis of Deterministic Xor Grid} contain the proofs of our main impossibility results for AND processing functions and XOR processing functions, respectively. Then, we derive our partial impossibility result in section \ref{Martingale Approach for NAND Processing Functions} by carefully expounding a promising approach for proving impossibility results for 2D regular grids with NAND, and possibly other, processing functions via the construction of pertinent supermartingales. Finally, we conclude our discussion and propose future research directions in section \ref{Conclusion}.

\subsection{2D Regular Grid Model}
\label{Deterministic 2D Grid Model}

A \textit{2D regular grid model} consists of an infinite DAG with vertices that are Bernoulli random variables (with range $\{0,1\}$) and edges that are independent BSCs. The root or source random variable of the grid is $X_{0,0} \sim \Ber\big(\frac{1}{2}\big)$, and we let $X_k = (X_{k,0},\dots,X_{k,k})$ be the vector of vertex random variables at distance (i.e. length of shortest path) $k \in \N \triangleq \{0,1,2,\dots\}$ from the root. So, $X_0 = (X_{0,0})$ and there are $k+1$ vertices at distance $k$. Furthermore, the 2D regular grid contains the directed edges $(X_{k,j},X_{k+1,j})$ and $(X_{k,j},X_{k+1,j+1})$ for every $k \in \N$ and every $j \in [k+1] \triangleq \{0,\dots,k\}$, where the notation $(v,w)$ denotes an edge from source vertex $v$ to destination vertex $w$. The underlying DAG of such a 2D regular grid is shown in Figure \ref{Figure: Grid}.

To construct a \textit{Bayesian network} (or directed graphical model) on this 2D regular grid, we fix some crossover probability parameter $\delta \in \big(0,\frac{1}{2}\big)$,\footnote{The cases $\delta = 0$ and $\delta = \frac{1}{2}$ are uninteresting because the former corresponds to a deterministic grid and the latter corresponds to an independent grid.} and two Boolean processing functions $f_1:\{0,1\} \rightarrow \{0,1\}$ and $f_{2}:\{0,1\}^2 \rightarrow \{0,1\}$. Then, for any $k \in \N\backslash\!\{0,1\}$ and $j \in \{1,\dots,k-1\}$, we define:\footnote{We can similarly define a more general model where every vertex $X_{k,j}$ has its own Boolean processing function $f_{k,j}$, but we will only analyze instances of the simpler model presented here.}
\begin{equation}
\label{Eq:det propagation 1}
X_{k,j} = f_{2}(X_{k-1,j-1} \oplus Z_{k,j,1},X_{k-1,j}\oplus Z_{k,j,2})
\end{equation}
and for any $k \in \N\backslash\!\{0\}$, we define:
\begin{equation}
\label{Eq:det propagation 2}
\begin{aligned}
X_{k,0} & = f_{1}(X_{k-1,0} \oplus Z_{k,0,2}) \\
\text{and} \enspace X_{k,k} & = f_{1}(X_{k-1,k-1} \oplus Z_{k,k,1})
\end{aligned}
\end{equation}
where $\oplus$ denotes addition modulo $2$, and $\{Z_{k,j,i} : k \in \N\backslash\!\{0\}, j \in [k+1], i \in \{1,2\}\}$ are independent and identically distributed (i.i.d.) $\Ber(\delta)$ random variables that are independent of everything else. This implies that each edge is a $\mathsf{BSC}(\delta)$, i.e. a BSC with parameter $\delta$. Together, \eqref{Eq:det propagation 1} and \eqref{Eq:det propagation 2} characterize the conditional distribution of any $X_{k,j}$ given its parents.

Observe that the sequence $\{X_k : k \in \N\}$ forms a Markov chain, and our goal is to determine whether or not the value at the root $X_{0}$ can be decoded from the observations $X_k$ as $k \rightarrow \infty$. Given $X_k$ for any fixed $k \in \N\backslash\!\{0\}$, inferring the value of $X_0$ is a binary hypothesis testing problem with minimum achievable probability of error:
\begin{equation}
\label{Eq: ML Probability of Error}
\P\!\left(h_{\mathsf{ML}}^k(X_k) \neq X_0\right) = \frac{1}{2}\left(1-\left\|P_{X_k}^+ - P_{X_k}^-\right\|_{\mathsf{TV}}\right)
\end{equation}
where $h_{\mathsf{ML}}^k:\{0,1\}^{k+1} \rightarrow \{0,1\}$ is the maximum likelihood (ML) decision rule based on $X_k$ at level $k$ (with knowledge of the 2D regular grid), $P_{X_k}^+$ and $P_{X_k}^-$ are the conditional distributions of $X_k$ given $X_0 = 1$ and $X_0 = 0$, respectively, and for any two probability measures $P$ and $Q$ on the same measurable space $(\Omega,\F)$, their \textit{total variation (TV) distance} is defined as:
\begin{equation}
\left\|P - Q\right\|_{\mathsf{TV}} \triangleq \sup_{A \in \F}{\left|P(A) - Q(A)\right|} \, .
\end{equation}
It is straightforward to verify that the sequence $\{\P(h_{\mathsf{ML}}^k(X_k) \neq X_0) : k \in \N\backslash\!\{0\}\}$ is non-decreasing and bounded above by $\frac{1}{2}$. Indeed, the monotonicity of this sequence is a consequence of \eqref{Eq: ML Probability of Error} and the data processing inequality for TV distance, and the upper bound on the sequence trivially holds because a randomly generated bit cannot beat the ML decoder. Therefore, the limit of this sequence exists, and we say that reconstruction of the root bit $X_0$ is impossible, or ``broadcasting is impossible,'' when:
\begin{equation}
\label{Eq: Deterministic Impossibility of Reconstruction}
\begin{aligned}
\lim_{k \rightarrow \infty}&{\P\!\left(h_{\mathsf{ML}}^k(X_k) \neq X_0\right)} = \frac{1}{2} \\
& \Leftrightarrow \quad \lim_{k \rightarrow \infty}{\left\|P_{X_k}^+ - P_{X_k}^-\right\|_{\mathsf{TV}}} = 0 
\end{aligned}
\end{equation}
where the equivalence follows from \eqref{Eq: ML Probability of Error}.\footnote{Likewise, we say that reconstruction is possible, or ``broadcasting is possible,'' when $\lim_{k \rightarrow \infty}{\P(h_{\mathsf{ML}}^k(X_k) \neq X_0)} < \frac{1}{2}$, or equivalently, $\lim_{k \rightarrow \infty}{\|P_{X_k}^+ - P_{X_k}^-\|_{\mathsf{TV}}} > 0$.} In every impossibility result in this paper, we will prove that reconstruction of $X_0$ is impossible in the sense of \eqref{Eq: Deterministic Impossibility of Reconstruction}.

\section{Main Results and Discussion}
\label{Main Results}

In this section, we state our main results, briefly delineate the main techniques or intuition used in the proofs, and discuss related results and models in the literature.

\subsection{Impossibility Results for AND and XOR 2D Regular Grids}
\label{Impossibility Results on 2D Regular Grids}

In contrast to our work in \cite{MakurMosselPolyanskiy2020} where we analyze broadcasting on random DAGs, the deterministic 2D regular grids we now study are much harder to analyze due to the dependence between adjacent vertices in a given layer. So, as we will explain later, we have to employ seemingly ad hoc proof techniques from percolation theory, coding theory, and martingale theory in this paper instead of the simple and elegant fixed point iteration intuition exploited in \cite{MakurMosselPolyanskiy2020}. As mentioned earlier, we analyze the setting where all Boolean processing functions in the 2D regular grid with two inputs are the same, and all Boolean processing functions in the 2D regular grid with one input are the \textit{identity} rule.

Our first main result shows that reconstruction is impossible for all $\delta \in \big(0,\frac{1}{2}\big)$ when AND processing functions are used.

\begin{theorem}[AND 2D Regular Grid]
\label{Thm: Deterministic And Grid}
If $\delta \in \big(0,\frac{1}{2}\big)$, and all Boolean processing functions with two inputs in the 2D regular grid are the AND rule, then broadcasting is impossible in the sense of \eqref{Eq: Deterministic Impossibility of Reconstruction}:
$$ \lim_{k \rightarrow \infty}{\left\|P_{X_k}^+ - P_{X_k}^-\right\|_{\mathsf{TV}}} = 0 \, . $$
\end{theorem} 

Theorem \ref{Thm: Deterministic And Grid} is proved in section \ref{Analysis of Deterministic And Grid}. The proof couples the 2D regular grid starting at $X_{0,0} = 0$ with the 2D regular grid starting at $X_{0,0} = 1$, and ``runs'' them together. Using a phase transition result concerning \textit{bond percolation} on 2D lattices, we show that we eventually reach a layer where the values of all vertices in the first grid equal the values of the corresponding vertices in the second grid. So, the two 2D regular grids ``couple'' almost surely regardless of their starting state. This implies that we cannot decode the starting state by looking at vertices in layer $k$ as $k \rightarrow \infty$. We note that in order to prove that the two 2D regular grids ``couple,'' we have to consider two different regimes of $\delta$ and provide separate arguments for each. The details of these arguments are presented in section \ref{Analysis of Deterministic And Grid}. 

Our second main result shows that reconstruction is impossible for all $\delta \in \big(0,\frac{1}{2}\big)$ when XOR processing functions are used.

\begin{theorem}[XOR 2D Regular Grid]
\label{Thm: Deterministic Xor Grid}
If $\delta \in \big(0,\frac{1}{2}\big)$, and all Boolean processing functions with two inputs in the 2D regular grid are the XOR rule, then broadcasting is impossible in the sense of \eqref{Eq: Deterministic Impossibility of Reconstruction}:
$$ \lim_{k \rightarrow \infty}{\left\|P_{X_k}^+ - P_{X_k}^-\right\|_{\mathsf{TV}}} = 0 \, . $$
\end{theorem}

Theorem \ref{Thm: Deterministic Xor Grid} is proved in section \ref{Analysis of Deterministic Xor Grid}. In the XOR 2D regular grid, every vertex at level $k$ can be written as a (binary) linear combination of the source bit and all the BSC noise random variables in the grid up to level $k$. This linear relationship can be captured by a binary matrix. The main idea of the proof is to perceive this matrix as a parity check matrix of a linear code. The problem of inferring $X_{0,0}$ from $X_k$ turns out to be equivalent to the problem of decoding the first bit of a codeword drawn uniformly from this code after observing a noisy version of the codeword. Basic facts from coding theory can then be used to complete the proof.

We remark that at first glance, Theorems \ref{Thm: Deterministic And Grid} and \ref{Thm: Deterministic Xor Grid} seem intuitively obvious from the random DAG model perspective of \cite[Section I-C]{MakurMosselPolyanskiy2020}. For example, consider a random DAG model where the number of vertices at level $k \in \N$ is $L_k = k+1$, two incoming edges for each vertex in level $k \in \N\backslash\!\{0\}$ are chosen randomly, uniformly, and independently (with replacement) from the vertices in level $k-1$, and all Boolean processing functions are the AND rule. Then, letting $\sigma_k \triangleq \sum_{j = 0}^{k}{X_{k,j}/(k+1)} \in [0,1]$ be the proportion of $1$'s in level $k \in \N$, the conditional expectation function $g(\sigma) = \E[\sigma_k|\sigma_{k-1} = \sigma]$ has only one fixed point regardless of the value of $\delta \in \big(0,\frac{1}{2}\big)$, and we intuitively expect $\sigma_k$ to tend to this fixed point (which roughly captures the equilibrium between AND gates killing $1$'s and $\mathsf{BSC}(\delta)$'s producing new $1$'s) as $k \rightarrow \infty$. So, reconstruction is impossible in this random DAG model, which suggests that reconstruction is also impossible in the AND 2D regular grid. However, although Theorems \ref{Thm: Deterministic And Grid} and \ref{Thm: Deterministic Xor Grid} seem intuitively easy to understand in this way, we emphasize that this random DAG intuition does not capture the subtleties engendered by the ``regularity'' of the 2D grid. In fact, the intuition derived from the random DAG model can even be somewhat misleading. Consider the random DAG model described above with all NAND processing functions (instead of AND processing functions). This model was analyzed in \cite[Theorem 2]{MakurMosselPolyanskiy2020}, because using alternating layers of AND and OR processing functions is equivalent to using all NAND processing functions (see \cite[Footnote 10]{MakurMosselPolyanskiy2020}). \cite[Theorem 2]{MakurMosselPolyanskiy2020} portrays that reconstruction of the source bit is possible for $\delta < \big(3-\sqrt{7}\big)/4$. Yet, evidence from \cite[Theorem 1]{HolroydMarcoviciMartin2019}, which establishes the ergodicity of 1D PCA with NAND gates, and the detailed discussion, results, and numerical simulations in subsection \ref{Partial Impossibility Result for NAND 2D Regular Grid} and section \ref{Martingale Approach for NAND Processing Functions}, suggest that reconstruction is actually impossible for the 2D regular grid with NAND processing functions. Therefore, the 2D regular grid setting of Theorems \ref{Thm: Deterministic And Grid} and \ref{Thm: Deterministic Xor Grid} should be intuitively understood using random DAG models with caution. Indeed, as sections \ref{Analysis of Deterministic And Grid} and \ref{Analysis of Deterministic Xor Grid} illustrate, the proofs of these theorems are nontrivial. 

The impossibility of broadcasting in Theorems \ref{Thm: Deterministic And Grid} and \ref{Thm: Deterministic Xor Grid} also seems intuitively plausible due to the ergodicity results for numerous 1D PCA\textemdash see e.g. \cite{Gray1982} and the references therein. (Indeed, as we delineated in subsection \ref{Motivation}, our conjecture was inspired by the positive rates conjecture for 1D PCA.) However, there are four key differences between broadcasting on 2D regular grids and 1D PCA. Firstly, the main question in the study of 1D PCA is whether a given automaton is ergodic, i.e. whether the Markov process defined by it converges to a unique invariant probability measure on the configuration space for all initial configurations. This question of ergodicity is typically addressed by considering the convergence of finite-dimensional distributions over the sites (i.e. weak convergence). Hence, for many 1D PCA that have special characteristics, such as translation invariance, finite interaction range, positivity, and attractiveness (or monotonicity), cf. \cite{Gray1982}, it suffices to consider the convergence of distributions on finite intervals, e.g. marginal distribution at a given site. In contrast to this setting, we are concerned with the stronger notion of convergence in TV distance. Indeed, Theorems \ref{Thm: Deterministic And Grid} and \ref{Thm: Deterministic Xor Grid} show that the TV distance between $P_{X_k}^+$ and $P_{X_k}^{-}$ vanishes as $k \rightarrow \infty$.

Secondly, since a 1D PCA has infinitely many sites, the problem of remembering or storing a bit in a 1D PCA (with binary state space) corresponds to distinguishing between the ``all $0$'s'' and ``all $1$'s'' initial configurations. On the other hand, as mentioned in subsection \ref{Motivation}, a 2D regular grid can be construed as a 1D PCA with boundary conditions; each level $k \in \N$ corresponds to an instance in discrete-time, and there are $L_k = k+1$ sites at time $k$. Moreover, its initial configuration has only one copy of the initial bit as opposed to infinitely many copies. As a result, compared a 2D regular grid, a 1D PCA (without boundary conditions) intuitively appears to have a stronger separation between the two initial states as time progresses. The aforementioned boundary conditions form another barrier to translating results from the 1D PCA literature to 2D regular grids.

Thirdly, in our broadcasting model in subsection \ref{Deterministic 2D Grid Model}, the independent BSCs are situated on the edges of the 2D regular grid. On the other hand, a corresponding 1D PCA (which removes the boundary conditions of the 2D regular grid) first uses the unadulterated bits from the previous layer while computing its processing functions in the current layer, and then applies independent BSC noise to the outputs of the functions. Equivalently, a 1D PCA with BSC noise behaves like a 2D regular grid where the edges between levels $0$ and $1$ are noise-free, and for any site at any level $k \geq 1$, the BSCs of its two outgoing edges are coupled so that they flip their input bit simultaneously (almost surely). It is due this difference between our broadcasting model and the canonical 1D PCA model that we cannot, for example, easily translate the ergodicity result for 1D PCA with NAND gates and BSC noise in \cite[Theorem 1]{HolroydMarcoviciMartin2019} to an impossibility result for broadcasting on 2D regular grids with NAND processing functions. We additionally remark, for completeness, that there are some known connections between vertex noise and edge noise, cf. \cite{DobrushinOrtyukov1977}, but these results do not help with our problem.

Fourthly, it is also worth mentioning that most results on 1D PCA pertain to the continuous-time setting\textemdash see e.g. \cite{Liggett1978, Gray1982} and the references therein. This is because sites are (almost surely) updated one by one in a continuous-time automaton, but they are updated in parallel in a discrete-time automaton. So, the discrete-time setting is often harder to analyze. Indeed, some of the only known discrete-time 1D PCA ergodicity results are in \cite[Theorem 1]{HolroydMarcoviciMartin2019} and \cite[
Section 3]{Gray1987}, where the latter outlines the proof of ergodicity of the $3$-input majority vote model (i.e. 1D PCA with $3$-input majority gates) for sufficiently small noise levels.\footnote{As Gray explains in \cite[Section 3]{Gray1987}, his proof of ergodicity is not complete; he is ``very detailed for certain parts of the argument and very sketchy in others'' \cite{Gray1987}. Although the references in \cite{Gray1987} indicate that Gray was preparing a paper with the complete proof, this paper was never published to our knowledge. So, the ergodicity of 1D PCA with $3$-input majority gates has not been rigorously established.} This is another reason why results from the 1D PCA literature cannot be easily transferred to our model.

\subsection{Partial Impossibility Result for NAND 2D Regular Grid}
\label{Partial Impossibility Result for NAND 2D Regular Grid}

We next present our final main result for the 2D regular grid with all NAND processing functions. Based on our broader conjecture in subsection \ref{Motivation}, we first state the following conjecture in analogy with Theorems \ref{Thm: Deterministic And Grid} and \ref{Thm: Deterministic Xor Grid}.

\begin{conjecture}[NAND 2D Regular Grid]
\label{Conj: NAND 2D Regular Grid}
If $\delta \in \big(0,\frac{1}{2}\big)$, and all Boolean processing functions with two inputs in the 2D regular grid are the NAND rule, then broadcasting is impossible in the sense of \eqref{Eq: Deterministic Impossibility of Reconstruction}:
$$ \lim_{k \rightarrow \infty}{\left\|P_{X_k}^+ - P_{X_k}^-\right\|_{\mathsf{TV}}} = 0 \, . $$ 
\end{conjecture}

As noted in section \ref{Introduction}, we will present a detailed ``program'' for proving this conjecture in section \ref{Martingale Approach for NAND Processing Functions}. This program is inspired by the potential function technique employed in the proof of ergodicity of 1D PCA with NAND gates in \cite[Theorem 1]{HolroydMarcoviciMartin2019}, and can be construed as a more \emph{general approach} to proving impossibility of broadcasting on 2D regular grids with other processing functions or ergodicity of the corresponding 1D PCA. In this subsection, we delineate a sufficient condition for proving Conjecture \ref{Conj: NAND 2D Regular Grid} that follows from the arguments in section \ref{Martingale Approach for NAND Processing Functions}, and provide accompanying numerical evidence that this sufficient condition is actually true. 

To this end, we begin with some necessary setup, notation, and definitions that are relevant to the NAND 2D regular grid. Let $\{X^+_k : k \in \N\}$ and $\{X^-_k : k \in \N\}$ denote versions of the Markov chain $\{X_k : k \in \N\}$ initialized at $X^+_0 = 1$ and $X^-_0 = 0$, respectively. Define the coupled 2D grid variables $\{Y_{k,j} = (X_{k,j}^-,X_{k,j}^+) : k \in \N, \, j \in [k+1]\}$, which yield the Markovian coupling $\{Y_k = (Y_{k,0},\dots,Y_{k,k}) : k \in \N\}$. The coupled Markov chain $\{Y_k : k \in \N\}$ ``runs'' its ``marginal'' Markov chains $\{X^+_k : k \in \N\}$ and $\{X^-_k : k \in \N\}$ on a common underlying 2D regular grid so that along any edge BSC, either both inputs are copied with probability $1 - 2\delta$, or a shared independent $\Ber\big(\frac{1}{2}\big)$ output bit is produced with probability $2\delta$. Moreover, we assume that each $Y_{k,j}$ for $k \in \N$ and $j \in [k+1]$ takes values in the extended alphabet set $\Y \triangleq \{0,1,u\}$, where $Y_{k,j} = 0$ only if $X_{k,j}^- = X_{k,j}^+ = 0$, $Y_{k,j} = 1$ only if $X_{k,j}^- = X_{k,j}^+ = 1$, and $Y_{k,j} = u$ means that either $X_{k,j}^- \neq X_{k,j}^+$ or it is ``unknown'' whether $X_{k,j}^- = X_{k,j}^+$. Our Markovian coupling, along with \eqref{Eq: Coupled BSC 2} and \eqref{Eq: Coupled Nand} from subsection \ref{Existence of Supermartingale}, completely determines the transition kernels of $\{Y_k : k \in \N\}$, i.e. the conditional distributions of $Y_{k+1} \in \Y^{k+2}$ given $Y_{k}$ for all $k \in \N$, and this coupled Markov chain starts at $Y_0 = u$ almost surely. As we will show in subsection \ref{Existence of Supermartingale}, it suffices to analyze the coupled Markov chain $\{Y_k : k \in \N\}$ to deduce the impossibility of broadcasting. (A more detailed explanation of our Markovian coupling can also be found in subsection \ref{Existence of Supermartingale}.) Lastly, in order to present our final main result, we introduce the class of cyclic potential functions (inspired by \cite{HolroydMarcoviciMartin2019}), a partial order over these potential functions, and a pertinent linear operator on the space of potential functions in the next definition.

\begin{definition}[Cyclic Potential Functions and Related Notions]
\label{Def: Cyclic Potential Functions and Related Notions}
Given any finite set of strings $v_1,\dots,v_m \in \Y^* = \cup_{k \in \N\backslash\!\{0\}}{\Y^k}$ and any associated coefficients $\alpha_1,\dots,\alpha_m \in \R$ (with $m \in \N\backslash\!\{0\}$), we may define a corresponding \emph{cyclic potential function} $w : \Y^* \rightarrow \R$ via the formal sum:
$$ w = \sum_{j = 1}^{m}{\alpha_j \{v_j\}} \, , $$
where curly braces are used to distinguish a string $v \in \Y^*$ from its associated potential function $\{v\} : \Y^* \rightarrow \R$. In particular, for every $k \in \N\backslash\!\{0\}$ and every string $y = (y_0 \cdots y_{k-1}) \in \Y^k$ of length $k$, the cyclic potential function $w$ is evaluated as follows:
$$ w[y] \triangleq \sum_{j = 1}^{m}{\alpha_j \I\!\left\{s_j \leq k\right\} \sum_{i = 0}^{k-1}{\I\!\left\{\left(y_{(i)_k} \cdots y_{(i + s_j - 1)_k}\right) = v_j\right\}}} $$ 
where $s_j$ denotes the length of $v_j$ for all $j \in \{1,\dots,m\}$, and $(i)_k \equiv i \Mod{k}$ for every $i \in \N$. Furthermore, we say that $w$ is \emph{$u$-only} if the strings $v_1,\dots,v_m$ all contain a $u$. For any fixed $r \in \N\backslash\!\{0\}$, we may also define a partial order $\cycgeq$ over the set of all cyclic potential functions for which the lengths of the underlying strings (with non-zero coefficients) defining their formal sums are bounded by $r$. Specifically, for any pair of such cyclic potential functions $w_1:\Y^* \rightarrow \R$ and $w_2:\Y^* \rightarrow \R$, we have:
$$ w_1 \cycgeq w_2 \quad \Leftrightarrow \quad \forall y \in \bigcup_{k \geq r}{\Y^k}, \enspace w_1[y] \geq w_2[y] \, . $$
Finally, we define the \emph{conditional expectation operator} $\CE$ on the space of cyclic potential functions based on the coupled NAND 2D regular grid as follows. For any input cyclic potential function $w$ (defined by the formal sum above), $\CE$ outputs the cyclic potential function with formal sum:
$$ \CE(w) \!\triangleq\! \sum_{j = 1}^{m}{\alpha_j \!\!\!\!\!\sum_{z \in \Y^{s_j+1}}{\!\!\!\!\P\!\left((Y_{s_j+1,1},\dots,Y_{s_j+1,s_j}) \!=\! v_j \middle| Y_{s_j} \!=\! z\right)\!\! \{z\}}} $$
where the probabilities are determined by the Markovian coupling $\{Y_k : k \in \N\}$.
\end{definition}

We note that Definition \ref{Def: Cyclic Potential Functions and Related Notions} collects several smaller definitions interspersed throughout section \ref{Martingale Approach for NAND Processing Functions}, where they are each carefully explained in greater detail, and many other closely related ideas and terminology are also introduced. Using the concepts in Definition \ref{Def: Cyclic Potential Functions and Related Notions}, the ensuing main theorem presents a sufficient condition for the impossibility of broadcasting on the NAND 2D regular grid. 

\begin{theorem}[Sufficient Condition for NAND 2D Regular Grid]
\label{Thm: Sufficient Condition for NAND 2D Regular Grid}
For any noise level $\delta \in \big(0,\frac{1}{2}\big)$, suppose that there exists $r \in \N\backslash\!\{0,1\}$, and a cyclic potential function $w_{\delta} : \Y^* \rightarrow \R$ whose formal sum is constructed with strings (with non-zero coefficients) of length at most $r-1$, such that:
\begin{enumerate}
\item $\w$ is $u$-only,
\item $\w \cycgeq \CE(\w)$,
\item There exists a constant $C = C(\delta) > 0$ (which may depend on $\delta$) such that $\w \cycgeq C \{u\}$,
\end{enumerate}
where $\cycgeq$ is defined using $r$ (as in Definition \ref{Def: Cyclic Potential Functions and Related Notions}), and $C \{u\} : \Y^* \rightarrow \R$ is the cyclic potential function consisting of a single string $(u) \in \Y$ with coefficient $C$. Then, Conjecture \ref{Conj: NAND 2D Regular Grid} is true, i.e. broadcasting on the 2D regular grid where all Boolean processing functions with two inputs are the NAND rule is impossible in the sense of \eqref{Eq: Deterministic Impossibility of Reconstruction}.
\end{theorem}

Theorem \ref{Thm: Sufficient Condition for NAND 2D Regular Grid} is proved in subsection \ref{Proof of Theorem Sufficient Condition for NAND 2D Regular Grid}. Intuitively, the first two conditions in the theorem statement ensure that $\{\w(Y_k) : k \in \N\}$ is a supermartingale, and the third condition ensures that this supermartingale upper bounds the total number of uncoupled grid variables at successive levels. Then, using a martingale convergence argument along with some careful analysis of the stochastic dynamics of the coupled 2D regular grid, we can deduce that the number of uncoupled grid variables converges to zero almost surely. Akin to the proof of Theorem \ref{Thm: Deterministic And Grid}, this implies that broadcasting is impossible on the NAND 2D regular grid. As noted earlier, the details of the entire argument can be found in section \ref{Martingale Approach for NAND Processing Functions}.

While Theorem \ref{Thm: Sufficient Condition for NAND 2D Regular Grid} describes specific cyclic potential functions that can be used to prove Conjecture \ref{Conj: NAND 2D Regular Grid}, we do not rigorously prove their existence for all $\delta \in \big(0,\frac{1}{2}\big)$ in this paper. However, much of the development in section \ref{Martingale Approach for NAND Processing Functions} aims to carefully explain these desired potential functions and derive results that enable us to computationally construct them. In particular, we demonstrate in Proposition \ref{Prop: Linear Programming Criterion} that for fixed values of $\delta$, $r$, and $C$, the problem of finding an appropriate cyclic potential function $\w$ satisfying the three conditions of Theorem \ref{Thm: Sufficient Condition for NAND 2D Regular Grid} can be posed as a \emph{linear program} (LP). This connection stems from a graph theoretic characterization of $\cycgeq$ (the details of which are expounded in subsection \ref{Graph Theoretic Characterization}). 

\begin{table*}[t]
\caption{Feasible Solutions $\alpha^*(\delta)$ of the LP \eqref{Eq: Final LP Formulation} for Different $\delta \in \big(0,\frac{1}{2}\big)$ when $r = 4$}
\begin{center}
\begin{tabular}{|c|c@{\hskip -0.03in}ccccccc|}
\hline
$\delta$\Tstrut & & $0.0001$ & $0.001$ & $0.01$ & $0.02$ & $0.05$ & $0.1$ & $0.2$ \\
\hline
$\alpha^*(\delta)$  
& $\begin{array}{c} (000) \\ (001) \\ (00u) \\ (010) \\ (011) \\ (01u) \\ (0u0) \\ (0u1) \\ (0uu) \\ (100) \\ (101) \\ (10u) \\ (110) \\ (111) \\ (11u) \\ (1u0) \\ (1u1) \\ (1uu) \\ (u00) \\ (u01) \\ (u0u) \\ (u10) \\ (u11) \\ (u1u) \\ (uu0) \\ (uu1) \\ (uuu) \end{array}$ & $\left[\begin{array}{c} 
         0\\
         0\\
    0.0000\\
         0\\
         0\\
    1.9991\\
    0.0004\\
    1.9990\\
    1.9986\\
         0\\
         0\\
    0.0002\\
         0\\
         0\\
    0.9988\\
    1.0005\\
    1.9996\\
    1.9995\\
    1.0000\\
    1.0002\\
    1.0000\\
    2.0003\\
    1.0006\\
    1.9988\\
    1.0000\\
    1.9991\\
    1.9990 \end{array}\right]$ 
& $\left[\begin{array}{c} 
         0\\
         0\\
    0.0000\\
         0\\
         0\\
    1.9908\\
    0.0040\\
    1.9904\\
    1.9864\\
         0\\
         0\\
    0.0020\\
         0\\
         0\\
    0.9884\\
    1.0047\\
    1.9958\\
    1.9948\\
    1.0000\\
    1.0020\\
    1.0000\\
    2.0027\\
    1.0058\\
    1.9881\\
    1.0000\\
    1.9910\\
    1.9897 \end{array}\right]$ 
& $\left[\begin{array}{c}        			
         0\\
         0\\
    0.0171\\
         0\\
         0\\
    1.9223\\
    0.0421\\
    1.9164\\
    1.9528\\
         0\\
         0\\
    0.0358\\
         0\\
         0\\
    1.0010\\
    1.3343\\
    1.9526\\
    1.9610\\
    1.0000\\
    1.0057\\
    1.0000\\
    2.2894\\
    1.3689\\
    1.9396\\
    1.3358\\
    1.9535\\
    1.9499 \end{array} \right]$ 
& $\left[\begin{array}{c}     
         0\\
         0\\
    0.0624\\
         0\\
         0\\
    1.7777\\
    0.0783\\
    1.7509\\
    1.8429\\
         0\\
         0\\
    0.0948\\
         0\\
         0\\
    0.9906\\
    1.6963\\
    1.8252\\
    1.8482\\
    1.0000\\
    1.0000\\
    1.0000\\
    2.5746\\
    1.7854\\
    1.7970\\
    1.6983\\
    1.8211\\
    1.8119 \end{array}\right]$ 
& $\left[\begin{array}{c}
		     0\\
         0\\
    0.1151\\
         0\\
         0\\
    1.3064\\
    0.1477\\
    1.2696\\
    1.3672\\
         0\\
         0\\
    0.1516\\
         0\\
         0\\
    0.7999\\
    1.3745\\
    1.4142\\
    1.4430\\
    1.0000\\
    1.0000\\
    1.0000\\
    2.0237\\
    1.5929\\
    1.4229\\
    1.4511\\
    1.4682\\
    1.4577 \end{array}\right]$ 
& $\left[\begin{array}{c}
	       0\\
         0\\
    0.1368\\
         0\\
         0\\
    0.7631\\
    0.1599\\
    0.7283\\
    0.8458\\
         0\\
         0\\
    0.1589\\
         0\\
         0\\
    0.5322\\
    0.7787\\
    0.9025\\
    0.9489\\
    1.0000\\
    1.0000\\
    1.0000\\
    1.1221\\
    1.0445\\
    1.0000\\
    1.0000\\
    1.0000\\
    1.0000 \end{array}\right]$ 
& $\left[\begin{array}{c}
		     0\\
         0\\
    0.1402\\
         0\\
         0\\
    0.3943\\
    0.1852\\
    0.4138\\
    0.5363\\
         0\\
         0\\
    0.1428\\
         0\\
         0\\
    0.3189\\
    0.5149\\
    0.6328\\
    0.7120\\
    1.0000\\
    1.0000\\
    1.0000\\
    1.0000\\
    1.0000\\
    1.0000\\
    1.0000\\
    1.0000\\
    1.0000 \end{array}\right]$ \\
\hline
\end{tabular}
\end{center}
\label{Table: LP Solutions}
\end{table*}

Armed with this connection, we present some illustrative LP simulation results (computed using MATLAB with CVX optimization packages) in Table \ref{Table: LP Solutions} which numerically construct cyclic potential functions $w_{\delta}^* : \Y^* \rightarrow \R$ that satisfy the three conditions of Theorem \ref{Thm: Sufficient Condition for NAND 2D Regular Grid} with the fixed constants $r = 4$ and $C = C(\delta) = 1$. Specifically, we consider different representative values of $\delta \in \big(0,\frac{1}{2}\big)$ in the first row of Table \ref{Table: LP Solutions}. For any such fixed $\delta$ value, the second row of Table \ref{Table: LP Solutions} displays a corresponding vector of coefficients $\alpha^*(\delta) \in \R^{27}$ (which has been rounded to $4$ decimal places) that defines a cyclic potential function $w_{\delta}^* : \Y^* \rightarrow \R$ satisfying the three conditions of Theorem \ref{Thm: Sufficient Condition for NAND 2D Regular Grid} and consisting of a formal sum over all strings of length $3$. Indeed, for the readers' convenience, we index these vectors of coefficients $\alpha^*(\delta)$ with $\Y^3$ in Table \ref{Table: LP Solutions}, so that each $\alpha^*(\delta)$ defines $w_{\delta}^*$ via the formal sum constructed by scaling each index in $\Y^3$ with the associated value in $\alpha^*(\delta)$ and adding them up (see Definition \ref{Def: Cyclic Potential Functions and Related Notions}). For example, the first column of Table \ref{Table: LP Solutions} states that when $\delta = 0.0001$, the cyclic potential function, $w_{\delta}^* = 1.9991 \{01u\} + 0.0004 \{0u0\} + \cdots + 1.9991\{uu1\} + 1.9990\{uuu\}$, satisfies the conditions of Theorem \ref{Thm: Sufficient Condition for NAND 2D Regular Grid} with $r = 4$ and $C = 1$. Furthermore, it is straightforward to see that the coefficients corresponding to strings in $\Y^3$ with no $u$'s are always zero in Table \ref{Table: LP Solutions}, and the coefficients corresponding to the last $9$ strings in $\Y^3$ that begin with a $u$ are all lower bounded by $1$. The former observation immediately confirms that the cyclic potential functions represented in Table \ref{Table: LP Solutions} satisfy the first condition of Theorem \ref{Thm: Sufficient Condition for NAND 2D Regular Grid}, and the latter observation shows that they satisfy the third condition of Theorem \ref{Thm: Sufficient Condition for NAND 2D Regular Grid} with $C = 1$ (using the idea of ``purification'' from section \ref{Martingale Approach for NAND Processing Functions}). We refer readers to subsection \ref{Linear Programming Criteria} for further details and discussion regarding how to compute the vectors $\alpha^*(\delta)$ by solving LPs.

We close this subsection with three further remarks. Firstly, Table \ref{Table: LP Solutions} only presents a small subset of our simulation results for brevity. We have solved LPs for various other values of $\delta \in \big(0,\frac{1}{2}\big)$ and always obtained associated feasible vectors of coefficients $\alpha^*(\delta)$. We do not present simulation results for $\delta > 0.146446\dots$ (other than $\delta = 0.2$), because part 2 of Proposition \ref{Prop: Evans Schulman} implies that broadcasting is impossible in the sense of \eqref{Eq: Deterministic Impossibility of Reconstruction} in this case\textemdash see \eqref{Eq: ES Threshold} and the discussion below in subsection \ref{Related Results in the Literature}. Moreover, the impossibility of broadcasting is intuitively more surprising for smaller values of noise $\delta$, so we emphasized the LP solutions for smaller values of $\delta$ in Table \ref{Table: LP Solutions} since they are more compelling. (Note, however, that we do \emph{not} know any rigorous monotonicity result which shows that broadcasting is impossible on the NAND 2D regular grid for $\delta^{\prime} \in \big(0,\frac{1}{2}\big)$ if it is impossible for a smaller $\delta \in \big(0,\frac{1}{2}\big)$ value, i.e. $\delta < \delta^{\prime}$.)

Secondly, it is worth mentioning that Theorem \ref{Thm: Sufficient Condition for NAND 2D Regular Grid} and associated LP simulation results (as in Table \ref{Table: LP Solutions}) yield non-rigorous computer-assisted proofs of the impossibility of broadcasting on NAND 2D regular grids for individual values of $\delta \in \big(0,\frac{1}{2}\big)$. Hence, Table \ref{Table: LP Solutions} provides strong evidence that broadcasting is indeed impossible on the NAND 2D regular grid for all $\delta \in \big(0,\frac{1}{2}\big)$ in the sense of \eqref{Eq: Deterministic Impossibility of Reconstruction}. We elaborate on possible approaches to rigorize our LP-based argument at the end of subsection \ref{Linear Programming Criteria}.

Finally, we emphasize that by fully establishing the AND and XOR cases in Theorems \ref{Thm: Deterministic And Grid} and \ref{Thm: Deterministic Xor Grid}, and partially establishing the NAND case in Theorem \ref{Thm: Sufficient Condition for NAND 2D Regular Grid} and our simulations, we have made substantial progress towards proving the 2D aspect our conjecture in subsection \ref{Motivation} that broadcasting is impossible for 2D regular grids for all choices of common Boolean processing functions. Indeed, there are $16$ possible $2$-input logic gates that can serve as our processing function. It turns out that to prove our conjecture, we only need to establish the impossibility of broadcasting for four nontrivial cases out of the $16$ gates. To elaborate further, notice that the two constant Boolean processing functions that always output $0$ or $1$ engender 2D regular grids where only the vertices at the boundary carry any useful information about the source bit. However, since the boundaries of the 2D regular grid are ergodic binary-state Markov chains (since $\delta \in \big(0,\frac{1}{2}\big)$), broadcasting is clearly impossible for such constant processing functions. The four $2$-input Boolean processing functions that are actually $1$-input functions, namely, the identity maps for the first or second input and the inverters (or NOT gates) for the first or second input, beget 2D regular grids that are actually trees. Moreover, these trees have branching number $1$, and hence, the results of \cite{Evansetal2000} (outlined in section \ref{Introduction}) imply that broadcasting is impossible for such $1$-input Boolean processing functions. The six remaining symmetric Boolean processing functions are AND, NAND, OR, NOR, XOR, and XNOR. Due to the symmetry of $0$'s and $1$'s in our model (see subsection \ref{Deterministic 2D Grid Model}), we need only prove the impossibility of broadcasting for the three logic gates AND, XOR, and NAND (which are equivalent to OR, XNOR, and NOR, respectively). This leaves four asymmetric $2$-input Boolean processing functions out of the original $16$. Once again, due to the symmetry of $0$'s and $1$'s in our model, we need only consider two of these logic gates: the first gate outputs $1$ when its two inputs are equal, and outputs its first input when its two inputs are different, and the second gate outputs $1$ when its two inputs are equal, and outputs its second input when its two inputs are different. However, due to the symmetry of the edge configuration in our 2D regular grid construction (see subsection \ref{Deterministic 2D Grid Model}), it suffices to analyze only the second gate. Since this second function is not commonly viewed as a logic gate, we write down its truth table for the readers' convenience:
\begin{equation}
\label{Eq: The IMP gate}
\begin{array}{|c|c|c|}
\hline
x_1 & x_2 & x_1 \, \Rightarrow \, x_2 \\
\hline
0 & 0 & 1 \\
0 & 1 & 1 \\
1 & 0 & 0 \\
1 & 1 & 1 \\
\hline
\end{array}
\end{equation}
and recognize it as the \emph{implication relation}, denoted as IMP. Therefore, to prove the 2D aspect of our conjecture that broadcasting on 2D regular grids is impossible, we only have to analyze four nontrivial Boolean processing functions:
\begin{equation}
1) \text{ AND} \, , \quad 2) \text{ XOR} \, , \quad 3) \text{ NAND} \, , \quad 4) \text{ IMP} \, .
\end{equation}
Clearly, Theorems \ref{Thm: Deterministic And Grid} and \ref{Thm: Deterministic Xor Grid} completely tackle the first two cases, and Theorem \ref{Thm: Sufficient Condition for NAND 2D Regular Grid}, the convincing computer simulations for its sufficient condition in Table \ref{Table: LP Solutions}, and the discussion in section \ref{Martingale Approach for NAND Processing Functions} partially address the third case. The fourth case could also be partially addressed using the approach used for the third case, but we do not delve into it in this work for brevity. 

\subsection{Related Results in the Literature}
\label{Related Results in the Literature}

In this subsection, we present and discuss two complementary impossibility results from the literature regarding broadcasting on general DAGs (not just 2D regular grids). To this end, consider any infinite DAG $\G$ with a single source vertex at level $k = 0$ and $L_k \in \N\backslash\!\{0\}$ vertices at level $k \in \N$ (where $L_0 = 1$), and assume that $\G$ is in topological ordering so that all its edges are directed from lower levels to higher levels. Suppose further that for any level $k \in \N\backslash\!\{0\}$, each vertex of $\G$ at level $k$ has indegree $d \in \N\backslash\!\{0\}$ and the $d$ incoming edges originate from vertices at level $k-1$.\footnote{This implies that $\G$ is actually a multigraph since we permit multiple edges to exist between two vertices in successive levels. As explained in \cite[Section I-C]{MakurMosselPolyanskiy2020}, we can perceive the Bayesian network defined on this multigraph as a true DAG by constructing auxiliary vertices.} Much like subsection \ref{Deterministic 2D Grid Model}, we define a Bayesian network on this infinite DAG by letting each vertex be a Bernoulli random variable: For $k \in \N$ and $j \in [L_k]$, let $X_{k,j} \in \{0,1\}$ be the random variable corresponding to the $j$th vertex at level $k$, and let $X_k = (X_{k,0},\dots,X_{k,L_k - 1})$. As before, the edges of $\G$ are independent $\mathsf{BSC}(\delta)$'s with common parameter $\delta \in \big(0,\frac{1}{2}\big)$, and the vertices of $\G$ combine their inputs using Boolean processing functions. We say that broadcasting is impossible on this DAG if and only if \eqref{Eq: Deterministic Impossibility of Reconstruction} holds.  

The first impossibility result, which we proved in \cite[Proposition 3]{MakurMosselPolyanskiy2020}, illustrates that if $L_k$ is sub-logarithmic for every sufficiently large $k \in \N$, then broadcasting is impossible.

\begin{proposition}[Slow Growth of Layers {\cite[Proposition 3]{MakurMosselPolyanskiy2020}}]
\label{Prop: Slow Growth of Layers}
For any $\delta \in \big(0,\frac{1}{2}\big)$ and any $d \in \N\backslash\!\{0\}$, if $L_k \leq \log(k)/\allowbreak (d \log(1/(2\delta)))$ for all sufficiently large $k \in \N$, then for all choices of Boolean processing functions (which may vary between vertices), broadcasting is impossible on $\G$ in the sense of \eqref{Eq: Deterministic Impossibility of Reconstruction}:
$$ \lim_{k \rightarrow \infty}{\left\|P_{X_k}^+ - P_{X_k}^-\right\|_{\mathsf{TV}}} = 0 $$
where $\log(\cdot)$ denotes the natural logarithm, and $P_{X_k}^+$ and $P_{X_k}^-$ denote the conditional distributions of $X_k$ given $X_{0,0} = 1$ and $X_{0,0} = 0$, respectively.
\end{proposition} 

For the 2D regular grid models that we consider, the underlying DAGs have $L_k = k+1$. Hence, the impossibility of broadcasting in our models is not trivial, because this result does not apply to them. We further remark that an analogous impossibility result to Proposition \ref{Prop: Slow Growth of Layers} for DAGs without the bounded indegree assumption is also proved in \cite[Proposition 4]{MakurMosselPolyanskiy2020}: When each vertex at level $k \in \N\backslash\!\{0\}$ of our DAG is connected to all $L_{k-1}$ vertices at level $k-1$, for any $\delta \in \big(0,\frac{1}{2}\big)$, if $L_k \leq \sqrt{\log(k)/\!\log(1/(2\delta))}$ for all sufficiently large $k \in \N$, then broadcasting is impossible in the sense of \eqref{Eq: Deterministic Impossibility of Reconstruction} for all choices of Boolean processing functions (which may vary between vertices).

The second impossibility result, which specializes a more general result proved by Evans and Schulman in the context of von Neumann's model of fault-tolerant computation using noisy circuits \cite[Lemma 2 and p.2373]{EvansSchulman1999}, portrays an upper bound on the mutual information between $X_0$ and $X_k$ for every $k \in \N$ which decays exponentially under appropriate conditions (also see \cite[Proposition 5]{MakurMosselPolyanskiy2020}). 

\begin{proposition}[Decay of Mutual Information {\cite[Lemma 2]{EvansSchulman1999}}]
\label{Prop: Evans Schulman}
The following are true:
\begin{enumerate}
\item For the Bayesian network defined on $\G$, given any choices of Boolean processing functions (which may vary between vertices), the mutual information (in natural units) between $X_0$ and $X_k$, denoted $I(X_0;X_k)$, satisfies:
$$ I(X_0;X_k) \leq \log(2) L_k \left((1-2\delta)^2 d\right)^k $$
where $L_k d^k$ bounds the number of paths from the source $X_0$ to layer $X_k$, and $(1-2\delta)^{2k}$ represents the ``aggregate contraction'' of mutual information along each path.
\item For the Bayesian network defined on the 2D regular grid in subsection \ref{Deterministic 2D Grid Model}, given any choices of Boolean processing functions (which may vary between vertices), the mutual information between $X_0$ and $X_k$ satisfies:
$$ I(X_0;X_k) \leq \log(2) \left(2 (1-2\delta)^2 \right)^k $$
where $2^k$ denotes the number of paths from the source $X_0$ to layer $X_k$, and $(1-2\delta)^{2k}$ represents the ``aggregate contraction'' of mutual information along each path.
\end{enumerate}
\end{proposition}

Proposition \ref{Prop: Evans Schulman} is discussed in far greater detail in \cite[Section II-C]{MakurMosselPolyanskiy2020}, where several related references are also provided.\footnote{We also refer interested readers to \cite{Makur2019}, \cite{MakurZheng2020}, \cite{PolyanskiyWu2017}, and the references therein for more on the related theory of strong data processing inequalities.} Note that for the Bayesian network on $\G$, part 1 of Proposition \ref{Prop: Evans Schulman} shows that if $(1 - 2\delta)^2 d < 1$ (or equivalently, $\delta > \frac{1}{2} - \frac{1}{2\sqrt{d}}$) and $L_k = o\big(1/((1-2\delta)^2 d)^k\big)$, then for all choices of Boolean processing functions, we have:
\begin{equation}
\label{Eq: MI Vanishes}
\lim_{k \rightarrow \infty}{I(X_0;X_k)} = 0 
\end{equation}
which implies that broadcasting is impossible in the sense of \eqref{Eq: Deterministic Impossibility of Reconstruction} (due to Pinsker's inequality). Likewise, for the Bayesian network on the 2D regular grid, part 2 of Proposition \ref{Prop: Evans Schulman} shows that if the ``Evans-Schulman condition'' holds:
\begin{equation}
\label{Eq: ES Threshold}
2 (1-2\delta)^2 < 1 \quad \Leftrightarrow \quad \delta > \frac{1}{2} - \frac{1}{2\sqrt{2}} = 0.146446\dots \, ,
\end{equation}
then for all choices of Boolean processing functions, we get \eqref{Eq: MI Vanishes} and broadcasting is impossible in the sense of \eqref{Eq: Deterministic Impossibility of Reconstruction}.

\subsection{3D Regular Grid Model and Toom's NEC Rule}
\label{3D Regular Grid Model}

In view of our broader conjecture in subsection \ref{Motivation} that broadcasting should be possible in 3D regular grids, in this subsection, we elucidate a connection between a 3D regular grid with all majority processing functions and a 2D PCA with noise on the edges and boundary conditions that uses Toom's NEC rule \cite{Toom1980}. We first define the 3D regular grid model (akin to subsection \ref{Deterministic 2D Grid Model}). A 3D regular grid is an infinite DAG whose vertex set $\N^3$ is the intersection of the 3D integer lattice and the 3D non-negative orthant, and corresponding to each vertex $v \in \N^3$, we associate a Bernoulli random variable $X_v \in \{0,1\}$. Furthermore, a 3D regular grid contains the directed edges $(X_{v},X_{v + e_1})$, $(X_{v},X_{v + e_2})$, and $(X_{v},X_{v + e_3})$ for every $v \in \N^3$, which are independent BSCs with crossover probability $\delta \in \big(0,\frac{1}{2}\big)$, where $e_i$ denotes the $i$th standard basis vector of appropriate dimension (which has $1$ in the $i$th position and $0$ elsewhere). In this subsection, we set all the Boolean processing functions at the vertices to be the \emph{majority} rule. This implies that:
\begin{enumerate}
\item For any vertex $v = (v_1,v_2,v_3) \in \N^3\backslash\!\{\0\}$ on an axis of $\R^3$, i.e. there exists a unique $i \in \{1,2,3\}$ such that $v_i > 0$, we have: 
\begin{equation}
\label{Eq: Maj Grid 1}
X_v = X_{v-e_i} \oplus Z_{v,i} 
\end{equation}
where $\0$ denotes the zero vector of appropriate dimension,
\item For any vertex $v = (v_1,v_2,v_3) \in \N^3\backslash\! \{\0\}$ in a plane spanned by two axes of $\R^3$, i.e. there exist distinct $i,j \in \{1,2,3\}$ such that $v_i,v_j > 0$ and $v_k = 0$ for $k \in \{1,2,3\}\backslash\{i,j\}$, we have: 
\begin{equation}
\label{Eq: Maj Grid 2}
X_v =
\begin{cases}
X_{v-e_i} \oplus Z_{v,i} \, , & \text{with probability } \frac{1}{2} \\
X_{v-e_j} \oplus Z_{v,j} \, , & \text{with probability } \frac{1}{2} \\
\end{cases} 
\end{equation}
where one of the two noisy inputs of $X_v$ is chosen randomly and independently of everything else, 
\item For any vertex $v = (v_1,v_2,v_3) \in \N^3$ in the interior of the 3D non-negative orthant, i.e. $v_1,v_2,v_3 > 0$, we have: 
\begin{equation}
\label{Eq: Maj Grid 3}
\begin{aligned}
X_v & = \maj(X_{v-e_1} \oplus Z_{v,1}, X_{v-e_2} \oplus Z_{v,2},\\
& \qquad \qquad \quad \, X_{v-e_3} \oplus Z_{v,3}) ,
\end{aligned}
\end{equation}
\end{enumerate}
where $\{Z_{v,i} : v \in \N^3\backslash\!\{\0\}, \, i \in \{1,2,3\}\}$ are i.i.d. $\Ber(\delta)$ noise random variables that are independent of everything else, and $\maj:\{0,1\}^3 \rightarrow \{0,1\}$ is the $3$-input majority Boolean function. This defines the Bayesian network corresponding to the \textit{majority 3D regular grid}. In particular, for any $k \in \N$, let the discretized $2$-simplex $\S_k \triangleq \{v = (v_1,v_2,v_3) \in \N^3 : v_1 + v_2 + v_3 = k \}$ denote the $k$th layer of vertices in the 3D regular grid at distance $k$ from the source (so that $\N^3 = \cup_{k \in \N}{\, \S_k}$), and $X_{\S_k} \triangleq (X_v : v \in \S_k)$ denote the corresponding collection of random variables. Then, the majority 3D regular grid is completely characterized by the set of Markov transition kernels $\big\{P_{X_{\S_k}|X_{\0}} : k \in \N\big\}$ (which are defined by \eqref{Eq: Maj Grid 1}, \eqref{Eq: Maj Grid 2}, and \eqref{Eq: Maj Grid 3}). (While the definitions of these Markov kernels suffice for our purposes here, if we were to fully analyze the broadcasting question for majority 3D regular grids, then we would also assume that the source bit is $X_{\0} \sim \Ber\big(\frac{1}{2}\big)$.)

We next recall a version of Toom's 2D PCA which has noise on the edges rather than the vertices, cf. \cite{Toom1980}. Consider the 2D integer lattice $\Z^2$ of sites, the binary state space $\{0,1\}$, and the \textit{configuration space} $\{0,1\}^{\Z^2}$ of functions $\xi : \Z^2 \rightarrow \{0,1\}$ (which map sites to their bit values). Moreover, fix the PCA's transition function to be $\maj$, and let its \textit{interaction neighborhood} be $\Nbhd = \{-e_1,\0,-e_2\} \subseteq \Z^2$, which together define Toom's NEC rule \cite{Toom1980}.\footnote{In this case, the transition function and interaction neighborhood really yield a South-West-Center (SWC) rule rather than Toom's NEC rule, but we will still use the historically inspired nomenclature.} With these building blocks in place, we now describe the dynamics of Toom's 2D PCA with edge noise. Suppose at time $k \in \N$, the automaton has configuration $\xi_k \in \{0,1\}^{\Z^2}$. Then, its configuration $\xi_{k+1} \in \{0,1\}^{\Z^2}$ at time $k + 1$ is determined as follows: Each site $x \in \Z^2$ simultaneously receives three noisy input bits $(\xi_k(x + v) \oplus Z_{k,x,v} : v \in \Nbhd)$ from its interaction neighborhood, and then computes its value at time $k+1$ using Toom's NEC rule, i.e. $\xi_{k+1}(x) = \maj(\xi_k(x + v) \oplus Z_{k,x,v} : v \in \Nbhd)$, where $\{Z_{k,x,v} : k \in \N, \, x \in \Z^2, \, v \in \Nbhd\}$ are i.i.d. $\Ber(\delta)$ noise random variables that are independent of everything else. This defines the transition kernel of the Markov process corresponding to Toom's 2D PCA with edge noise. (Note that if we seek to analyze the ergodicity of this PCA, then we would have to additionally fix different initial configurations $\xi_0 \in \{0,1\}^{\Z^2}$ at time $k = 0$.) 

By modifying Toom's 2D PCA with edge noise outlined above, we now introduce \textit{Toom's 2D PCA with boundary conditions}. For any $k \in \N$, consider the projected discretized $2$-simplices $\hat{\S}_k \triangleq \{x = (x_1,x_2) \in \N^2 : x_1 + x_2 \leq k\}$. Using the notation introduced above, define the jointly distributed binary random variables $\{\xi_{k}(x) : x \in \hat{\S}_k, \, k \in \N\}$, and let $\xi_k \triangleq (\xi_k(x) : x \in \hat{\S}_k)$ be the values of the bits of this automaton at time $k \in \N$. The dynamics of Toom's 2D PCA with boundary conditions follows the dynamics of Toom's 2D PCA explained above (mutatis mutandis). Specifically, conditioned on $\xi_k$ at time $k \in \N$, the probability distribution of $\xi_{k+1}$ at time $k+1$ is characterized by:
\begin{enumerate}
\item For the site $\0 \in \hat{\S}_{k+1}$, we have: 
\begin{equation}
\label{Eq: Toom PCA 1}
\xi_{k+1}(\0) = \xi_{k}(\0) \oplus Z_{k,\0,\0} \, ,
\end{equation}
and for the sites $(k+1) e_i \in \hat{\S}_{k+1}$ with $i \in \{1,2\}$, we have:
\begin{equation}
\label{Eq: Toom PCA 1.5}
\xi_{k+1}((k+1) e_i) = \xi_{k}(k e_i) \oplus Z_{k, (k+1) e_i,-e_i} \, ,
\end{equation}
\item For any site $x = (x_1,x_2) \in \hat{\S}_{k+1}\backslash\{\0,(k+1)e_1,(k+1)e_2\}$ such that $x_i > 0$ and $x_j = 0$ for $i,j \in \{1,2\}$, we have: 
\begin{equation}
\label{Eq: Toom PCA 2}
\xi_{k+1}(x) =
\begin{cases}
\xi_{k}(x) \oplus Z_{k,x,\0} , & \!\!\! \text{with probability } \frac{1}{2} \\
\xi_{k}(x - e_i) \oplus Z_{k,x,-e_i} , & \!\!\! \text{with probability } \frac{1}{2} \\
\end{cases} 
\end{equation}
where one of the two noisy inputs of $\xi_{k+1}(x)$ is chosen randomly and independently of everything else, and for any site $x = (x_1,x_2) \in \hat{\S}_{k+1}\backslash\{(k+1)e_1,(k+1)e_2\}$ such that $x_1 + x_2 = k+1$, we have:
\begin{equation}
\label{Eq: Toom PCA 2.5}
\xi_{k+1}(x) =
\begin{cases}
\xi_{k}(x - e_1) \oplus Z_{k,x,-e_1} , & \!\!\!\! \text{with probability } \frac{1}{2} \\
\xi_{k}(x - e_2) \oplus Z_{k,x,-e_2} , & \!\!\!\! \text{with probability } \frac{1}{2} \\
\end{cases} 
\end{equation}
where, once again, one of the two noisy inputs of $\xi_{k+1}(x)$ is chosen randomly and independently of everything else,
\item For any other site $x = (x_1,x_2) \in \hat{\S}_{k+1}$ such that $x_1,x_2 > 0$, we have: 
\begin{equation}
\label{Eq: Toom PCA 3}
\begin{aligned}
\xi_{k+1}(x) & = \maj(\xi_{k}(x - e_1) \oplus Z_{k,x,-e_1}, \\
& \enspace \xi_{k}(x) \oplus Z_{k,x,\0}, \xi_{k}(x - e_2) \oplus Z_{k,x,-e_2}) ,
\end{aligned}
\end{equation}
\end{enumerate}
where, as before, $\{Z_{k,x,v} : k \in \N, \, x \in \hat{\S}_{k+1}, \, v \in \Nbhd\}$ are i.i.d. $\Ber(\delta)$ noise random variables that are independent of everything else. This completely specifies the transition kernels $\big\{P_{\xi_k|\xi_0(\0)} : k \in \N\big\}$ of the Markov process corresponding to Toom's 2D PCA with boundary conditions (which are defined by \eqref{Eq: Toom PCA 1}, \eqref{Eq: Toom PCA 1.5}, \eqref{Eq: Toom PCA 2}, \eqref{Eq: Toom PCA 2.5}, and \eqref{Eq: Toom PCA 3}). (To define a valid joint probability distribution of $\{\xi_{k}(x) : x \in \hat{\S}_k, \, k \in \N\}$, we can additionally impose the initial condition $\xi_0(\0) \sim \Ber\big(\frac{1}{2}\big)$ akin to the uniform source distribution for the majority 3D regular grid.)

It turns out that the majority 3D regular grid and Toom's 2D PCA with boundary conditions are statistically equivalent, i.e. their BSC noise random variables can be coupled so that the Markov processes $\{X_{\S_k} : k \in \N\}$ and $\{\xi_k : k \in \N\}$ are equal almost surely. This equivalence is presented in the ensuing proposition. 

\begin{proposition}[Majority 3D Regular Grid Equivalence]
\label{Prop: Majority 3D Regular Grid}
If $X_{\0} = \xi_0(\0)$ almost surely with any initial source distribution, then we can couple the majority 3D regular grid and Toom's 2D PCA with boundary conditions so that almost surely:
$$ \forall v = (v_1,v_2,v_3) \in \N^3, \enspace X_v = \xi_{v_1 + v_2 + v_3}((v_1,v_2)) \, , $$
or equivalently:
$$ \forall k \in \N, \, \forall x = (x_1,x_2) \in \hat{\S}_k, \enspace \xi_{k}(x) = X_{(x_1,x_2,k - x_1 - x_2)} \, . $$
\end{proposition}

Proposition \ref{Prop: Majority 3D Regular Grid} is proved in appendix \ref{Proof of Proposition Majority 3D Regular Grid} via a \emph{projection} argument. The result implies that broadcasting is impossible in the majority 3D regular grid if and only if reconstruction of $\xi_0(\0)$ is impossible in Toom's 2D PCA with boundary conditions. Since the standard 2D PCA (with noise on the vertices) that uses Toom's NEC rule is non-ergodic \cite{Toom1980}, we believe that reconstruction should be possible in our variant of Toom's 2D PCA with boundary conditions. Toom's proof of non-ergodicity in \cite{Toom1980} is quite sophisticated, and a simpler combinatorial version of it has been proposed in \cite{Gacs1995}. We feel that it is an interesting open problem to establish the feasibility of broadcasting in the 3D regular grid with majority processing functions by possibly modifying the simple version of Toom's proof in \cite{Gacs1995}.

\section{Percolation Analysis of 2D Regular Grid with AND Processing Functions}
\label{Analysis of Deterministic And Grid}

In this section, we prove Theorem \ref{Thm: Deterministic And Grid}. Recall that we are given a 2D regular grid where all Boolean processing functions with two inputs are the AND rule, and all Boolean processing functions with one input are the identity rule, i.e. $f_{2}(x_1,x_2) = x_1 \wedge x_2$ and $f_{1}(x) = x$, where $\wedge$ denotes the logical AND operation. 

As in our proof of \cite[Theorem 1]{MakurMosselPolyanskiy2020}, we begin by constructing a useful ``monotone Markovian coupling'' (see \cite[Chapter 5]{LevinPeresWilmer2009} for basic definitions of Markovian couplings). Let $\{X^+_k : k \in \N\}$ and $\{X^-_k : k \in \N\}$ denote versions of the Markov chain $\{X_k : k \in \N\}$ (i.e. with the same transition kernels) initialized at $X^+_0 = 1$ and $X^-_0 = 0$, respectively. Note that the marginal distributions of $X_k^+$ and $X_k^-$ are $P^+_{X_k}$ and $P^-_{X_k}$, respectively. Furthermore, define the coupled 2D grid variables $\{Y_{k,j} = (X_{k,j}^-,X_{k,j}^+) : k \in \N, \, j \in [k+1]\}$, so that our Markovian coupling of the Markov chains $\{X^+_k : k \in \N\}$ and $\{X^-_k : k \in \N\}$ is the Markov chain $\{Y_k = (Y_{k,0},\dots,Y_{k,k}) : k \in \N\}$. We will couple $\{X^+_k : k \in \N\}$ and $\{X^-_k : k \in \N\}$ to ``run'' on a common underlying 2D regular grid with shared edge BSCs. 

Recall that each edge $\mathsf{BSC}(\delta)$ either copies its input bit with probability $1 - 2\delta$, or generates an independent $\Ber\big(\frac{1}{2}\big)$ output bit with probability $2\delta$.\footnote{This idea stems from the study of Fortuin-Kasteleyn random cluster representations of Ising models \cite{Grimmett1997}.} This follows from appropriately interpreting the following decomposition of the BSC transition matrix:
\begin{equation}
\label{Eq: BSC Decomposition}
\def\arraystretch{1.1} 
\underbrace{\left[\begin{array}{cc} 1-\delta & \delta \\ \delta & 1-\delta \end{array}\right]}_{\text{BSC matrix}} = (1-2\delta)\underbrace{\left[\begin{array}{cc} 1 & 0 \\ 0 & 1 \end{array}\right]}_{\text{copy matrix}} + (2\delta) \underbrace{\left[\begin{array}{cc} \frac{1}{2} & \frac{1}{2} \\ \frac{1}{2} & \frac{1}{2} \end{array}\right]}_{\text{random bit}} .
\def\arraystretch{1} 
\end{equation}
Since the underlying 2D regular grid is fixed, we couple $\{X^+_k : k \in \N\}$ and $\{X^-_k : k \in \N\}$ so that along any edge BSC of the grid, say $(X_{k,j},X_{k+1,j})$, $X_{k,j}^+$ and $X_{k,j}^-$ are either both copied with probability $1 - 2\delta$, or a shared independent $\Ber\big(\frac{1}{2}\big)$ bit is produced with probability $2\delta$ that is used by both $X_{k+1,j}^+$ and $X_{k+1,j}^-$. The Markovian coupling $\{Y_k : k \in \N\}$ exhibits the following properties:
\begin{enumerate}
\item The ``marginal'' Markov chains are $\{X^+_k : k \in \N\}$ and $\{X^-_k : k \in \N\}$.
\item For every $k \in \N$, $X_{k+1}^+$ is conditionally independent of $X_{k}^-$ given $X_k^{+}$, and $X_{k+1}^-$ is conditionally independent of $X_{k}^+$ given $X_k^{-}$. 
\item For every $k \in \N$ and every $j \in [k+1]$, $X_{k,j}^+ \geq X_{k,j}^-$ almost surely.
\end{enumerate} 
Here, the third (monotonicity) property of our coupling holds because $1 = X_{0,0}^+ \geq X_{0,0}^- = 0$ is true by assumption, each edge BSC preserves monotonicity, and AND processing functions are symmetric and monotone non-decreasing. In this section, probabilities of events that depend on the coupled 2D grid variables $\{Y_{k,j} : k \in \N, \, j \in [k+1]\}$ are defined with respect to this Markovian coupling. 

Since the marginal Markov chains $\{X^+_k : k \in \N\}$ and $\{X^-_k : k \in \N\}$ run on the same 2D regular grid with common BSCs, we keep track of the Markov chain $\{Y_k : k \in \N\}$ in a single coupled 2D regular grid. This 2D regular grid has the same underlying graph as the 2D regular grid described in subsection \ref{Deterministic 2D Grid Model}. Its vertices are the coupled 2D grid variables $\{Y_{k,j} = (X_{k,j}^-,X_{k,j}^+) : k \in \N, \, j \in [k+1]\}$, and we relabel the alphabet of these variables for simplicity. So, each $Y_{k,j}  = (X_{k,j}^-,X_{k,j}^+) \in \Y$ with:
\begin{equation}
\Y \triangleq \left\{0_{\mathsf{c}},1_{\mathsf{u}},1_{\mathsf{c}}\right\}
\end{equation}
where $0_{\mathsf{c}} = (0,0)$, $1_{\mathsf{u}} = (0,1)$, and $1_{\mathsf{c}} = (1,1)$. (Note that we do not require a letter $0_{\mathsf{u}} = (1,0)$ in this alphabet due to the monotonicity in the coupling.) Furthermore, each edge of the coupled 2D regular grid is a channel (conditional distribution) $W$ between the alphabets $\Y$ and $\Y$ that captures the action of a shared $\mathsf{BSC}(\delta)$\textemdash we describe $W$ using the following row stochastic matrix:
\begin{equation}
\label{Eq: Coupled BSC} 
W = \bbordermatrix{
		& 0_{\mathsf{c}} & 1_{\mathsf{u}} & 1_{\mathsf{c}} \cr
      0_{\mathsf{c}} & 1-\delta & 0 & \delta \cr
      1_{\mathsf{u}} & \delta & 1-2\delta & \delta \cr
      1_{\mathsf{c}} & \delta & 0 & 1-\delta \cr }
\end{equation}
where the $(i,j)$th entry gives the probability of output $j$ given input $i$. It is straightforward to verify that $W$ describes the aforementioned Markovian coupling. Finally, the AND rule can be equivalently described on the alphabet $\Y$ as:
\begin{equation}
\label{Eq: Coupled And} 
\begin{array}{|c|c|c|}
\hline
y_1 & y_2 & y_1 \wedge y_2 \\
\hline
0_{\mathsf{c}} & \star & 0_{\mathsf{c}} \\
1_{\mathsf{u}} & 1_{\mathsf{u}} & 1_{\mathsf{u}} \\
1_{\mathsf{u}} & 1_{\mathsf{c}} & 1_{\mathsf{u}} \\
1_{\mathsf{c}} & 1_{\mathsf{c}} & 1_{\mathsf{c}} \\
\hline
\end{array}
\end{equation}
where $\star$ denotes any letter in $\Y$, and the symmetry of the AND rule covers all other possible input combinations. This coupled 2D regular grid model completely characterizes the Markov chain $\{Y_k : k \in \N\}$, which starts at $Y_0 = 1_{\mathsf{u}}$ almost surely. We next prove Theorem \ref{Thm: Deterministic And Grid} by further analyzing this model.  

\renewcommand{\proofname}{Proof of Theorem \ref{Thm: Deterministic And Grid}}

\begin{proof}
We first bound the TV distance between $P_{X_k}^+$ and $P_{X_k}^-$ using Dobrushin's maximal coupling characterization of TV distance, cf. \cite[Chapter 4.2]{LevinPeresWilmer2009}:
$$ \left\|P_{X_k}^+ - P_{X_k}^-\right\|_{\mathsf{TV}} \leq \P\!\left(X_k^+ \neq X_k^- \right) = 1 - \P\!\left(X_k^+ = X_k^- \right) . $$
The events $\{X_k^+ = X_k^-\}$ are non-decreasing in $k$, i.e. $\{X_k^+ = X_k^-\} \subseteq \{X_{k+1}^+ = X_{k+1}^-\}$ for all $k \in \N$. Indeed, suppose for any $k \in \N$, the event $\{X_k^+ = X_k^-\}$ occurs. Since we have:
\begin{align*}
\left\{X_k^+ = X_k^-\right\} & = \left\{Y_k \in \{0_{\mathsf{c}},1_{\mathsf{c}}\}^{k+1}\right\} \\
& = \left\{\parbox[]{14em}{there are no $1_{\mathsf{u}}$'s in level $k$ of the coupled 2D regular grid}\right\} ,
\end{align*}
the channel \eqref{Eq: Coupled BSC} and the rule \eqref{Eq: Coupled And} imply that there are no $1_{\mathsf{u}}$'s in level $k+1$. Hence, the event $\{X_{k+1}^+ = X_{k+1}^-\}$ occurs as well. Letting $k \rightarrow \infty$, we can use the continuity of $\P$ with the events $\{X_k^+ = X_k^-\}$ to get:
\begin{align*}
\lim_{k \rightarrow\infty}{\left\|P_{X_k}^+ - P_{X_k}^-\right\|_{\mathsf{TV}}} & \leq 1 - \lim_{k \rightarrow \infty}{\P\!\left(X_k^+ = X_k^- \right)} \\
& = 1 - \P\!\left(A\right) 
\end{align*}
where we define:
$$ A \triangleq \left\{\parbox[]{20em}{$\exists k \in \N$, there are no $1_{\mathsf{u}}$'s in level $k$ of the coupled 2D regular grid}\right\} . $$
Therefore, it suffices to prove that $\P(A) = 1$.

To prove this, we recall a well-known result from \cite[Section 3]{Durrett1984} on \textit{oriented bond percolation} in 2D lattices. Given the underlying DAG of our 2D regular grid from subsection \ref{Deterministic 2D Grid Model}, suppose we independently keep each edge ``open'' with some probability $p \in [0,1]$, and delete it (``closed'') with probability $1-p$. Define the event:
$$ \Omega_{\infty} \triangleq \left\{\text{there is an infinite open path starting at the root}\right\} $$
and the quantities:
\begin{align*}
R_k & \triangleq \sup\!\left\{j \in [k+1] : \parbox[]{11em}{there is an open path from the root to the vertex $(k,j)$}\right\} \\
L_k & \triangleq \inf\!\left\{j \in [k+1] : \parbox[]{11em}{there is an open path from the root to the vertex $(k,j)$}\right\}
\end{align*}
which are the rightmost and leftmost vertices at level $k \in \N$, respectively, that are connected to the root. (Here, we refer to the vertex $X_{k,j}$ using $(k,j)$ as we do not associate a random variable to it.) It is proved in \cite[Section 3]{Durrett1984} that the occurrence of $\Omega_{\infty}$ experiences a phase transition phenomenon as the open probability parameter $p$ varies from $0$ to $1$.

\begin{lemma}[Oriented Bond Percolation {\cite[Section 3]{Durrett1984}}]
\label{Lemma: Oriented Bond Percolation}
For the aforementioned bond percolation process on the 2D regular grid, there exists a critical threshold $\delta_{\mathsf{perc}} \in \big(\frac{1}{2},1\big)$ around which we observe the following phase transition phenomenon:
\begin{enumerate}
\item If $p > \delta_{\mathsf{perc}}$, then $\P_p(\Omega_{\infty}) > 0$ and:
\begin{equation}
\label{Eq: Rightmost and Leftmost Open Paths}
\begin{aligned}
\P_p&\Bigg(\lim_{k \rightarrow \infty}{\frac{R_k}{k}} = \frac{1 + \alpha(p)}{2} \\
& \quad \enspace \text{and } \lim_{k \rightarrow \infty}{\frac{L_k}{k}} = \frac{1 - \alpha(p)}{2} \, \Bigg| \, \Omega_{\infty}\Bigg) = 1 
\end{aligned}
\end{equation}
for some constant $\alpha(p) > 0$, where $\alpha(p)$ is defined in \cite[Section 3, Equation (6)]{Durrett1984}, and $\P_p$ is the probability measure defined by the bond percolation process. 
\item If $p < \delta_{\mathsf{perc}}$, then $\P_p(\Omega_{\infty}) = 0$.
\end{enumerate}
\end{lemma}

We will use Lemma \ref{Lemma: Oriented Bond Percolation} to prove $\P(A) = 1$ by considering two cases.

\textbf{Case 1:} Suppose $1 - 2\delta < \delta_{\mathsf{perc}}$ (i.e. $\delta > (1-\delta_{\mathsf{perc}})/2$) in our coupled 2D grid. The root of the coupled 2D regular grid is $Y_{0,0} = 1_{\mathsf{u}}$ almost surely, and we consider an oriented bond percolation process (as described above) with $p = 1 - 2\delta$. In particular, we say that each edge of the grid is open if and only if the corresponding BSC copies its input (with probability $1 - 2\delta$). In this context, $\Omega_{\infty}^c$ is the event that there exists $k \in \N$ such that none of the vertices at level $k$ are connected to the root via a sequence of BSCs that are copies. Suppose the event $\Omega_{\infty}^c$ occurs. Since \eqref{Eq: Coupled BSC} and \eqref{Eq: Coupled And} portray that a $1_{\mathsf{u}}$ moves from level $k$ to level $k+1$ only if one of its outgoing edges is open (and the corresponding BSC is a copy), there exists $k \in \N$ such that none of the vertices at level $k$ are $1_{\mathsf{u}}$'s. This proves that $\Omega_{\infty}^c \subseteq A$. Therefore, using part 2 of Lemma \ref{Lemma: Oriented Bond Percolation}, we get $\P(A) = 1$.

\textbf{Case 2:} Suppose $1 - \delta > \delta_{\mathsf{perc}}$ (i.e. $\delta < 1-\delta_{\mathsf{perc}}$) in our coupled 2D grid. Consider an oriented bond percolation process (as described earlier) with $p = 1 - \delta$ that runs on the 2D regular grid, where an edge is open if and only if the corresponding BSC is either copying or generating a $0$ as the new bit (i.e. this BSC takes a $0_{\mathsf{c}}$ to a $0_{\mathsf{c}}$, which happens with probability $1-\delta$ as shown in \eqref{Eq: Coupled BSC}). Let $B_k$ for $k \in \N\backslash\!\{0\}$ be the event that the BSC from $Y_{k-1,0}$ to $Y_{k,0}$ generates a new bit which equals $0$. Then, $\P(B_k) = \delta$ and $\{B_k:k \in \N\backslash\!\{0\}\}$ are mutually independent. So, the second Borel-Cantelli lemma tells us that infinitely many of the events $\{B_k:k \in \N\backslash\!\{0\}\}$ occur almost surely. Furthermore, $B_k \subseteq \{Y_{k,0} = 0_{\mathsf{c}}\}$ for every $k \in \N\backslash\!\{0\}$. 

We next define the following sequence of random variables for all $i \in \N\backslash\!\{0\}$: 
\begin{align*}
L_i & \triangleq \min\!\left\{k \geq T_{i-1} + 1 : B_k \text{ occurs}\right\} \\
T_i & \triangleq 1 + \max\!\left\{k \geq L_i : \parbox[]{13em}{$\exists j \in [k+1]$, $Y_{k,j}$ is connected to $Y_{L_i,0}$ by an open path}\right\}
\end{align*}
where we set $T_0 \triangleq 0$. Note that when $T_{i-1} = \infty$, we let $L_i = \infty$ almost surely. Furthermore, when $T_{i-1} < \infty$, $L_i < \infty$ almost surely, because infinitely many of the events $\{B_k:k \in \N\backslash\!\{0\}\}$ occur almost surely. We also note that when $L_{i} < \infty$, the set:
$$ \left\{k \geq L_i : \parbox[]{16em}{$\exists j \in [k+1]$, $Y_{k,j}$ is connected to $Y_{L_i,0}$ by an open path}\right\} $$
is non-empty since $Y_{L_i,0}$ is always connected to itself, and $T_i - L_i - 1$ denotes the length of the longest open path connected to $Y_{L_i,0}$ (which could be infinity). Lastly, when $L_{i} = \infty$, we let $T_i = \infty$ almost surely. 

Let $\F_k$ for every $k \in \N$ be the $\sigma$-algebra generated by the random variables $(Y_0,\dots,Y_k)$ and all the BSCs before level $k$ (where we include all events determining whether these BSCs are copies, and all events determining the independent bits they produce). Then, $\{\F_k : k \in \N\}$ is a \emph{filtration}. It is straightforward to verify that $L_i$ and $T_i$ are \textit{stopping times} with respect to $\{\F_k : k \in \N\}$ for all $i \in \N\backslash\!\{0\}$. We can show this inductively. $T_0 = 0$ is trivially a stopping time, and if $T_{i-1}$ is a stopping time, then $L_i$ is clearly a stopping time. So, it suffices to prove that $T_i$ is a stopping time given $L_i$ is a stopping time. For any finite $m \in \N\backslash\!\{0\}$, $\{T_i = m\}$ is the event that $L_i \leq m-1$ and the length of the longest open path connected to $Y_{L_i,0}$ is $m - 1 - L_i$. This event is contained in $\F_m$ because the event $\{L_i \leq m-1\}$ is contained in $\F_{m-1} \subseteq \F_m$ (since $L_i$ is a stopping time), and the length of the longest open path can be determined from $\F_m$ (rather than $\F_{m-1}$). Hence, $T_i$ is indeed a stopping time when $L_i$ is a stopping time.  

Now observe that:
\begin{align}
& \P(\exists k \in \N\backslash\!\{0\}, \, T_k = \infty) \nonumber \\
& = \P(T_1 = \infty) \nonumber \\
& \quad \, + \sum_{m = 2}^{\infty}{\P(\exists k \in \N\backslash\!\{0,1\}, \, T_k = \infty | T_1 = m) \P(T_1 = m)} \nonumber \\
& = \P(T_1 = \infty)  \nonumber \\
& \quad \, + \sum_{m = 2}^{\infty}{\P(\exists k \in \N\backslash\!\{0\}, \, T_k + m = \infty) \P(T_1 = m)} \nonumber \\
& = \P(T_1 = \infty) + (1 - \P(T_1 = \infty)) \P(\exists k \in \N\backslash\!\{0\}, \, T_k= \infty)
\label{Eq: Preliminary Stopping Time Step} 
\end{align}
where the first equality uses the law of total probability, the third equality follows from straightforward calculations, and the second equality follows from the fact that for all $m \in \N\backslash\!\{0,1\}$:
\begin{align*}
& \P(\exists k \in \N\backslash\!\{0,1\}, \, T_k = \infty | T_1 = m) \\
& \qquad \qquad \qquad = \P(\exists k \in \N\backslash\!\{0\}, \, T_k + m = \infty) \, . 
\end{align*}
This relation holds because the random variables $\{(L_i,T_i) : i \in \N\backslash\!\{0,1\}\}$ given $T_1 = m$ have the same distribution as the random variables $\{(L_{i-1} + m,T_{i-1} + m) : i \in \N\backslash\!\{0,1\}\}$. In particular, the conditional distribution of $L_i$ given $T_1 = m$ corresponds to the distribution of $L_{i-1} + m$, and the conditional distribution of $T_i$ given $T_1 = m$ corresponds to the distribution of $T_{i-1} + m$. These distributional equivalences implicitly use the fact that $\{T_i : i \in \N\}$ are stopping times. Indeed, the conditioning on $\{T_1 = m\}$ in these equivalences can be removed because the event $\{T_1 = m\}$ is in $\F_m$ since $T_1$ is a stopping time, and $\{T_1 = m\}$ is therefore independent of the events $\{B_k : k > m\}$ and the events that determine when the BSCs below level $m$ are open. 

Next, rearranging \eqref{Eq: Preliminary Stopping Time Step}, we get:
$$ \P(\exists k \in \N\backslash\!\{0\}, \, T_k = \infty) \P(T_1 = \infty) = \P(T_1 = \infty) \, . $$
Since $\P(T_1 = \infty) = \P(\Omega_{\infty}) > 0$ by part 1 of Lemma \ref{Lemma: Oriented Bond Percolation}, we have:
\begin{equation}
\label{Eq: Infinite Path Existence}
\P(\exists k \in \N\backslash\!\{0\}, \, T_k = \infty) = 1 \, . 
\end{equation}
For every $k \in \N\backslash\!\{0\}$, define the events:
\begin{align*}
\Omega_k^{\mathsf{left}} & \triangleq \left\{\parbox[]{18em}{there exists an infinite open path starting at the vertex $Y_{k,0}$}\right\} , \\
\Omega_k^{\mathsf{right}} & \triangleq \left\{\parbox[]{18em}{there exists an infinite open path starting at the vertex $Y_{k,k}$}\right\} .
\end{align*}
If the event $\{\exists k \in \N\backslash\!\{0\}, \, T_k = \infty\}$ occurs, we can choose the smallest $m \in \N\backslash\!\{0\}$ such that $T_m = \infty$, and for this $m$, there is an infinite open path starting at $Y_{L_m,0} = 0_{\mathsf{c}}$ (where $Y_{L_m,0} = 0_{\mathsf{c}}$ because $B_{L_m}$ occurs). Hence, using \eqref{Eq: Infinite Path Existence}, we have:
$$ \P\!\left(\exists k \in \N, \, \{Y_{k,0} = 0_{\mathsf{c}}\} \cap \Omega_k^{\mathsf{left}}\right) = 1 \, . $$
Likewise, we can also prove that:
$$ \P\!\left(\exists k \in \N, \, \{Y_{k,k} = 0_{\mathsf{c}}\} \cap \Omega_k^{\mathsf{right}}\right) = 1 $$
which implies that:
\begin{equation}
\label{Eq: Paths Meet}
\P\!\left(\exists k,m \in \N, \, \{Y_{k,0} = Y_{m,m} = 0_{\mathsf{c}}\} \cap \Omega_k^{\mathsf{left}} \cap \Omega_m^{\mathsf{right}} \right) = 1 \, .
\end{equation}
 
To finish the proof, consider $k,m \in \N$ such that $Y_{k,0} = Y_{m,m} = 0_{\mathsf{c}}$, and suppose $\Omega_k^{\mathsf{left}}$ and $\Omega_m^{\mathsf{right}}$ both happen. For every $n > \max\{k,m\}$, define the quantities:
\begin{align*}
R_n^{\mathsf{left}} & \triangleq \sup\!\left\{j \in [n+1] : \parbox[]{11em}{there is an open path from $Y_{k,0}$ to $Y_{n,j}$}\right\} \\
L_n^{\mathsf{right}} & \triangleq \inf\!\left\{j \in [n+1] : \parbox[]{11em}{there is an open path from $Y_{m,m}$ to $Y_{n,j}$}\right\}
\end{align*}
which are the rightmost and leftmost vertices at level $n$ that are connected to $Y_{k,0}$ and $Y_{m,m}$, respectively, by open paths. Using \eqref{Eq: Rightmost and Leftmost Open Paths} from part 1 of Lemma \ref{Lemma: Oriented Bond Percolation}, we know that almost surely:
\begin{align*}
\lim_{n \rightarrow \infty}{\frac{R_n^{\mathsf{left}}}{n}} & = \lim_{n \rightarrow \infty}{\frac{R_n^{\mathsf{left}}}{n-k}} = \frac{1 + \alpha(1-\delta)}{2} \, , \\
\lim_{n \rightarrow \infty}{\frac{L_n^{\mathsf{right}}}{n}} & = \lim_{n \rightarrow \infty}{\frac{L_n^{\mathsf{right}} - m}{n-m}} = \frac{1 - \alpha(1-\delta)}{2} \, . 
\end{align*}
This implies that almost surely:
$$ \lim_{n \rightarrow \infty}{\frac{R_n^{\mathsf{left}} - L_n^{\mathsf{right}}}{n}} = \alpha(1-\delta) > 0 $$
which means that for some sufficiently large level $n > \max\{k,m\}$, the rightmost open path from $Y_{k,0}$ meets the leftmost open path from $Y_{m,m}$:
$$ \left|R_n^{\mathsf{left}} - L_n^{\mathsf{right}}\right| \leq 1 \, . $$
By construction, all the vertices in these two open paths are equal to $0_{\mathsf{c}}$. Furthermore, since \eqref{Eq: Coupled BSC} and \eqref{Eq: Coupled And} demonstrate that AND gates and BSCs output $0_{\mathsf{c}}$'s or $1_{\mathsf{c}}$'s when their inputs are $0_{\mathsf{c}}$'s or $1_{\mathsf{c}}$'s, it is straightforward to inductively establish that all vertices at level $n$ that are either to left of $R_n^{\mathsf{left}}$ or to the right of $L_n^{\mathsf{right}}$ take values in $\{0_{\mathsf{c}},1_{\mathsf{c}}\}$. This shows that every vertex at level $n$ must be equal to $0_{\mathsf{c}}$ or $1_{\mathsf{c}}$ because the two aforementioned open paths meet. Hence, there exists a level $n \in \N$ with no $1_{\mathsf{u}}$'s, i.e. the event $A$ occurs. Therefore, we get $\P(A) = 1$ using \eqref{Eq: Paths Meet}.

Combining the two cases completes the proof as $\P(A) = 1$ for any $\delta \in \big(0,\frac{1}{2}\big)$.
\end{proof}

\renewcommand{\proofname}{Proof}

We remark that this proof can be perceived as using the technique presented in \cite[Theorem 5.2]{LevinPeresWilmer2009}. Indeed, let $T \triangleq \inf\{k \in \N : X_k^+ = X_k^-\}$ be a stopping time (with respect to the filtration $\{\F_k : k \in \N\}$ defined earlier) denoting the first time that the marginal Markov chains $\{X_k^+ : k \in \N\}$ and $\{X_k^- : k \in \N\}$ meet. (Note that $\{T = \infty\}$ corresponds to the event that these chains never meet.) Since the events $\{X_k^+ = X_k^-\}$ for $k \in \N$ form a non-decreasing sequence of sets, $\{T > k\} = \{X_k^+ \neq X_k^-\}$. We can use this relation to obtain the following bound on the TV distance between $P_{X_k}^+$ and $P_{X_k}^-$:
\begin{align}
\left\|P_{X_k}^+ - P_{X_k}^-\right\|_{\mathsf{TV}} & \leq \P\!\left(X_k^+ \neq X_k^- \right) \nonumber \\
& = \P\!\left(T > k\right) \nonumber \\
& = 1 - \P\!\left(T \leq k\right) 
\label{Eq: Mixing Time Technique}
\end{align}
where letting $k \rightarrow \infty$ and using the continuity of $\P$ produces:
\begin{align}
\lim_{k \rightarrow\infty}{\left\|P_{X_k}^+ - P_{X_k}^-\right\|_{\mathsf{TV}}} & \leq 1 - \P\!\left(\exists k \in \N, \, T \leq k\right) \nonumber \\
& = 1 - \P\!\left(T < \infty\right) . 
\end{align}
These bounds correspond to the ones shown in \cite[Theorem 5.2]{LevinPeresWilmer2009}. Since the event $A = \{\exists k \in \N, \, T \leq k\} = \{T < \infty\}$, our proof that $A$ happens almost surely also demonstrates that the two marginal Markov chains meet after a finite amount of time almost surely.

\section{Coding Theoretic Analysis of 2D Regular Grid with XOR Processing Functions}
\label{Analysis of Deterministic Xor Grid}

We now turn to proving Theorem \ref{Thm: Deterministic Xor Grid}. We will use some rudimentary coding theory ideas in this section, and refer readers to \cite{RichardsonUrbanke2008} for an introduction to the subject. We let $\mathbb{F}_2 = \{0,1\}$ denote the Galois field of order $2$ (i.e. integers with addition and multiplication modulo $2$), $\mathbb{F}_2^n$ with $n \in \N\backslash\!\{0,1\}$ denote the vector space over $\mathbb{F}_2$ of column vectors with $n$ entries from $\mathbb{F}_2$, and $\mathbb{F}_2^{m \times n}$ with $m,n \in \N\backslash\!\{0,1\}$ denote the space of $m \times n$ matrices with entries in $\mathbb{F}_2$. (All matrix and vector operations in this section will be performed modulo $2$.) Now fix some matrix $H \in \mathbb{F}_2^{m \times n}$ that has the following block structure:
\begin{equation}
\label{Eq: Block Structure of PCM}
H = \left[ 
\begin{array}{cc}
1 & B_1 \\
\0 & B_2 
\end{array} \right]
\end{equation}
where $\0 \triangleq [0 \cdots 0]^{\T} \in \mathbb{F}_2^{m-1}$ denotes the zero vector (whose dimension will be understood from context in the sequel), $B_1 \in \mathbb{F}_2^{1 \times (n-1)}$, and $B_2 \in \mathbb{F}_2^{(m-1)\times(n-1)}$. Consider the following two problems:
\begin{enumerate}
\item \textbf{Coding Problem:} Let $\C \triangleq \{x \in \mathbb{F}_2^{n} : H x = \0\}$ be the \textit{linear code} defined by the \textit{parity check matrix} $H$. Let $X = [X_1 \enspace X_2^{\T}]^{\T}$ with $X_1 \in \mathbb{F}_2$ and $X_2 \in \mathbb{F}_2^{n-1}$ be a codeword drawn uniformly from $\C$. Assume that there exists a codeword $x = [1 \enspace x_2^{\T}]^{\T} \in \C$ (i.e. $B_1 x_2 = 1$ and $B_2 x_2 = \0$). Then, since $\C \ni x^{\prime} \mapsto x^{\prime} + x \in \C$ is a bijective map that flips the first bit of its input, $X_1$ is a $\Ber\big(\frac{1}{2}\big)$ random variable. We observe the codeword $X$ through an additive noise channel model and see $Y_1 \in \mathbb{F}_2$ and $Y_2 \in \mathbb{F}_2^{n-1}$:
\begin{equation}
\label{Eq: Coding Problem}
\left[ \begin{array}{c} Y_1 \\ Y_2 \end{array} \right] = X + \left[ \begin{array}{c} Z_1 \\ Z_2 \end{array} \right] = \left[ \begin{array}{c} X_1 + Z_1 \\ X_2 + Z_2 \end{array} \right]
\end{equation} 
where $Z_1 \in \mathbb{F}_2$ is a $\Ber\big(\frac{1}{2}\big)$ random variable, $Z_2 \in \mathbb{F}_2^{n-1}$ is a vector of i.i.d. $\Ber(\delta)$ random variables that are independent of $Z_1$, and both $Z_1,Z_2$ are independent of $X$. Our problem is to decode $X_1$ with minimum probability of error after observing $Y_1,Y_2$. This can be achieved by using the ML decoder for $X_1$ based on $Y_1,Y_2$.
\item \textbf{Inference Problem:} Let $X^{\prime} \in \mathbb{F}_2$ be a $\Ber\big(\frac{1}{2}\big)$ random variable, and $Z \in \mathbb{F}_2^{n-1}$ be a vector of i.i.d. $\Ber(\delta)$ random variables that are independent of $X^{\prime}$. Suppose we see the observations $S_1^{\prime} \in \mathbb{F}_2$ and $S_2^{\prime} \in \mathbb{F}_2^{m-1}$ through the model:
\begin{equation}
\label{Eq: Inference Problem}
\left[ \begin{array}{c} S_1^{\prime} \\ S_2^{\prime} \end{array} \right] = H \left[ \begin{array}{c} X^{\prime} \\ Z \end{array} \right] =  \left[ \begin{array}{c} X^{\prime} + B_1 Z \\ B_2 Z \end{array} \right] . 
\end{equation}
Our problem is to decode $X^{\prime}$ with minimum probability of error after observing $S_1^{\prime},S_2^{\prime}$. This can be achieved by using the ML decoder for $X^{\prime}$ based on $S_1^{\prime},S_2^{\prime}$.
\end{enumerate}
As we will soon see, the inference problem above corresponds to our setting of reconstruction in the 2D regular grid with XOR processing functions. The next lemma illustrates that this inference problem is in fact ``equivalent'' to the aforementioned coding problem, and this connection will turn out to be useful since the coding problem admits simpler analysis.

\begin{lemma}[Equivalence of Problems]
\label{Lemma: Equivalence of Problems}
For the coding problem in \eqref{Eq: Coding Problem} and the inference problem in \eqref{Eq: Inference Problem}, the following statements hold:
\begin{enumerate}
\item The minimum probabilities of error for the coding and inference problems are equal. 
\item Suppose the random variables in the coding and inference problems are coupled so that $X_1 = X^{\prime}$ and $Z_2 = Z$ almost surely (i.e. these variables are shared by the two problems), $X_2$ is generated from a conditional distribution $P_{X_2|X_1}$ such that $X$ is uniform on $\C$, $Z_1$ is generated independently, $(Y_1,Y_2)$ is defined by \eqref{Eq: Coding Problem}, and $(S_1^{\prime},S_2^{\prime})$ is defined by \eqref{Eq: Inference Problem}. Then, $S_1^{\prime} = B_1 Y_2$ and $S_2^{\prime} = B_2 Y_2$ almost surely.
\item Under the aforementioned coupling, $(S_1^{\prime},S_2^{\prime})$ is a sufficient statistic of $(Y_1,Y_2)$ for performing inference about $X_1$ (in the coding problem).
\end{enumerate}
\end{lemma}

\begin{proof} ~\newline
\indent
\textbf{Part 1:}
We first show that the minimum probabilities of error for the two problems are equal. The inference problem has prior $X^{\prime} \sim \Ber\big(\frac{1}{2}\big)$, and the following likelihoods for every $s_1^{\prime} \in \mathbb{F}_2$ and every $s_2^{\prime} \in \mathbb{F}_2^{m-1}$:
\begin{align}
& P_{S_1^{\prime},S_2^{\prime}|X^{\prime}}\!\left(s_1^{\prime},s_2^{\prime}\middle|0\right) \nonumber \\
& \qquad \quad = \sum_{z \in \mathbb{F}_2^{n-1}}{P_Z(z) \I\!\left\{B_1 z = s_1^{\prime}, B_2 z = s_2^{\prime}\right\}} \, , \label{Eq: Inference Likelihood 1} \\
& P_{S_1^{\prime},S_2^{\prime}|X^{\prime}}\!\left(s_1^{\prime},s_2^{\prime}\middle|1\right) \nonumber \\
& \qquad \quad = \sum_{z \in \mathbb{F}_2^{n-1}}{P_Z(z) \I\!\left\{B_1 z = s_1^{\prime} + 1, B_2 z = s_2^{\prime}\right\}} \, .
\label{Eq: Inference Likelihood 2}
\end{align}
On the other hand, the coding problem has prior $X_1 \sim \Ber\big(\frac{1}{2}\big)$, and the following likelihoods for every $y_1 \in \mathbb{F}_2$ and every $y_2 \in \mathbb{F}_2^{n-1}$:
\begin{align}
& P_{Y_1,Y_2|X_1}(y_1,y_2|0) \nonumber \\
& = P_{Y_1|X_1}(y_1|0) P_{Y_2|X_1}(y_2|0) \nonumber \\
& = \frac{1}{2} \sum_{x_2 \in \mathbb{F}_2^{n-1}}{P_{Y_2|X_2}(y_2|x_2) P_{X_2|X_1}(x_2|0)} \nonumber \\
& = \frac{1}{2} \sum_{x_2 \in \mathbb{F}_2^{n-1}}{P_{Z_2}(y_2-x_2) \I\!\left\{B_1 x_2 = 0, B_2 x_2 = \0\right\} \frac{2}{|\C|}} \nonumber \\
& = \frac{1}{|\C|} \sum_{z_2 \in \mathbb{F}_2^{n-1}}{P_{Z_2}(z_2) \I\!\left\{B_1 z_2 = B_1 y_2, B_2 z_2 = B_2 y_2\right\}} \, , \label{Eq: Coding Likelihood 1} \\
& P_{Y_1,Y_2|X_1}(y_1,y_2|1) \nonumber \\
& = \frac{1}{|\C|} \sum_{z_2 \in \mathbb{F}_2^{n-1}}{P_{Z_2}(z_2) \I\!\left\{B_1 z_2 = B_1 y_2 + 1, B_2 z_2 = B_2 y_2\right\}} \, , 
\label{Eq: Coding Likelihood 2}
\end{align}
where the third equality uses the fact that $X_2$ is uniform over a set of cardinality $|\C|/2$ given any value of $X_1$, because $X_1 \sim \Ber\big(\frac{1}{2}\big)$ and $X$ is uniform on $\C$. For the coding problem, define $S_1 \triangleq B_1 Y_2$ and $S_2 \triangleq B_2 Y_2$. Due to the Fisher-Neyman factorization theorem \cite[Theorem 3.6]{Keener2010}, \eqref{Eq: Coding Likelihood 1} and \eqref{Eq: Coding Likelihood 2} demonstrate that $(S_1,S_2)$ is a sufficient statistic of $(Y_1,Y_2)$ for performing inference about $X_1$. 

Continuing in the context of the coding problem, define the set:
$$ \C^{\prime} \triangleq \left\{x \in \mathbb{F}_2^{n-1} : B_1 x = 0 , B_2 x = \0\right\} , $$
which is also a linear code, and for any fixed $s_1 \in \mathbb{F}_2$ and $s_2 \in \mathbb{F}_2^{m-1}$, define the set:
$$ \mathcal{S}(s_1,s_2) \triangleq \left\{(y_1,y_2) \in \mathbb{F}_2 \times \mathbb{F}_2^{n-1} \! : B_1 y_2 = s_1, B_2 y_2 = s_2\right\} \! . $$
If there exists $y_2^{\prime} \in \mathbb{F}_2^{n-1}$ such that $B_1 y_2^{\prime} = s_1$ and $B_2 y_2^{\prime} = s_2$, then $\mathcal{S}(s_1,s_2) = \{(y_1,y_2 + y_2^{\prime}) \in \mathbb{F}_2 \times \mathbb{F}_2^{n-1} : y_2 \in \C^{\prime}\}$, which means that $|\mathcal{S}(s_1,s_2)| = 2|\C^{\prime}| = |\C|$ (where the final equality holds because each vector in $\C^{\prime}$ corresponds to a codeword in $\C$ whose first letter is $0$, and we have assumed that there are an equal number of codewords in $\C$ with first letter $1$). Hence, for every $s_1 \in \mathbb{F}_2$ and every $s_2 \in \mathbb{F}_2^{m-1}$, the likelihoods of $(S_1,S_2)$ given $X_1$ can be computed from \eqref{Eq: Coding Likelihood 1} and \eqref{Eq: Coding Likelihood 2}:
\begin{align}
& P_{S_1,S_2|X_1}(s_1,s_2|0) \nonumber \\
& = \!\!\!\! \sum_{y_1 \in \mathbb{F}_2, \, y_2 \in \mathbb{F}_2^{n-1}}{\!\!\! P_{Y_1,Y_2|X_1}(y_1,y_2|0) \I\!\left\{B_1 y_2 = s_1, B_2 y_2 = s_2\right\}} \nonumber \\
& = \frac{|\mathcal{S}(s_1,s_2)|}{|\C|} \sum_{z_2 \in \mathbb{F}_2^{n-1}}{P_{Z_2}(z_2) \I\!\left\{B_1 z_2 = s_1, B_2 z_2 = s_2\right\}} \nonumber \\
& = \sum_{z_2 \in \mathbb{F}_2^{n-1}}{P_{Z_2}(z_2) \I\!\left\{B_1 z_2 = s_1, B_2 z_2 = s_2\right\}} \, , \label{Eq: Sufficient Coding Likelihood 1} \\
& P_{S_1,S_2|X_1}(s_1,s_2|1) \nonumber \\
& = \sum_{z_2 \in \mathbb{F}_2^{n-1}}{P_{Z_2}(z_2) \I\!\left\{B_1 z_2 = s_1 + 1, B_2 z_2 = s_2\right\}} \, ,
\label{Eq: Sufficient Coding Likelihood 2}
\end{align}
where the second equality follows from \eqref{Eq: Coding Likelihood 1} and the third equality clearly holds in the $|\mathcal{S}(s_1,s_2)| = 0$ case as well. The likelihoods \eqref{Eq: Sufficient Coding Likelihood 1} and \eqref{Eq: Sufficient Coding Likelihood 2} are exactly the same as the likelihoods \eqref{Eq: Inference Likelihood 1} and \eqref{Eq: Inference Likelihood 2}, respectively, that we computed earlier for the inference problem. Thus, the sufficient statistic $(S_1,S_2)$ of $(Y_1,Y_2)$ for $X_1$ in the coding problem is equivalent to the observation $(S_1^{\prime},S_2^{\prime})$ in the inference problem in the sense that they are defined by the same probability model. As a result, the minimum probabilities of error in these formulations must be equal.

\textbf{Part 2:} We now assume that the random variables in the two problems are coupled as in the lemma statement. To prove that $S_1^{\prime} = S_1$ and $S_2^{\prime} = S_2$ almost surely, observe that:
\begin{align*}
\left[ \begin{array}{c} S_1 \\ S_2 \end{array} \right] & = \left[ \begin{array}{c} B_1 Y_2 \\ B_2 Y_2 \end{array} \right] \\
& = \left[ \begin{array}{c} B_1 X_2 + B_1 Z_2 \\ B_2 X_2 + B_2 Z_2 \end{array} \right] \\
& = \left[ \begin{array}{c} X_1 + B_1 Z_2 \\ B_2 Z_2 \end{array} \right] \\
& = H \left[ \begin{array}{c} X_1 \\ Z_2 \end{array} \right] \\
& = \left[ \begin{array}{c} S_1^{\prime} \\ S_2^{\prime} \end{array} \right] 
\end{align*}
where the second equality uses \eqref{Eq: Coding Problem}, the third equality holds because $B_1 X_2 = X_1$ and $B_2 X_2 = \0$ since $X \in \C$ is a codeword, and the last equality uses \eqref{Eq: Inference Problem} and the fact that $X_1 = X^{\prime}$ and $Z_2 = Z$ almost surely. This proves part 2.

\textbf{Part 3:} Since $(S_1,S_2)$ is a sufficient statistic of $(Y_1,Y_2)$ for performing inference about $X_1$ in the coding problem, and $S_1^{\prime} = S_1$ and $S_2^{\prime} = S_2$ almost surely under the coupling in the lemma statement, $(S_1^{\prime},S_2^{\prime})$ is also a sufficient statistic of $(Y_1,Y_2)$ for performing inference about $X_1$ under this coupling. This completes the proof.
\end{proof}

Recall that we are given a 2D regular grid where all Boolean processing functions with two inputs are the XOR rule, and all Boolean processing functions with one input are the identity rule, i.e. $f_{2}(x_1,x_2) = x_1 \oplus x_2$ and $f_{1}(x) = x$. We next prove Theorem \ref{Thm: Deterministic Xor Grid} using Lemma \ref{Lemma: Equivalence of Problems}.

\renewcommand{\proofname}{Proof of Theorem \ref{Thm: Deterministic Xor Grid}}

\begin{proof}
We first prove that the problem of decoding the root bit in the XOR 2D regular grid is captured by the inference problem defined in \eqref{Eq: Inference Problem}. Let $E_k$ denote the set of all directed edges in the 2D regular grid above level $k \in \N$. Furthermore, let us associate each edge $e \in E_k$ with an independent $\Ber(\delta)$ random variable $Z_e \in \mathbb{F}_2$. Since a $\mathsf{BSC}(\delta)$ can be modeled as addition of an independent $\Ber(\delta)$ bit (in $\mathbb{F}_2$), the random variables $\{Z_e : e \in E_k\}$ define the BSCs of the 2D regular grid up to level $k$. Moreover, each vertex at level $k \in \N\backslash\!\{0\}$ of the XOR 2D regular grid is simply a sum (in $\mathbb{F}_2$) of its parent vertices and the random variables on the edges between it and its parents:
\begin{align*}
\forall j \, \in \{1,\dots,& \, k-1\}, \\
X_{k,j} & = X_{k-1,j-1} \oplus X_{k-1,j} \\
& \quad \, \oplus Z_{(X_{k-1,j-1},X_{k,j})} \oplus Z_{(X_{k-1,j},X_{k,j})} \, , \\
X_{k,0} & = X_{k-1,0} \oplus Z_{(X_{k-1,0},X_{k,0})} \, , \\
X_{k,k} & = X_{k-1,k-1} \oplus Z_{(X_{k-1,k-1},X_{k,k})} \, .
\end{align*}
These recursive formulae for each vertex in terms of its parent vertices can be unwound so that each vertex is represented as a linear combination (in $\mathbb{F}_2$) of the root bit and all the edge random variables: 
\begin{equation}
\label{Eq: Sum Representation of Nodes}
\begin{aligned}
\forall k \in \N&\backslash\!\{0\}, \, \forall j \in [k+1], \\
X_{k,j} & =  \left(\binom{k}{j} \Mod{2}\right) X_{0,0} \, + \sum_{e \in E_k}{b_{k,j,e} Z_e} 
\end{aligned}
\end{equation}
where the coefficient of $X_{0,0}$ can be computed by realizing that the coefficients of the vertices in the ``2D regular grid above $X_{k,j}$'' (with $X_{k,j}$ as the root) are defined by the recursion of Pascal's triangle, and $\{b_{k,j,e} \in \mathbb{F}_2: k \in \N\backslash\!\{0\}, \, j \in [k+1], \, e \in E_k\}$ are some fixed coefficients. We do not require detailed knowledge of the values of $\{b_{k,j,e} \in \mathbb{F}_2: k \in \N\backslash\!\{0\}, \, j \in [k+1], \, e \in E_k\}$, but they can also be evaluated via straightforward counting if desired. 

In the remainder of this proof, we will fix $k$ to be a power of $2$: $k = 2^m$ for $m \in  \N\backslash\!\{0\}$. Then, we have:
\begin{equation}
\label{Eq: Lucas}
\binom{k}{j} \equiv \binom{2^m}{j} \equiv \left\{
\begin{array}{ll}
1 \, , & j \in \{0,k\} \\
0 \, ,  & j \in \{1,\dots,k-1\}
\end{array}
\right. \Mod{2}
\end{equation}
since by \emph{Lucas' theorem} (see \cite{Fine1947}), the parity of $\binom{k}{j}$ is $0$ if and only if at least one of the digits of $j$ in base $2$ is strictly greater than the corresponding digit of $k$ in base $2$, and the base $2$ representation of $k = 2^m$ is $10\cdots0$ (with $m$ $0$'s). So, for each $k$, we can define a binary matrix $H_k \in \mathbb{F}_2^{(k+1)\times (|E_k|+1)}$ whose rows are indexed by the vertices at level $k$ and columns are indexed by $1$ (first index corresponding to $X_{0,0}$) followed by the edges in $E_k$, and whose rows are made up of the coefficients in \eqref{Eq: Sum Representation of Nodes} (where the first entry of each row is given by \eqref{Eq: Lucas}). Clearly, we can write \eqref{Eq: Sum Representation of Nodes} in matrix-vector form using $H_k$ for every $k$:
\begin{equation}
\label{Eq: Xor Grid Matrix Version}
\left[ \begin{array}{c} X_{k,0} \\ X_{k,1} \\ \vdots \\ X_{k,k-1} \\ X_{k,k} \end{array} \right] = \underbrace{\left[ \begin{array}{cccc} 1 & \text{---} & b_{k,0,e} & \text{---} \\ 0 & \text{---} & b_{k,1,e} & \text{---} \\ \vdots & & \vdots & \\ 0 & \text{---} & b_{k,k-1,e} & \text{---} \\ 1 & \text{---} & b_{k,k,e} & \text{---} \end{array} \right]}_{\displaystyle{\triangleq H_k}} \left[ \begin{array}{c} X_{0,0} \\ \mid \\ Z_e \\ \mid \end{array} \right] 
\end{equation}
where the vector on the right hand side of \eqref{Eq: Xor Grid Matrix Version} has first element $X_{0,0}$ followed by the random variables $\{Z_e : e \in E_k\}$ (indexed consistently with $H_k$). Our XOR 2D regular grid reconstruction problem is to decode $X_{0,0}$ from the observations $(X_{k,0},\dots,X_{k,k})$ with minimum probability of error. Note that we can apply a row operation to $H_k$ that replaces the last row of $H_k$ with the sum of the first and last rows of $H_k$ to get the binary matrix $H_k^{\prime} \in \mathbb{F}_2^{(k+1)\times (|E_k|+1)}$, and correspondingly, we can replace $X_{k,k}$ with $X_{k,0} + X_{k,k}$ in \eqref{Eq: Xor Grid Matrix Version} to get the ``equivalent'' formulation:
\begin{equation}
\label{Eq: Row Operated Problem}
\left[ \begin{array}{c} X_{k,0} \\ X_{k,1} \\ \vdots \\ X_{k,k-1} \\ X_{k,0} + X_{k,k} \end{array} \right] = H_k^{\prime} \left[ \begin{array}{c} X_{0,0} \\ \mid \\ Z_e \\ \mid \end{array} \right] 
\end{equation}
for every $k$. Indeed, since we only perform invertible operations to obtain \eqref{Eq: Row Operated Problem} from \eqref{Eq: Xor Grid Matrix Version}, the minimum probability of error for ML decoding $X_{0,0}$ from the observations $(X_{k,0},\dots,X_{k,k})$ under the model \eqref{Eq: Xor Grid Matrix Version} is equal to the minimum probability of error for ML decoding $X_{0,0}$ from the observations $(X_{k,0},\dots,X_{k,k-1},X_{k,0} + X_{k,k})$ under the model \eqref{Eq: Row Operated Problem}. Furthermore, since $H_k^{\prime}$ is of the form \eqref{Eq: Block Structure of PCM}, the equivalent XOR 2D regular grid reconstruction problem in \eqref{Eq: Row Operated Problem} is exactly of the form of the inference problem in \eqref{Eq: Inference Problem}.

We next transform the XOR 2D regular grid reconstruction problem in \eqref{Eq: Xor Grid Matrix Version}, or equivalently, \eqref{Eq: Row Operated Problem}, into a coding problem. By Lemma \ref{Lemma: Equivalence of Problems}, the inference problem in \eqref{Eq: Row Operated Problem} is ``equivalent'' to a coupled coding problem analogous to \eqref{Eq: Coding Problem}. To describe this coupled coding problem, consider the linear code defined by the parity check matrix $H_k^{\prime}$:
\begin{align*}
\C_k & \triangleq \left\{w \in \mathbb{F}_2^{|E_k| + 1} : H_k^{\prime} w = \0 \right\} \\
& = \left\{w \in \mathbb{F}_2^{|E_k| + 1} : H_k w = \0 \right\} 
\end{align*}
where the second equality shows that the parity check matrix $H_k$ also generates $\C_k$ (because row operations do not change the nullspace of a matrix). As required by the coding problem, this linear code contains a codeword of the form $[1 \enspace w_2^{\T}]^{\T} \in \C_k$ for some $w_2 \in \mathbb{F}_2^{|E_k|}$. To prove this, notice that such a codeword exists if and only if the first column $[1 \, 0 \cdots 0]^{\T}$ of $H_k^{\prime}$ is in the span of the remaining columns of $H_k^{\prime}$. Assume for the sake of contradiction that such a codeword does not exist. Then, we can decode $X_{0,0}$ in the setting of \eqref{Eq: Row Operated Problem} with zero probability of error, because the observation vector on the left hand side of \eqref{Eq: Row Operated Problem} is in the span of the second to last columns of $H_k^{\prime}$ if and only if $X_{0,0} = 0$.\footnote{It is worth mentioning that in the ensuing coding problem in \eqref{Eq: Xor Coding Problem}, if such a codeword does not exist, we can also decode the first codeword bit with zero probability of error because all codewords must have the first bit equal to $0$.} This leads to a contradiction since it is clear that we cannot decode the root bit with zero probability of error in the XOR 2D regular grid. Hence, a codeword of the form $\big[1 \enspace w_2^{\T}\big]^{\T} \in \C_k$ for some $w_2 \in \mathbb{F}_2^{|E_k|}$ always exists. Next, we let $W_k = [X_{0,0} \enspace \text{---} \, W_{k,e} \, \text{---} \, ]^{\T} \in \C_k$ be a codeword that is drawn uniformly from $\C_k$, where the first element of $W_k$ is $X_{0,0}$ and the remaining elements of $W_k$ are $\{W_{k,e} : e \in E_k\}$. In the coupled coding problem, we observe $W_k$ through the additive noise channel model:
\begin{equation}
\label{Eq: Xor Coding Problem}
Y_k \triangleq \left[ \begin{array}{c} Y_{0,0}^k \\ \mid \\ Y_{k,e} \\ \mid \end{array} \right] = W_k + \left[ \begin{array}{c} Z_{0,0}^k \\ \mid \\ Z_{e} \\ \mid \end{array} \right]
\end{equation}
where $\{Z_e : e \in E_k\}$ are the BSC random variables that are independent of $W_k$, $Z_{0,0}^k$ is a completely independent $\Ber\big(\frac{1}{2}\big)$ random variable, $Y_{0,0}^k = X_{0,0} \oplus Z_{0,0}^k$, and $Y_{k,e} = W_{k,e} \oplus Z_{e}$ for $e \in E_k$. Our goal is to decode the first bit of the codeword, $X_{0,0}$, with minimum probability of error from the observation $Y_k$. Since we have coupled the coding problem \eqref{Eq: Xor Coding Problem} and the inference problem \eqref{Eq: Row Operated Problem} according to the coupling in part 2 of Lemma \ref{Lemma: Equivalence of Problems}, part 3 of Lemma \ref{Lemma: Equivalence of Problems} shows that $(X_{k,0},\dots,X_{k,k-1},X_{k,0} + X_{k,k})$, or equivalently:
\begin{equation}
\label{Eq: X_k as deterministic function of Y_k}
\left[ \begin{array}{c} X_{k,0} \\ \vdots \\ X_{k,k} \end{array} \right] = \left[\begin{array}{c} \displaystyle{\sum_{e \in E_{k}}{b_{k,0,e} Y_{k,e}}} \\ \vdots \\ \displaystyle{\sum_{e \in E_{k}}{b_{k,k,e} Y_{k,e}}} \end{array} \right] , 
\end{equation}
is a sufficient statistic of $Y_k$ for performing inference about $X_{0,0}$ in the coding problem \eqref{Eq: Xor Coding Problem}. Hence, the ML decoder for $X_{0,0}$ based on the sufficient statistic $(X_{k,0},\dots,X_{k,k-1},X_{k,0} + X_{k,k})$ (without loss of generality), which achieves the minimum probability of error in the coding problem \eqref{Eq: Xor Coding Problem}, makes an error if and only if the ML decision rule for $X_{0,0}$ based on $(X_{k,0},\dots,X_{k,k-1},X_{k,0} + X_{k,k})$, which achieves the minimum probability of error in the inference problem \eqref{Eq: Row Operated Problem}, makes an error. Therefore, as shown in part 1 of Lemma \ref{Lemma: Equivalence of Problems}, the minimum probabilities of error in the XOR 2D regular grid reconstruction problem \eqref{Eq: Xor Grid Matrix Version} and the coding problem \eqref{Eq: Xor Coding Problem} are equal, and it suffices to analyze the coding problem \eqref{Eq: Xor Coding Problem}.

In the coding problem \eqref{Eq: Xor Coding Problem}, we observe the codeword $W_k$ after passing it through memoryless BSCs. We now establish a ``cleaner'' model where $W_k$ is passed through memoryless binary erasure channels. Recall that each $\mathsf{BSC}(\delta)$ copies its input bit with probability $1-2\delta$ and generates an independent $\Ber\big(\frac{1}{2}\big)$ output bit with probability $2\delta$ (as shown in \eqref{Eq: BSC Decomposition} in section \ref{Analysis of Deterministic And Grid}), i.e. for any $e \in E_k$, instead of setting $Z_e \sim \Ber(\delta)$, we can generate $Z_e$ as follows:
$$ Z_e = \left\{
\begin{array}{ll}
0 \, , & \text{with probability } 1-2\delta \\
\Ber\!\left(\frac{1}{2}\right) , & \text{with probability } 2\delta 
\end{array}
\right. $$ 
where $\Ber\big(\frac{1}{2}\big)$ denotes an independent uniform bit. Suppose we know which BSCs among $\{Z_e : e \in E_k\}$ generate independent bits in \eqref{Eq: Xor Coding Problem}. Then, we can perceive each BSC in $\{Z_e : e \in E_k\}$ as an independent \textit{binary erasure channel} (BEC) with erasure probability $2\delta$, denoted $\mathsf{BEC}(2\delta)$, which erases its input with probability $2\delta$ and produces the erasure symbol $\mathsf{e}$ if and only if the corresponding $\mathsf{BSC}(\delta)$ generates an independent bit, and copies its input with probability $1-2\delta$ otherwise. (Note that the BSC defined by $Z_{0,0}^k$ corresponds to a $\mathsf{BEC}(1)$ which always erases its input.) Consider observing the codeword $W_k$ under this \textit{BEC model}, where $X_{0,0}$ is erased almost surely, and the remaining bits of $W_k$ are erased independently with probability $2\delta$, i.e. we observe $Y_k^{\prime} = [\mathsf{e} \enspace \text{---} \, Y_{k,e}^{\prime} \, \text{---} \,]^{\T} \in \{0,1,\mathsf{e}\}^{|E_k| + 1}$, where the first entry corresponds to the erased value of $X_{0,0}$, and for every $e \in E_k$, $Y_{k,e}^{\prime} = W_{k,e}$ with probability $1-2\delta$ and $Y_{k,e}^{\prime} = \mathsf{e}$ with probability $2\delta$. Clearly, we can obtain $Y_k$ from $Y_k^{\prime}$ by replacing every instance of $\mathsf{e}$ in $Y_k^{\prime}$ with an independent $\Ber\big(\frac{1}{2}\big)$ bit. Since the BECs reveal additional information about which BSCs generate independent bits, the minimum probability of error in ML decoding $X_{0,0}$ based on $Y_k^{\prime}$ under the BEC model lower bounds the minimum probability of error in ML decoding $X_{0,0}$ based on $Y_k$ under the BSC model \eqref{Eq: Xor Coding Problem}.\footnote{Indeed, the ML decoder for $X_{0,0}$ based on $Y_k^{\prime}$ has a smaller (or equal) probability of error than the decoder which first translates $Y_k^{\prime}$ into $Y_k$ by replacing every $\mathsf{e}$ with an independent $\Ber\big(\frac{1}{2}\big)$ bit, and then applies the ML decoder for $X_{0,0}$ based on $Y_k$ as in the coding problem \eqref{Eq: Xor Coding Problem}. We also remark that the fact that a $\mathsf{BEC}(2\delta)$ is ``less noisy'' than a $\mathsf{BSC}(\delta)$ is well-known in information theory, cf. \cite[Section 6, Equation (16)]{PolyanskiyWu2017}.} In the rest of the proof, we establish conditions under which the minimum probability of error for the BEC model is $\frac{1}{2}$, and then show as a consequence that the minimum probability of error in the XOR 2D regular grid reconstruction problem in \eqref{Eq: Xor Grid Matrix Version} tends to $\frac{1}{2}$ as $k \rightarrow \infty$.

Let $I_k \subseteq E_k$ denote the set of indices where the corresponding elements of $W_k$ are \textit{not} erased in the BEC model:
$$ I_k \triangleq \left\{e \in E_k : Y_{k,e}^{\prime} = W_{k,e} \right\} . $$
The ensuing lemma is a standard exercise in coding theory which shows that the ML decoder for $X_{0,0}$ only fails under the BEC model when a special codeword exists in $\C_k$; see the discussion in \cite[Section 3.2]{RichardsonUrbanke2008}.

\begin{lemma}[Bit-wise ML Decoding {\cite[Section 3.2]{RichardsonUrbanke2008}}]
\label{Lemma: Bit-wise ML Decoding}
Suppose we condition on some realization of $Y_k^{\prime}$ (in the BEC model), which determines a corresponding realization of the set of indices $I_k$. Then, the ML decoder for $X_{0,0}$ based on $Y_k^{\prime}$ (with codomain $\mathbb{F}_2$) makes an error with probability $\frac{1}{2}$ if and only if there exists a codeword $w \in \C_k$ with first element $w_1 = 1$ and $w_e = 0$ for all $e \in I_k$.
\end{lemma}

We next illustrate that such a special codeword exists whenever two particular erasures occur. Let $e_1 \in E_k$ and $e_2 \in E_k$ denote the edges $(X_{k-1,0},X_{k,0})$ and $(X_{k-1,k-1},X_{k,k})$ in the 2D regular grid, respectively. Consider the vector $\omega^k \in \mathbb{F}_2^{|E_k| + 1}$ such that $\omega^k_1 = 1$ (i.e. the first bit is $1$), $\omega^k_{e_1} = \omega^k_{e_2} = 1$, and all other elements of $\omega^k$ are $0$. Then, $\omega^k \in \C_k$ because:
\begin{align*}
H_k \, \omega^k & = \left[ \begin{array}{cccc} 1 & \text{---} & b_{k,0,e} & \text{---} \\ 0 & \text{---} & b_{k,1,e} & \text{---} \\ \vdots & & \vdots & \\ 0 & \text{---} & b_{k,k-1,e} & \text{---} \\ 1 & \text{---} & b_{k,k,e} & \text{---} \end{array} \right] \omega^k \\
& = \left[ \begin{array}{c} 1 \oplus b_{k,0,e_1} \oplus b_{k,0,e_2} \\ b_{k,1,e_1} \oplus b_{k,1,e_2} \\ \vdots \\ b_{k,k-1,e_1} \oplus b_{k,k-1,e_2} \\ 1 \oplus b_{k,k,e_1} \oplus b_{k,k,e_2} \end{array} \right] \\
& = \0 
\end{align*} 
where we use the facts that $b_{k,0,e_1} = 1$, $b_{k,0,e_2} = 0$, $b_{k,k,e_1} = 0$, $b_{k,k,e_2} = 1$, and for any $j \in \{1,\dots,k-1\}$, $b_{k,j,e_1} = 0$ and $b_{k,j,e_2} = 0$. (Note that the value of $b_{k,j,e_i}$ for $i \in \{0,1\}$ and $j \in [k+1]$ is determined by checking the dependence of vertex $X_{k,j}$ on the variable $Z_{e_i}$ in \eqref{Eq: Sum Representation of Nodes}, which is straightforward because $e_i$ is an edge between the last two layers at the side of the 2D regular grid up to level $k$). Since $\omega^k$ has two $1$'s at the indices $e_1$ and $e_2$ (besides the first bit), if the BECs corresponding to the indices $e_1$ and $e_2$ erase their inputs, i.e. $e_1,e_2 \notin I_k$, then $\omega^k \in \C_k$ satisfies the conditions of Lemma \ref{Lemma: Bit-wise ML Decoding} and the ML decoder for $X_{0,0}$ based on $Y_k^{\prime}$ under the BEC model makes an error with probability $\frac{1}{2}$. Hence, we define the event:
\begin{align*}
B_k & \triangleq \left\{Y_{k,e_1}^{\prime} = Y_{k,e_2}^{\prime} = \mathsf{e}\right\} \\
& = \left\{\parbox[]{18em}{BECs corresponding to edges $e_1 \in E_k$ and $e_2 \in E_k$ erase their inputs}\right\} \\
& = \left\{\parbox[]{18em}{BSCs corresponding to edges $e_1 \in E_k$ and $e_2 \in E_k$ generate independent bits}\right\} .
\end{align*}
As the ML decoder for $X_{0,0}$ based on $Y_k^{\prime}$ under the BEC model makes an error with probability $\frac{1}{2}$ conditioned on $B_k$, we must have:
$$ P_{Y_k^{\prime}|X_{0,0}}(y^{\prime}|0) = P_{Y_k^{\prime}|X_{0,0}}(y^{\prime}|1) $$
for all realizations $y^{\prime} \in \{0,1,\mathsf{e}\}^{|E_k| + 1}$ of $Y_k^{\prime}$ such that $B_k$ occurs, i.e. $y^{\prime}_1 = y^{\prime}_{e_1} = y^{\prime}_{e_2} = \mathsf{e}$. This implies that $Y_k^{\prime}$ is conditionally independent of $X_{0,0}$ given $B_k$ (where we also use the fact that $X_{0,0}$ is independent of $B_k$). Furthermore, it is straightforward to verify that $Y_k$ is also conditionally independent of $X_{0,0}$ given $B_k$, because $Y_k$ can be obtained from $Y_k^{\prime}$ by replacing $\mathsf{e}$'s with completely independent $\Ber\big(\frac{1}{2}\big)$ bits. Thus, since \eqref{Eq: X_k as deterministic function of Y_k} shows that $X_k$ is a deterministic function of $Y_k$, $X_k$ is conditionally independent of $X_{0,0}$ given $B_k$.

To finish the proof, notice that $\P(B_k) = (2\delta)^2$ for every $k$, and the events $\{B_k : k = 2^m, \, m \in \N\backslash\!\{0\}\}$ are mutually independent because the BSCs in the 2D regular grid are all independent. So, infinitely many of the events $\{B_k : k = 2^m, \, m \in \N\backslash\!\{0\}\}$ occur almost surely by the second Borel-Cantelli lemma. Let us define:
$$ \forall n \in \N\backslash\!\{0\}, \enspace A_n \triangleq \bigcup_{m = 1}^{n}{B_{2^m}} $$ 
where the continuity of the underlying probability measure $\P$ yields $\lim_{n \rightarrow \infty}{\P(A_n)} = 1$. Then, since $X_k$ is conditionally independent of $X_{0,0}$ given $B_k$, and $X_{r}$ is conditionally independent of $X_{0,0}$ and $B_k$ given $X_k$ for any $r > k$, we have that $X_{2^m}$ is conditionally independent of $X_{0,0}$ given $A_m$ for every $m \in \N\backslash\!\{0\}$. Hence, we obtain:
$$ \forall m\in\N\backslash\!\{0\}, \enspace \P\!\left(h_{\mathsf{ML}}^{2^m}(X_{2^m}) \neq X_{0,0} \,\middle|\,A_m \right) = \frac{1}{2} $$
where $h_{\mathsf{ML}}^{k} : \mathbb{F}_2^{k+1} \rightarrow \mathbb{F}_2$ denotes the ML decoder for $X_{0,0}$ based on $X_{k}$ for the XOR 2D regular grid reconstruction problem in \eqref{Eq: Xor Grid Matrix Version}. Finally, observe that:
\begin{align*}
& \lim_{m \rightarrow \infty}{\P\!\left(h_{\mathsf{ML}}^{2^m}(X_{2^m}) \neq X_{0,0} \right)} \\
& \qquad = \lim_{m \rightarrow \infty} \P\!\left(h_{\mathsf{ML}}^{2^m}(X_{2^m}) \neq X_{0,0} \,\middle|\,A_m \right) \P(A_m) \\
& \qquad \qquad \quad \enspace + \P\!\left(h_{\mathsf{ML}}^{2^m}(X_{2^m}) \neq X_{0,0} \,\middle|\,A_m^{c} \right) \P(A_m^{c}) \\
& \qquad = \lim_{m \rightarrow \infty}{\P\!\left(h_{\mathsf{ML}}^{2^m}(X_{2^m}) \neq X_{0,0} \,\middle|\,A_m \right)} \\
& \qquad = \frac{1}{2} \, .
\end{align*}
This completes the proof since the above condition establishes \eqref{Eq: Deterministic Impossibility of Reconstruction}.   
\end{proof}

\renewcommand{\proofname}{Proof}

\section{Martingale Approach for 2D Regular Grid with NAND Processing Functions}
\label{Martingale Approach for NAND Processing Functions}

\begin{figure*}[t]
\centering
\includegraphics[trim = 21mm 60mm 30mm 72mm, clip, width=0.85\linewidth]{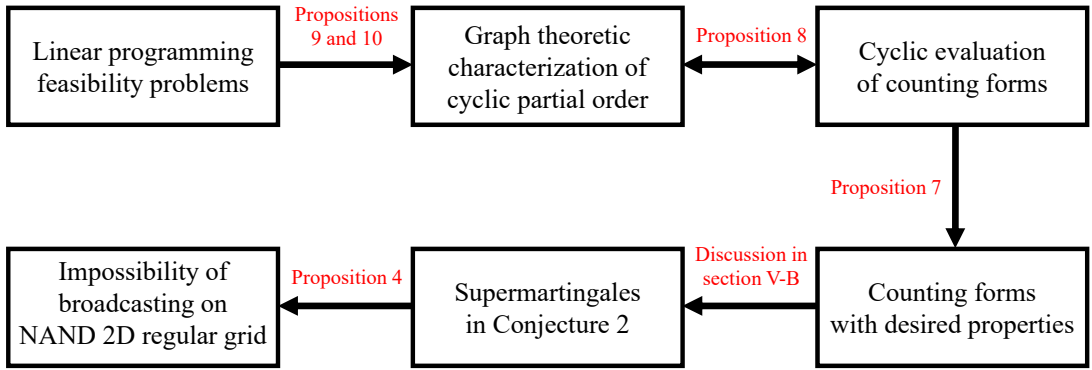} 
\caption{A high-level schematic of the proposed martingale-based approach for establishing the impossibility of broadcasting on 2D regular grids with NAND processing functions.}
\label{Figure: Martingale method}
\end{figure*}

Finally, we consider the 2D regular grid where all Boolean processing functions with two inputs are the NAND rule, and all Boolean processing functions with one input are the identity rule, i.e. $f_{2}(x_1,x_2) = \neg(x_1 \wedge x_2)$ and $f_{1}(x) = x$, where $\neg$ denotes the logical NOT operation. In this section, we will illustrate a promising program for proving Conjecture \ref{Conj: NAND 2D Regular Grid}; this program is easily \emph{applicable to other Boolean processing functions}, e.g., IMP, but we do not consider other functions here for brevity. In particular, we will establish Theorem \ref{Thm: Sufficient Condition for NAND 2D Regular Grid}, derive results required to generate the accompanying numerical evidence in Table \ref{Table: LP Solutions}, and in the process, present several peripheral results for completeness.

To guide the readers, we briefly outline the ensuing subsections. In subsection \ref{Existence of Supermartingale}, we will describe the Markovian coupling setup of subsection \ref{Partial Impossibility Result for NAND 2D Regular Grid}, and then show that the existence of certain structured supermartingales implies Conjecture \ref{Conj: NAND 2D Regular Grid}. In subsection \ref{Potential Functions}, inspired by \cite{HolroydMarcoviciMartin2019}, we will introduce counting forms and elucidate how they define the structured superharmonic potential functions that produce the aforementioned supermartingales. Then, to efficiently test the desired structural properties of counting forms, we will develop cyclic evaluation of counting forms and derive corresponding graph theoretic characterizations in subsection \ref{Graph Theoretic Characterization}. With these required pieces in place, we will prove Theorem \ref{Thm: Sufficient Condition for NAND 2D Regular Grid} in subsection \ref{Proof of Theorem Sufficient Condition for NAND 2D Regular Grid}. Lastly, we will transform the graph theoretic tests for the desired superharmonic potential functions in subsection \ref{Graph Theoretic Characterization} into simple LPs in subsection \ref{Linear Programming Criteria}. These LPs are then solved to generate Table \ref{Table: LP Solutions}. This chain of ideas is illustrated in Figure \ref{Figure: Martingale method} so that readers may refer back to it as they proceed through this section.

\subsection{Existence of Supermartingale}
\label{Existence of Supermartingale}

As mentioned in subsection \ref{Partial Impossibility Result for NAND 2D Regular Grid}, similar to our proof of Theorem \ref{Thm: Deterministic And Grid} in section \ref{Analysis of Deterministic And Grid}, we begin by constructing a Markovian coupling. Let $\{X^+_k : k \in \N\}$ and $\{X^-_k : k \in \N\}$ denote versions of the Markov chain $\{X_k : k \in \N\}$ initialized at $X^+_0 = 1$ and $X^-_0 = 0$, respectively. Moreover, define the coupled 2D grid variables $\{Y_{k,j} = (X_{k,j}^-,X_{k,j}^+) : k \in \N, \, j \in [k+1]\}$, which yield the Markovian coupling $\{Y_k = (Y_{k,0},\dots,Y_{k,k}) : k \in \N\}$. Recall from \eqref{Eq: BSC Decomposition} that each edge $\mathsf{BSC}(\delta)$ either copies its input bit with probability $1 - 2\delta$, or generates an independent $\Ber\big(\frac{1}{2}\big)$ output bit with probability $2\delta$. As before, the Markovian coupling $\{Y_k : k \in \N\}$ ``runs'' its ``marginal'' Markov chains $\{X^+_k : k \in \N\}$ and $\{X^-_k : k \in \N\}$ on a common underlying 2D regular grid so that along any edge BSC, either both inputs are copied with probability $1 - 2\delta$, or a shared independent $\Ber\big(\frac{1}{2}\big)$ output bit is produced with probability $2\delta$. We will analyze the Markov chain $\{Y_k : k \in \N\}$ in the remainder of this section.

To simplify our analysis, we will again keep track of $\{Y_k : k \in \N\}$ in a single coupled 2D regular grid. The edge configuration of the underlying graph of the coupled 2D regular grid is exactly as described in subsection \ref{Deterministic 2D Grid Model}. The vertices of the coupled 2D regular grid are indexed by the coupled 2D regular grid variables $\{Y_{k,j} \in \Y : k \in \N, \, j \in [k+1] \}$, where we simplify the alphabet set of these variables from $\{(0,0),(0,1),(1,0),(1,1)\}$ to:
\begin{equation}
\label{Eq: The coupled NAND alphabet}
\Y \triangleq \left\{0,1,u\right\}
\end{equation}
such that for any $k \in \N$ and any $j \in [k+1]$, $Y_{k,j} = 0$ only if $X_{k,j}^- = X_{k,j}^+ = 0$, $Y_{k,j} = 1$ only if $X_{k,j}^- = X_{k,j}^+ = 1$, and $Y_{k,j} = u$ means that either $X_{k,j}^- \neq X_{k,j}^+$ or it is ``unknown'' whether $X_{k,j}^- = X_{k,j}^+$.\footnote{Equivalently, we let $0 = (0,0)$ and $1 = (1,1)$ with abuse of notation, and we let $u$ represent the unknown (or uncoupled) state.} So, $Y_{k,j} = u$ represents that the variables $X_{k,j}^-$ and $X_{k,j}^+$ are uncoupled; in particular, we do not distinguish between the possible values the pair $(X_{k,j}^-,X_{k,j}^+)$ might have. Furthermore, much like \eqref{Eq: Coupled BSC}, each shared edge $\mathsf{BSC}(\delta)$ of the coupled 2D regular grid is described by a row stochastic matrix, or channel, on the alphabet $\Y$:
\begin{equation}
\label{Eq: Coupled BSC 2} 
W = \bbordermatrix{
		& 0 & u & 1 \cr
      0 & 1-\delta & 0 & \delta \cr
      u & \delta & 1-2\delta & \delta \cr
      1 & \delta & 0 & 1-\delta \cr } ,
\end{equation}
and $W$ appropriately captures the aforementioned Markovian coupling. Lastly, the NAND rule can be equivalently described on the alphabet $\Y$ as:\footnote{Note that the NAND of two $u$'s could be $1$, e.g. $(1,0)$ and $(0,1)$ pass through a NAND gate to produce $(1,1)$. However, we conservatively treat the NAND of two $u$'s as unknown, i.e. $u$.}
\begin{equation}
\label{Eq: Coupled Nand} 
\begin{array}{|c|c|c|}
\hline
y_1 & y_2 & \neg(y_1 \wedge y_2) \\
\hline
0 & \star & 1 \\
u & u & u \\
u & 1 & u \\
1 & 1 & 0 \\
\hline
\end{array}
\end{equation}
where $\star$ denotes any letter in $\Y$ as before, and the symmetry of the NAND rule covers all other possible input combinations. This coupled 2D regular grid model completely characterizes the Markov chain $\{Y_k : k \in \N\}$, which starts at $Y_0 = u$ almost surely.

For every $k \in \N$, define the random variable:
\begin{equation}
\label{Eq: N_k rv}
N_k \triangleq \sum_{j = 0}^{k}{\I\!\left\{Y_{k,j} = u\right\}}
\end{equation}
which gives the number of $u$'s at level $k$ in the coupled 2D regular grid, where $\I\{\cdot\}$ denotes the indicator function which equals $1$ if its input proposition is true, and $0$ otherwise. The next lemma illustrates that the almost sure convergence of $\{N_k : k \in\N\}$ to zero implies Conjecture \ref{Conj: NAND 2D Regular Grid}.

\begin{lemma}[Vanishing $u$'s Condition]
\label{Lemma: Vanishing u's Condition}
For any noise level $\delta \in \big(0,\frac{1}{2}\big)$, if we have $\lim_{k \rightarrow \infty}{N_k} = 0$ almost surely, then Conjecture \ref{Conj: NAND 2D Regular Grid} is true, i.e. $\lim_{k \rightarrow \infty}{\|P_{X_k}^+ - P_{X_k}^-\|_{\mathsf{TV}}} = 0$. 
\end{lemma}

\begin{proof}
Define the event:
\begin{align*}
A & \triangleq \left\{\exists k \in \N, \, \forall j \in [k+1], \enspace Y_{k,j} \in \{0,1\}\right\} \\
& = \left\{\parbox[]{18em}{$\exists k \in \N$, there are no $u$'s in level $k$ of the coupled 2D regular grid}\right\} ,
\end{align*}
and observe that, as shown at the outset of the proof of Theorem \ref{Thm: Deterministic And Grid} in section \ref{Analysis of Deterministic And Grid}, we have:
\begin{align*}
\lim_{k \rightarrow\infty}{\left\|P_{X_k}^+ - P_{X_k}^-\right\|_{\mathsf{TV}}} & \leq 1 - \lim_{k \rightarrow \infty}{\P\!\left(X_k^+ = X_k^- \right)} \\
& = 1 - \P\!\left(A\right) 
\end{align*}
where we use \eqref{Eq: Coupled BSC 2}, \eqref{Eq: Coupled Nand}, and the continuity of probability measures. Therefore, similar to section \ref{Analysis of Deterministic And Grid}, it suffices to prove that if $\P(\lim_{k \rightarrow \infty}{N_k} = 0) = 1$, then $\P(A) = 1$. To establish this implication, we next show that $\big\{\!\lim_{k \rightarrow \infty}{N_k} = 0\big\} = A$.

Indeed, suppose $\lim_{k \rightarrow \infty}{N_k} = 0$. Then, there exists $K \in \N$ such that for all $k \in \N\backslash[K+1]$, $N_k \leq \frac{1}{2}$. However, since $N_m \in \N$ for all $m \in \N$, we must have $N_k = 0$ for all $k \in \N\backslash[K+1]$. This implies that $A$ occurs. Conversely, suppose $A$ occurs, i.e. there exists a level $k \in \N$ with no $u$'s. Then, $N_k = 0$, and due to \eqref{Eq: Coupled BSC 2} and \eqref{Eq: Coupled Nand}, $N_m = 0$ for all $m \geq k$. Hence, $\lim_{k \rightarrow \infty}{N_k} = 0$. This completes the proof.	
\end{proof}

Due to Lemma \ref{Lemma: Vanishing u's Condition}, we now focus our attention on establishing the almost sure convergence of $\{N_k : k \in\N\}$ to zero. A general approach to establishing such almost sure convergence of random variables is via \textit{Doob's martingale convergence theorem}, cf. \cite[Chapter V, Theorem 4.1]{Cinlar2011}. However, since we do not have a martingale in our problem, we need to construct one from our Markov chain $\{Y_k : k \in \N\}$. It is well-known in probability theory that ergodicity and related properties of discrete-time time-homogeneous Markov chains with countable state spaces can be analyzed using tools from the intimately related fields of martingale theory, Lyapunov theory, and potential theory (see e.g. \cite[Chapter 5]{Bremaud1999}). For example, \textit{Foster's theorem}, cf. \cite[Theorem 1.1, Chapter 5]{Bremaud1999}, characterizes the positive recurrence of Markov chains via the existence of \textit{Lyapunov functions} with certain properties. It is also closely related to martingale convergence based criteria for recurrence of Markov chains. Indeed, such martingale arguments typically construct a martingale from the Markov chain under consideration. The canonical approach of doing this is to apply a \textit{harmonic function}, which is an eigen-function of the Markov chain's conditional expectation operator with eigenvalue $1$, to the random variables defining the Markov chain\textemdash this is a specialization of the so called \textit{Dynkin martingale} (or \textit{L\'{e}vy's martingale}). Likewise, a supermartingale can be constructed by applying a superharmonic function to the Markov chain. The study of (super)harmonic functions in classical analysis is known as potential theory, and variants of harmonic functions turn out to be precisely the desired Lyapunov functions of Foster's theorem. 

The ensuing conjecture presents a particularly useful family of superharmonic functions with special properties that can be used to construct a supermartingale from $\{Y_k : k \in \N\}$.

\begin{conjecture}[Existence of Supermartingale]
\label{Conj: Martingale Construction}
Let $\Y^* = \cup_{k \in \N\backslash\!\{0\}}{\Y^k}$ denote the set of all non-empty finite-length strings with letters in $\Y$, and as in the proof of Theorem \ref{Thm: Deterministic And Grid} in section \ref{Analysis of Deterministic And Grid}, let $\{\F_k : k \in \N\}$ denote a filtration such that each $\F_k$ is the $\sigma$-algebra generated by the random variables $(Y_0,\dots,Y_k)$ and all the BSCs before level $k$ (where we include all events determining whether these BSCs are copies, and all events determining the independent bits they produce) for $k \in \N$. Then, there exists a parametrized family of Borel measurable \emph{superharmonic} functions $\big\{f_{\delta} : \Y^* \rightarrow \R \, \big| \, \delta \in \big(0,\frac{1}{2}\big) \big\}$ with parameter $\delta$ such that the discrete-time stochastic process $\{f_{\delta}(Y_{k}) : k \in \N\}$ satisfies the following properties for every $\delta \in \big(0,\frac{1}{2}\big)$:
\begin{enumerate}
\item $\{f_{\delta}(Y_{k}) : k \in \N\}$ is a \emph{supermartingale} adapted to $\{\F_k : k \in \N\}$, which implies that:
\begin{equation}
\label{Eq: Property 1}
\forall k \in \N, \enspace \E\!\left[f_{\delta}(Y_{k+1}) \middle| \F_k\right] = \E\!\left[f_{\delta}(Y_{k+1}) \middle| Y_k\right] \leq f_{\delta}(Y_{k})
\end{equation}
since $\{Y_k : k \in \N\}$ forms a Markov chain.
\item There exists a constant $C = C(\delta) > 0$ (which may depend on $\delta$) such that:\footnote{In fact, it would suffice to take $C = 1$ in \eqref{Eq: Property 2} (as used in our simulations).}
\begin{equation}
\label{Eq: Property 2}
\forall k \in \N, \enspace f_{\delta}(Y_{k}) \geq C N_k \enspace \text{almost surely} \, .
\end{equation}
\end{enumerate}
\end{conjecture}

The aforementioned properties should be reminiscent of the conditions in Theorem \ref{Thm: Sufficient Condition for NAND 2D Regular Grid}. If \textit{existence} of the supermartingale in Conjecture \ref{Conj: Martingale Construction} can be rigorously established, then the impossibility of broadcasting on 2D regular grids with NAND processing functions follows as a consequence. The next proposition justifies this implication.  

\begin{proposition}[Supermartingale Sufficient Condition]
\label{Prop: Martingale Condition}
If Conjecture \ref{Conj: Martingale Construction} is true, then Conjecture \ref{Conj: NAND 2D Regular Grid} is true. Equivalently, if $\delta \in \big(0,\frac{1}{2}\big)$, and there exists a Borel measurable superharmonic function $f_{\delta} : \Y^* \rightarrow \R$ such that $\{f_{\delta}(Y_{k}) : k \in \N\}$ is a supermartingale adapted to $\{\F_k : k \in \N\}$ and \eqref{Eq: Property 2} holds for some constant $C = C(\delta) > 0$, then broadcasting is impossible on the 2D regular grid with NAND processing functions in the sense of \eqref{Eq: Deterministic Impossibility of Reconstruction}. 
\end{proposition}

\begin{proof}
Fix any noise level $\delta \in \big(0,\frac{1}{2}\big)$, and consider the supermartingale $\{f_{\delta}(Y_{k}) : k \in \N\}$ (see property \eqref{Eq: Property 1}). Letting $X^- \triangleq -\min\{X,0\}$ denote the ``negative part'' of a random variable $X$, observe that:
$$ 0 = \sup_{k \in \N}{\E\!\left[f_{\delta}(Y_k)^- \right]} < +\infty $$
because \eqref{Eq: Property 2} implies that for all $k \in \N$, $f_{\delta}(Y_{k}) \geq 0$ almost surely since $N_k \geq 0$ almost surely. Hence, applying Doob's martingale convergence theorem, cf. \cite[Chapter V, Theorem 4.1]{Cinlar2011}, we have that the limiting random variable:
$$ F_{\delta} \triangleq \lim_{k \rightarrow \infty}{f_{\delta}(Y_k)} \geq 0 $$ 
exists (almost surely) and is integrable, i.e. $\E[F_{\delta}] < +\infty$. Define the limiting random variable:
$$ N_{\infty} \triangleq \limsup_{k \rightarrow \infty}{N_k} \geq 0 $$ 
(almost surely). Then, taking limits in \eqref{Eq: Property 2}, we have that:
$$ F_{\delta} \geq C N_{\infty} \enspace \text{almost surely} $$
which in turn produces:
$$ \E\!\left[F_{\delta}\right] \geq C \, \E\!\left[N_{\infty}\right] . $$
Thus, we get that $\E[N_{\infty}] < +\infty$. Now notice that it suffices to prove that $N_{\infty} = 0$ almost surely, since we can then employ Lemma \ref{Lemma: Vanishing u's Condition} to establish Conjecture \ref{Conj: NAND 2D Regular Grid}, and hence, this proposition.

We next prove that $N_{\infty} = 0$ almost surely. Observe that if $\E[N_{\infty}] = 0$, then $N_{\infty} = 0$ almost surely, because the expectation of a non-negative random variable is zero when the random variable itself is zero. Thus, we assume, for the sake of contradiction, that $\E[N_{\infty}] > 0$. Using Markov's inequality and the definition of superior limits, we obtain that for all $t > 0$:
\begin{align*}
& \P\!\left(\exists K = K(t) \in \N, \forall k \in \N\backslash[K+1], \, N_k \leq 2t\right) \\
& \qquad \qquad \qquad \qquad \qquad \qquad \qquad \qquad \geq \P\!\left(N_{\infty} \leq t\right) \\
& \qquad \qquad \qquad \qquad \qquad \qquad \qquad \qquad \geq 1 - \frac{\E\!\left[N_{\infty}\right]}{t} \, . 
\end{align*}
Fix any (small) $\epsilon \in (0,1)$, and let $t = \E[N_{\infty}]/\epsilon$. Then, the previous inequality can be written as:
\begin{equation}
\label{Eq: High Prob Event}
\begin{aligned}
& \P\!\left(\exists K = K(\epsilon) \in \N, \forall k \in \N\backslash[K+1], \, N_k \leq \frac{2 \, \E[N_{\infty}]}{\epsilon}\right) \\
& \qquad \qquad \qquad \qquad \qquad \qquad \qquad \qquad \qquad \qquad \geq 1 - \epsilon \, . 
\end{aligned}
\end{equation}
Now, for every $k \in \N\backslash\!\{0\}$, define the stopping time (with respect to the filtration $\{\F_k : k \in \N\}$):
\begin{align*}
T_k \triangleq \inf\bigg\{m \in \N : & \, \exists \, m_1 <\dots < m_{k} = m, \, \forall i \in \{1,\dots,k\}, \\
& \, N_{m_i} \leq \frac{2 \, \E[N_{\infty}]}{\epsilon}\bigg\} 
\end{align*}
which denotes the level index of the $k$th occasion when the number of $u$'s is bounded by $2 \E[N_{\infty}]/\epsilon$. Note that $T_k = +\infty$ for outcomes in the sample space where there do not exist $k$ levels $m_1 <\dots < m_{k}$ such that $N_{m_i} \leq 2 \E[N_{\infty}]/\epsilon$ for all $i \in \{1,\dots,k\}$. Moreover, for $i \in \{0,1\}$, define the corresponding sequences of events $\{A_{k,i} : k \in \N\backslash\!\{0\}\}$ as follows: 
$$ A_{k,i} \triangleq \left\{T_k = +\infty\right\} \cup \left\{\parbox[]{13em}{every $\mathsf{BSC}(\delta)$ on an outgoing edge starting at a vertex at level $T_k + i < +\infty$ that equals $u$ outputs an independent bit}\right\} $$
which contains all outcomes of the sample space where there are no $u$'s at level $T_k + i$. Observe that for every $k \in \N\backslash\!\{0\}$, $A_{k,0} \subseteq A_{k+1,0}$. Indeed, if $A_{k,0}$ occurs, then either $T_k = +\infty$ or there are no $u$'s at level $T_k + 1$ due to \eqref{Eq: Coupled BSC 2} and \eqref{Eq: Coupled Nand}. In the former case, we have $T_{k+1} = +\infty$, and in the latter case, there are no $u$'s at level $T_{k+1}$ due to \eqref{Eq: Coupled BSC 2} and \eqref{Eq: Coupled Nand}. So, $A_{k+1,0}$ must occur. 

It is straightforward to verify that for every $k \in \N\backslash\!\{0\}$:
\begin{align}
\P\!\left(A_{k+1,0} \middle| A_{k,0}^{c}\right) & \geq \P\!\left(A_{k+1,0} \cap \{T_{k+1} < +\infty\} \middle| A_{k,0}^{c}\right) \nonumber \\
& = \P\!\left(A_{k+1,0} \middle| \{T_{k+1} < +\infty\} \cap A_{k,0}^{c}\right) \nonumber \\
& \quad \, \cdot \P\!\left(T_{k+1} < +\infty \middle| A_{k,0}^{c}\right) \nonumber \\
& \geq (2\delta)^{4 \, \E[N_{\infty}]/\epsilon} \, \P\!\left(T_{k+1} < +\infty \middle| A_{k,0}^{c}\right) 
\label{Eq: Pre-Lower Bound on Probability for Counterpart BC}
\end{align}
where to obtain \eqref{Eq: Pre-Lower Bound on Probability for Counterpart BC}, we argue that if $A_{k,0}^{c}$ occurs (which means that $T_k$ is finite and the collection of outgoing BSCs from level $T_k$ output at least one $u$) and $T_{k+1}$ is finite, then $\P(A_{k+1,0} \, | \, \{T_{k+1} < +\infty\} \cap A_{k,0}^{c})$ is the probability that the BSCs on the pairs of outgoing edges of $u$'s at level $T_{k + 1}$ produce independent bits. Since there are at most $2 \, \E[N_{\infty}]/\epsilon$ $u$'s at level $T_{k + 1}$, and the outgoing BSCs from these $u$'s are conditionally independent given $A_{k,0}^c$ and $\{T_{k+1} < +\infty\}$, we obtain the bound in \eqref{Eq: Pre-Lower Bound on Probability for Counterpart BC}. Next, notice that for every $k \in \N\backslash\!\{0\}$, conditioned on $A_{k,0}^c$ happening, we have the inclusion relation $A_{k,1} \subseteq \{T_{k+1} < +\infty\}$. Indeed, if $A_{k,0}^c$ occurs, then $T_k$ must be finite. Suppose $A_{k,1}$ also occurs. Then, either there are no $u$'s at level $T_k + 1$, or there are $u$'s at level $T_k + 1$ whose outgoing BSCs generate independent bits. In both cases, there are no $u$'s at level $T_k + 2$ (due to \eqref{Eq: Coupled BSC 2} and \eqref{Eq: Coupled Nand}). Hence, $T_{k+1}$ must be finite. Using this conditional inclusion relation, continuing from \eqref{Eq: Pre-Lower Bound on Probability for Counterpart BC}, we have that for every $k \in \N\backslash\!\{0\}$:\footnote{As shown later in \eqref{Eq: Finiteness of Stopping Times}, it is easy to use \eqref{Eq: High Prob Event} to lower bound $\P(T_{k+1} < +\infty)$, but this does not immediately lower bound the term $\P(T_{k+1} < +\infty | A_{k,0}^c)$ in \eqref{Eq: Pre-Lower Bound on Probability for Counterpart BC}. So, we must apply a different approach here.}
\begin{align}
\P\!\left(A_{k+1,0} \middle| A_{k,0}^{c}\right) & \geq (2\delta)^{4 \, \E[N_{\infty}]/\epsilon} \, \P\!\left(A_{k,1} \middle| A_{k,0}^{c}\right) \nonumber \\
& \geq (2\delta)^{4 \, \E[N_{\infty}]/\epsilon} \, (2\delta)^{8 \, \E[N_{\infty}]/\epsilon} \nonumber \\
& = (2\delta)^{12 \, \E[N_{\infty}]/\epsilon} 
\label{Eq: Lower Bound on Probability for Counterpart BC}
\end{align}
where to obtain \eqref{Eq: Lower Bound on Probability for Counterpart BC}, we argue akin to \eqref{Eq: Pre-Lower Bound on Probability for Counterpart BC} that if $A_{k,0}^{c}$ occurs, then $\P(A_{k,1} | A_{k,0}^{c})$ is the probability that the BSCs on the pairs of outgoing edges of $u$'s at level $T_{k} + 1$ produce independent bits. Since there are at most $4 \, \E[N_{\infty}] / \epsilon$ $u$'s at level $T_k + 1$ due to \eqref{Eq: Coupled BSC 2} and \eqref{Eq: Coupled Nand} (because there are at most $2 \, \E[N_{\infty}] / \epsilon$ $u$'s at level $T_k$ that can propagate to the next level via two outgoing edges each), and the outgoing BSCs from these $u$'s are conditionally independent given $A_{k,0}^c$, we obtain the bound in \eqref{Eq: Lower Bound on Probability for Counterpart BC}.

To conclude the proof, we will employ the ``counterpart of the Borel-Cantelli lemma,'' cf. \cite[Lemma 1]{Bruss1980}. To this end, notice using \eqref{Eq: Lower Bound on Probability for Counterpart BC} that:
$$ \sum_{k = 1}^{\infty}{\P\!\left(A_{k+1,0}\middle|A_{k,0}^c\right)} \geq  \sum_{k = 1}^{\infty}{(2\delta)^{12 \, \E[N_{\infty}]/\epsilon}} = +\infty \, . $$
Then, since $\{ A_{k,0} : k \in \N\backslash\!\{0\}\}$ is a non-decreasing sequence of sets, the counterpart of the Borel-Cantelli lemma yields:
\begin{equation}
\label{Eq: A_k i.o.}
\P\!\left(\bigcup_{k = 1}^{\infty}{A_{k,0}}\right) = 1 \, . 
\end{equation}
(Note that if there exists some $k \in \N\backslash\!\{0\}$ such that $\P(A_{k,0}^c) = 0$ and $\P(A_{k+1,0}|A_{k,0}^c)$ is not well-defined, then $\P(A_{k,0}) = 1$, and the above equality trivially holds.) On the other hand, \eqref{Eq: High Prob Event} yields the following inequality:
\begin{align}
& \P\!\left(\forall k \in \N\backslash\!\{0\}, \, T_k < +\infty \right) \nonumber \\
& \geq \P\!\left(\exists K = K(\epsilon) \in \N, \forall k \in \N\backslash[K+1], \, N_k \leq \frac{2 \, \E[N_{\infty}]}{\epsilon}\right) \nonumber \\
& \geq 1 - \epsilon \, . 
\label{Eq: Finiteness of Stopping Times}
\end{align}
Finally, using both \eqref{Eq: A_k i.o.} and \eqref{Eq: Finiteness of Stopping Times}, we get that:
\begin{align*}
\P\!\left(N_{\infty} = 0\right) & \geq \P\!\left(\left(\bigcup_{k = 1}^{\infty}{A_{k,0}}\right)\!\cap\!\left(\bigcap_{k = 1}^{\infty}{\{T_k < +\infty\}}\right)\right) \\
& \geq 1 - \epsilon 
\end{align*}
where the first inequality holds because if there exists a finite level $T_k$ such that all outgoing BSCs from this level do not output $u$'s, then we must have $N_{T_k + 1} = 0$, and therefore, $N_{\infty} = 0$ (due to \eqref{Eq: Coupled BSC 2} and \eqref{Eq: Coupled Nand}). Letting $\epsilon \rightarrow 0$ produces:
$$ \P\!\left(N_{\infty} = 0\right) = 1 \, , $$
which implies that $\E[N_{\infty}] = 0$. This contradicts the assumption that $\E[N_{\infty}] > 0$. Hence, $N_{\infty} = 0$ almost surely, which completes the proof. 
\end{proof}

Although Proposition \ref{Prop: Martingale Condition} shows that it suffices to prove Conjecture \ref{Conj: Martingale Construction} in order to establish the impossibility of broadcasting on 2D regular grids with NAND processing functions, constructing the supermartingales in Conjecture \ref{Conj: Martingale Construction} turns out to be nontrivial. Therefore, in the ensuing subsections, we propose an approach for constructing the structured supermartingales outlined in Conjecture \ref{Conj: Martingale Construction}, and bolster this proposal by demonstrating how to compute them efficiently.

\subsection{Counting Forms and Potential Functions}
\label{Potential Functions}

In this subsection, we will explain an extended version of the potential function technique in \cite{HolroydMarcoviciMartin2019} that offers an avenue to construct the superharmonic functions delineated in Conjecture \ref{Conj: Martingale Construction}. In particular, we will consider superharmonic potential functions composed of so called ``counting forms.'' To this end, we begin with a simple motivating example. Consider the potential function $w:\Y^* \rightarrow \R$ given by:
\begin{equation}
\begin{aligned}
\forall y \in \Y^*, \enspace w(y) & = \big[\text{number of } (u) \text{'s in } y\big] \\
& \quad \,\, - \big[\text{number of } (01u) \text{'s in } y\big] \, .
\end{aligned}
\end{equation}
This function counts the number of $u$'s that appear in an input string $y$ and subtracts the number of times the sub-string $(01u)$ appears in $y$. So, for example, if the input string is $y = (0010uu01101u0)$, then $w(y) = 3 - 1 = 2$. Since computing $w$ requires a count of sub-strings of length (at most) $3$, $w$ can only be meaningfully evaluated for strings of length at least $3$, and we say that $w$ has ``rank'' $3$. Potential functions obtained by such linear combinations of counts will be the subject of much of our discussion from hereon. The ensuing definition formally generalizes the aforementioned example.

\begin{definition}[Basis Clauses and Counting Forms]
\label{Def: Counting Form}
For every finite string $v \in \Y^*$ with length $r \in \N\backslash\!\{0\}$, we define an associated \emph{basis clause} $\{v\} : \Y^* \rightarrow \R$ via:
\begin{align*}
\forall k \in & \, \N\backslash\!\{0\},\, \forall y = (y_1 \cdots y_k) \in \Y^k , \\
& \{v\}(y) \triangleq 
\begin{cases}
\displaystyle{\sum_{i = 1}^{k-r+1}{\I\!\left\{(y_i \cdots y_{i + r - 1}) = v\right\}}} \, , & r \leq k \\
0 \, , & r > k 
\end{cases} 
\end{align*}
which is a map that counts the number of sub-strings $v$ that can be found in the input string $y$. (Note that when explicitly writing out the letters of a string $v$, we will often use parentheses for clarity, e.g. $v = (01u) \in \Y^3$. On the other hand, curly braces will be used to distinguish the string $v$ from its associated basis clause $\{v\}$, e.g. $\{01u\}$ denotes the basis clause associated with the string $(01u)$. Moreover, throughout this paper, we will identify a basis clause $\{v\}$ with the string $v$.) The length of such a basis clause $\{v\}$ is called its \emph{rank}, which we denote as $\rank(\{v\}) = r$. For any finite set of basis clauses $\{v_1\},\dots,\{v_m\} : \Y^* \rightarrow \R$ and coefficients $\alpha_1,\dots,\alpha_m \in \R$ (with $m \in \N\backslash\!\{0\}$), the formal sum $w = \alpha_1 \{v_1\} + \cdots + \alpha_m \{v_m\}$ is said to be a \emph{counting form}, and $w : \Y^* \rightarrow \R$ is a \emph{potential function} defined by:
$$ \forall y \in \Y^* , \enspace w(y) \triangleq \sum_{j = 1}^{m}{\alpha_j \{v_j\}(y)} \, . $$
(Note that we omit the curly braces notation for general counting forms since there is no cause for confusion between strings and counting forms.) The rank of a counting form $w$ is defined as the maximal rank of its basis clauses:
$$ \rank(w) \triangleq \max_{\substack{i \in \{1,\dots,m\}: \\ \alpha_i \neq 0}}{\rank(\{v_i\})} \, , $$
and we say that $w$ has \emph{pure rank} if $\rank(w) = \rank(\{v_i\})$ for every $i \in \{1,\dots,m\}$ such that $\alpha_i \neq 0$. Furthermore, we define a counting form $w$ to be \emph{$u$-only} if all the strings $v_1,\dots,v_m$ defining its basis clauses contain a $u$, i.e. $w$ is not $u$-only if:
$$ \exists \, i \in \{1,\dots,m\}, \enspace v_i \in \{0,1\}^{\rank(\{v_i\})} \, . $$ 
\end{definition} 

As another concrete illustration of the concepts in Definition \ref{Def: Counting Form}, consider the $u$-only counting form:
\begin{equation}
\label{Eq: Holroyd potential}
w_* \triangleq 2 \{u\} + \{u 1\} + \{1 u\} + \{u 1 0\} + \{0 1 u\} - 2\{0 u 0\} \, .
\end{equation}
This counting form has $\rank(w_*) = 3$, and basis clauses $\{u\}$, $\{u 1\}$, $\{1 u\}$, $\{u 1 0\}$, $\{0 1 u\}$, $\{0 u 0\}$ with coefficients $2$, $1$, $1$, $1$, $1$, $-2$, respectively. Moreover, the potential of the string $y = (0 1 u u 1 u 0 0 1 u 1)$ is $w_*(y) = 2 (4) + (2) + (3) + (0) + (2) - 2(0) = 15$.

As mentioned earlier, we will restrict our search for superharmonic functions satisfying conditions \eqref{Eq: Property 1} and \eqref{Eq: Property 2} in Conjecture \ref{Conj: Martingale Construction} to the class of $u$-only potential functions defined above. This restriction to $u$-only counting forms is intuitively sound, because our goal is to show that \eqref{Eq: Property 2} is satisfied and the number of $u$'s per layer of the coupled NAND 2D regular grid vanishes almost surely as the depth tends to infinity\textemdash see Lemma \ref{Lemma: Vanishing u's Condition}. Let us fix a rank $r \in \N\backslash\!\{0\}$, and only consider $u$-only counting forms with at most this rank. Moreover, define the set of finite strings $\Y_r^* \subset \cup_{k \geq r}{\Y^k}$ such that for any $y \in \Y_r^*$, the first and last $r$ letters of $y$ do not contain $u$'s, or equivalently, $y = (y_1 \cdots y_k) \in \Y_r^*$ with $k \geq r$ if and only if $y_1,\dots,y_r,y_{k-r+1},\dots,y_k \in \{0,1\}$. Since the BSCs in the coupled NAND 2D regular grid are independent, a simple second Borel-Cantelli lemma argument can be used to show that with probability $1$, there exists a level $t \in \N$ such that $Y_{t,0},\dots,Y_{t,r-1},Y_{t,t-r+1},\dots,Y_{t,t} \in \{0,1\}$ in the coupled 2D regular grid. (This is because $Y_{t,0},\dots,Y_{t,r-1},Y_{t,t-r+1},\dots,Y_{t,t} \in \{0,1\}$ if the BSCs above these vertices generate independent bits.) Define the event:
\begin{equation}
\label{Eq: The special event}
\begin{aligned}
A \triangleq \big\{ \exists \, t \in \N, \, \forall \, l \geq t, & \enspace Y_{l,0},\dots,Y_{l,r-1}, \\
& \enspace Y_{l,l-r+1},\dots,Y_{l,l} \in \{0,1\}\big\} \, .
\end{aligned}
\end{equation}
Then, due to \eqref{Eq: Coupled BSC 2} and \eqref{Eq: Coupled Nand}, we have that $\P(A) = 1$, i.e. $A$ occurs almost surely. For convenience in our ensuing exposition, we will ``condition the event $A$'' and make the simplifying assumption that our $u$-only counting forms are potential functions with domain $\Y_r^*$ (instead of $\Y^*$). We will rigorize what we mean by ``conditioning on $A$'' during the proof of Theorem \ref{Thm: Sufficient Condition for NAND 2D Regular Grid} in subsection \ref{Proof of Theorem Sufficient Condition for NAND 2D Regular Grid}.

Therefore, we seek a $u$-only counting form $w_{\delta} = \alpha_1(\delta) \{v_1\} + \cdots + \alpha_m(\delta) \{v_m\}$ with $m \in \N\backslash\!\{0\}$ basis clauses $\{v_1\},\dots,\{v_m\}$ which contain $u$'s, and corresponding $\delta$-dependent coefficients $\alpha_1(\delta),\dots,\alpha_m(\delta) \in \R\backslash\!\{0\}$, respectively, such that for all $\delta \in \big(0,\frac{1}{2}\big)$, the potential function $w_{\delta} : \Y_r^* \rightarrow \R$ satisfies:
\begin{enumerate}
\item The supermartingale conditions:
\begin{equation}
\label{Eq: Prop 1}
\forall k \in \N, \enspace \E\!\left[\left|w_{\delta}(Y_{k})\right|\right] < +\infty \, , 
\end{equation}
\begin{equation}
\label{Eq: Prop 2}
\begin{aligned}
\forall k \geq r-1, & \, \forall y \in \Y^{k+1} \cap \Y_r^*, \\
& \E\!\left[w_{\delta}(Y_{k+1}) \middle| Y_k = y\right] \leq w_{\delta}(y) \, .
\end{aligned}
\end{equation}
\item There exists a constant $C = C(\delta) > 0$ such that:
\begin{equation}
\label{Eq: Prop 3}
\forall y \in \Y_r^*, \enspace w_{\delta}(y) \geq C \, \{u\}(y) 
\end{equation}
where $\{u\}$ is a basis clause with rank $1$.
\end{enumerate}
The condition \eqref{Eq: Prop 1} holds trivially, because for all $k \in \N$:
\begin{align}
\E\!\left[|w_{\delta}(Y_{k})|\right] & \leq \sum_{i = 1}^{m}{|\alpha_i(\delta)| \, \E\!\left[|\{v_i\}(Y_{k})|\right]} \nonumber \\
& \leq (k + 1)\sum_{i = 1}^{m}{|\alpha_i(\delta)|} < +\infty
\end{align}
using the triangle inequality and the fact that $0 \leq \{v_i\}(y) \leq k+1$ for all $i \in \{1,\dots,m\}$ and all $y \in \Y^{k+1}$. Moreover, it is intuitively clear that akin to the supermartingale depicted in \eqref{Eq: Property 1}, conditions \eqref{Eq: Prop 1} and \eqref{Eq: Prop 2} essentially define a supermartingale $\{w_{\delta}(Y_{k}) : k \geq r-1\}$ ``conditioned on $A$.'' Likewise, \eqref{Eq: Prop 3} is essentially equivalently to \eqref{Eq: Property 2} ``conditioned on $A$.'' Hence, by extending the argument in Proposition \ref{Prop: Martingale Condition}, the impossibility of broadcasting on the NAND 2D regular grid can be established from conditions \eqref{Eq: Prop 1}, \eqref{Eq: Prop 2}, and \eqref{Eq: Prop 3} (as we will explain in subsection \ref{Proof of Theorem Sufficient Condition for NAND 2D Regular Grid}).  

The discussion heretofore reveals that to prove the impossibility of broadcasting on the NAND 2D regular grid, it suffices to prove the \emph{existence} of a $u$-only counting form $w_{\delta}$ (with rank at most $r$) such that \eqref{Eq: Prop 2} and \eqref{Eq: Prop 3} are satisfied. Naturally, one approach towards proving the existence of such a $w_{\delta}$ is to construct an explicit example. In the remainder of this section, we will make some partial progress on this important problem of constructing an example. 

Both the constraints \eqref{Eq: Prop 2} and \eqref{Eq: Prop 3} require us to verify inequalities between certain potentials of strings in $\Y_r^*$. To state these constraints directly in terms of counting forms, we will transform them using the next two definitions. The first of these is an equivalent linear operator on counting forms that captures the action of the conditional expectation operators in \eqref{Eq: Prop 2}.

\begin{definition}[Conditional Expectation Operator]
\label{Def: Conditional Expectation Operator}
For any input basis clause $\{v\}$ with rank $s \in \N\backslash\!\{0\}$, the \emph{conditional expectation operator} $\CE$ outputs the following counting form which has pure rank $s+1$:
$$ \CE(\{v\}) \triangleq \sum_{z \in \Y^{s+1}}{\P\!\left((Y_{s+1,1},\dots,Y_{s+1,s}) = v \, \middle| \, Y_{s} = z\right) \{z\}} $$
where the coefficients $\big\{\P((Y_{s+1,1},\dots,Y_{s+1,s}) = v \, | \, Y_{s} = z) : z \in \Y^{s+1}\big\}$ in the above formal sum of basis clauses are given by the Markov transition probabilities of the coupled NAND 2D regular grid (see \eqref{Eq: Coupled BSC 2}  and \eqref{Eq: Coupled Nand}). Furthermore, for any input counting form $w = \alpha_1 \{v_1\} + \cdots + \alpha_m \{v_m\}$ with basis clauses $\{v_1\},\dots,\{v_m\}$ and coefficients $\alpha_1,\dots,\alpha_m \in \R\backslash\!\{0\}$ (with $m \in \N\backslash\!\{0\}$), $\CE$ outputs the following counting form:
$$ \CE(w) \triangleq \alpha_1 \CE(\{v_1\}) + \cdots + \alpha_m \CE(\{v_m\}) $$
which has $\rank(\CE(w)) = \rank(w) + 1$.
\end{definition}

We note that the set of all counting forms is a vector space over $\R$, and the sets of all $u$-only counting forms and all counting forms with rank at most $r$ are linear subspaces of this larger vector space. Specifically, the intersection of these two linear subspaces, namely, the set of all $u$-only counting forms with rank at most $r$, is also a linear subspace. In particular, for any real scalar $\gamma \in \R$ and any pair of counting forms $w_1 = \alpha_1 \{v_1\} + \cdots + \alpha_m \{v_m\}$ and $w_2 = \beta_1 \{v_1\} + \cdots + \beta_m \{v_m\}$, which have a common set of basis clauses $\{v_1\},\dots,\{v_m\}$ without loss of generality, and coefficients $\alpha_1,\dots,\alpha_m$ and $\beta_1,\dots,\beta_m$, respectively, we define vector addition by $w_1 + w_2 = (\alpha_1 + \beta_1) \{v_1\} + \cdots + (\alpha_m + \beta_m) \{v_m\}$ and scalar multiplication by $\gamma w_1 = \gamma \alpha_1 \{v_1\} + \cdots + \gamma \alpha_m \{v_m\}$. It is straightforward to verify that $\rank(w_1 + w_2) \leq \max\{\rank(w_1),\rank(w_2)\}$ and $\rank(\gamma w_1) = \rank(w_1)$ (for $\gamma \neq 0$). Moreover, we will use the notation $w_1 - w_2$ to mean $w_1 - w_2 = w_1 + (-1)w_2$ in the sequel.

In Definition \ref{Def: Conditional Expectation Operator}, if the basis clause $\{v\}$ contains $u$'s, then the coefficient $\P((Y_{s+1,1},\dots,Y_{s+1,s}) = v \, | \, Y_{s} = z) = 0$ when $z \in \{0,1\}^{s+1}$ (due to \eqref{Eq: Coupled BSC 2}  and \eqref{Eq: Coupled Nand}). Thus, the conditional expectation operator $\CE$ maps the vector space of $u$-only counting forms with rank at most $r$ to the vector space of $u$-only counting forms with rank at most $r+1$. The ensuing proposition illustrates that the action of the conditional expectations in \eqref{Eq: Prop 2} is equivalent to the action of $\CE$ in Definition \ref{Def: Conditional Expectation Operator}.

\begin{proposition}[Equivalence of $\CE$]
\label{Prop: Equivalence of CE}
Consider any $u$-only counting form $w_{\delta}$ with rank at most $r$. Then, for all $k \geq r-1$ and for every $y \in \Y^{k+1} \cap \Y_r^*$, we have: 
$$ \big(\CE(w_{\delta})\big)(y) = \E\!\left[w_{\delta}(Y_{k+1}) \middle| Y_k = y\right] . $$
\end{proposition}

\begin{proof}
First, consider any basis clause $\{v\}$ that contains $u$'s and has $\rank(\{v\}) = s \leq r$, and fix any $k \geq r-1$ and any $y = (y_0 \cdots y_k) \in \Y^{k+1} \cap \Y_r^*$. Then, observe that:
\begin{align*}
& \E\!\left[\{v\}(Y_{k+1})\middle| Y_k = y\right] \\
& = \sum_{i = 0}^{k-s+2}{\P\!\left((Y_{k+1,i},\dots,Y_{k+1,i + s-1}) = v \, \middle| \, Y_k = y\right)} \\
& = \sum_{i = 1}^{k-s+1}{\begin{aligned} 
\P((Y_{k+1,i},\dots,Y_{k+1,i + s-1}) = v \, | \, Y_{k,i-1} & = y_{i-1},\\
\dots,Y_{k,i + s-1} & = y_{i + s-1}) 
\end{aligned}} \\
& = \sum_{i = 1}^{k-s+1}{\P\!\left((Y_{s+1,1},\dots,Y_{s+1,s}) = v \middle| Y_s = (y_{i-1} \cdots y_{i + s-1})\right)} \\
& = \sum_{i = 1}^{k-s+1}{\sum_{z \in \Y^{s+1}}{\begin{aligned}
& \P\!\left((Y_{s+1,1},\dots,Y_{s+1,s}) = v \, \middle| \, Y_s = z\right) \\
& \cdot \I\!\left\{(y_{i-1} \cdots y_{i + s-1}) = z\right\} 
\end{aligned}}} \\
& = \sum_{z \in \Y^{s+1}}{\P\!\left((Y_{s+1,1},\dots,Y_{s+1,s}) = v \, \middle| \, Y_s = z\right) \{z\}(y)} \\
& = \big(\CE(\{v\})\big)(y) 
\end{align*}
where the first equality follows from Definition \ref{Def: Counting Form}, the second equality holds because the coupled 2D grid variables $Y_{k+1,i},\dots,Y_{k+1,i + s-1}$ only depend on the variables $Y_{k,i-1},\dots,Y_{k,i + s-1}$ in the previous layer, and because $\P((Y_{k+1,0},\dots,Y_{k+1,s-1}) = v \, | \, Y_{k,0} = y_{0},\dots,\allowbreak Y_{k,s-1} = y_{s-1}) = 0$ and $\P((Y_{k+1,k-s+2},\dots,Y_{k+1,k+1}) = v \, | \, Y_{k,k-s+1} = y_{k-s+1},\dots,Y_{k,k} = y_{k}) = 0$ since $v$ contains $u$'s and $y \in \Y_r^*$, the third equality holds because \eqref{Eq: Coupled BSC 2} and \eqref{Eq: Coupled Nand} determine the Markov transition kernels between all pairs of consecutive layers in the coupled NAND 2D regular grid, the fifth equality follows from Definition \ref{Def: Counting Form} (where we treat the string $z$ as a basis clause $\{z\}$), and the final equality follows from Definitions \ref{Def: Counting Form} and \ref{Def: Conditional Expectation Operator}.

This implies that for all $k \geq r-1$ and for every $y = (y_0 \cdots y_k) \in \Y^{k+1} \cap \Y_r^*$, we also have: 
\begin{align*}
\E\!\left[w_{\delta}(Y_{k+1})\middle| Y_k = y\right] & = \sum_{i = 1}^{m}{\alpha_i(\delta) \, \E\!\left[\{v_i\}(Y_{k+1})\middle| Y_k = y\right]} \\
& = \sum_{i = 1}^{m}{\alpha_i(\delta) \big(\CE(\{v_i\})\big)(y)} \\
& = \big(\CE(w_{\delta})\big)(y)
\end{align*}
where we use Definitions \ref{Def: Counting Form} and \ref{Def: Conditional Expectation Operator}, and we let $w_{\delta} = \alpha_1(\delta) \{v_1\} + \cdots + \alpha_m(\delta) \{v_m\}$ (without loss of generality) with $m \in \N\backslash\!\{0\}$ basis clauses $\{v_1\},\dots,\{v_m\}$ which contain $u$'s, and corresponding $\delta$-dependent coefficients $\alpha_1(\delta),\dots,\alpha_m(\delta) \in \R$, respectively. This completes the proof.
\end{proof}

The second definition we require to transform the constraints \eqref{Eq: Prop 2} and \eqref{Eq: Prop 3} into inequalities in terms of counting forms is that of a partial order over counting forms.

\begin{definition}[Equivalence Classes and Partial Order]
\label{Def: Partial Order}
Let $\0$ denote the zero counting form with no basis clauses and $\rank(\0) = 0$ (with abuse of notation), which outputs $0$ on every input string. Define the \emph{equivalence relation} $\simeq$ over $u$-only counting forms as follows. For any pair of $u$-only counting forms $w_1$ and $w_2$ with $\max\{\rank(w_1),\rank(w_2)\} \leq r$, $w_1 \simeq w_2$ if we have:
$$ \forall y \in \Y_r^*, \enspace w_1(y) = w_2(y) \, . $$
Due to this equivalence relation, the aforementioned vector space of $u$-only counting forms with rank at most $r$ is actually a vector space of equivalence classes of such counting forms, where any specific $u$-only counting form $w$ with rank at most $r$ corresponds to the equivalence class $\{w + z : z \text{ is } u \text{-only}, \, \rank(z) \leq r, \, z \simeq \0\}$. Furthermore, we define the \emph{partial order} $\succeq$ over $u$-only counting forms as follows. For any pair of $u$-only counting forms $w_1$ and $w_2$ with $\max\{\rank(w_1),\rank(w_2)\} \leq r$, $w_1 \succeq w_2$ if we have:
$$ \forall y \in \Y_r^*, \enspace w_1(y) \geq w_2(y) \, . $$
This makes the set of equivalence classes of $u$-only counting forms with rank at most $r$ a partially ordered vector space. Lastly, we say that a $u$-only counting form $w$ with rank at most $r$ is \emph{non-negative} if $w \succeq \0$, and we will often write $w_1 \succeq w_2$ equivalently as $w_1 - w_2 \succeq \0$.
\end{definition}

We remark that Definition \ref{Def: Partial Order} illustrates that the rank of a $u$-only counting form is a property of the counting form itself, and not the corresponding equivalence class. For example, when $r \geq 2$, the basis clause $\{u\}$ has unit rank, and is equivalent to the following $u$-only counting forms with pure rank $2$:
\begin{align}
\{u\} & \simeq \{0u\} + \{1u\} + \{uu\} \, , \\
\{u\} & \simeq \{u0\} + \{u1\} + \{uu\} \, .
\end{align}
In fact, given a $u$-only counting form $w_1$ with rank $s \leq r$, these equations exemplify a more general and rather useful trick for constructing an equivalent $u$-only counting form $w_2$ with pure rank $t$ such that $w_1 \simeq w_2$ and $s \leq t \leq r$. The next lemma formally presents this trick.

\begin{lemma}[Purification of Counting Forms]
\label{Lemma: Purification of Counting Forms}
Consider any $u$-only counting form $w_1 = \alpha_1 \{v_1\} + \cdots + \alpha_m \{v_m\}$ with basis clauses $\{v_1\},\dots,\{v_m\}$ and coefficients $\alpha_1,\dots,\alpha_m \in \R\backslash\!\{0\}$ for some $m \in \N\backslash\!\{0\}$, and rank $s \leq r$. For any integer $t \in [s,r]$, construct the $u$-only counting forms:
\begin{align}
\label{Eq: Pure 1}
w_2 & = \sum_{i = 1}^{m}{\alpha_i \sum_{z \in \Y^{t-\rank(\{v_i\})}}{\{v_i,z\}}} \, , \\
w_3 & = \sum_{i = 1}^{m}{\alpha_i \sum_{z \in \Y^{t-\rank(\{v_i\})}}{\{z,v_i\}}} \, , 
\label{Eq: Pure 2} 
\end{align}
where $\{v_i,z\} , \{z,v_i\} \in \Y^{t}$ denote the basis clauses obtained by concatenating the strings $v_i \in \Y^{\rank(\{v_i\})}$ and $z \in \Y^{t-\rank(\{v_i\})}$, and the inner summations equal $\{v_i\}$ when $\rank(\{v_i\}) = t$. Then, $w_1 \simeq w_2 \simeq w_3$, and $w_2,w_3$ have pure rank $t$.
\end{lemma}

\begin{proof}
It is obvious that $w_2,w_3$ have pure rank $t$. Moreover, the proof of $w_1 \simeq w_3$ is the same as the proof of $w_1 \simeq w_2$ by symmetry. So, we only need to establish that $w_1 \simeq w_2$. To prove this, it suffices to show that for any basis clause $\{v_i\}$ of $w_1$ with $\rank(\{v_i\}) < t$:
\begin{equation}
\label{Eq: Clause Purify}
\{v_i\} \simeq \sum_{z \in \Y^{t-\rank(\{v_i\})}}{\{v_i,z\}} \, .
\end{equation}
To this end, consider any basis clause $\{v_i\}$ with $s^{\prime} = \rank(\{v_i\}) < t$. Then, for any $k \geq r$ and any string $y = (y_1 \cdots y_k) \in \Y^{k} \cap \Y_r^*$, we have: 
\begin{align*}
& \sum_{z \in \Y^{t-s^{\prime}}}{\{v_i,z\}(y)} \\
& \qquad = \sum_{z \in \Y^{t-s^{\prime}}}{\sum_{j = 1}^{k - t + 1}{\I\!\left\{(y_j\cdots y_{j+t-1}) = (v_i,z)\right\}}} \\
& \qquad = \sum_{j = 1}^{k - t + 1} \I\!\left\{(y_j \cdots y_{j+s^{\prime}-1}) = v_i\right\} \cdot \\
& \qquad \qquad \qquad \underbrace{\sum_{z \in \Y^{t-s^{\prime}}}{\I\!\left\{(y_{j+s^{\prime}} \cdots y_{j+t-1}) = z\right\}}}_{= \, 1} \\
& \qquad = \sum_{j = 1}^{k - s^{\prime} + 1}{\I\!\left\{(y_j \cdots y_{j+s^{\prime}-1}) = v_i\right\}} \\
& \qquad \quad \, - \underbrace{\sum_{j = k - t + 2}^{k - s^{\prime} + 1}{\I\!\left\{(y_j \cdots y_{j+s^{\prime}-1}) = v_i\right\}}}_{= \, 0} \\
& \qquad = \{v_i\}(y) 
\end{align*}
where $(v_i,z)$ denotes the concatenation of the strings $v_i$ and $z$, the first and final equalities follow from Definition \ref{Def: Counting Form}, the second equality follows from swapping the order of summations, and the second summation in the third equality equals zero, because it only depends on the sub-string $(y_{k - t + 2} \cdots y_k)$ with length $k - (k - t + 2) + 1 = t - 1 \leq r$, and $(y_{k - t + 2} \cdots y_k)$ contains no $u$'s since $y \in \Y_r^*$ while $v_i$ contains $u$'s since $w_1$ is $u$-only. This establishes \eqref{Eq: Clause Purify}, and therefore, completes the proof.
\end{proof}

Returning to our main discussion, with Definitions \ref{Def: Conditional Expectation Operator} and \ref{Def: Partial Order} in place, we can recast the constraints \eqref{Eq: Prop 2} and \eqref{Eq: Prop 3} as follows. We seek to construct a $u$-only counting form $w_{\delta}$ with rank at most $r$ such that for all $\delta \in \big(0,\frac{1}{2}\big)$:
\begin{align}
\label{Eq: Form Cond 1}
w_{\delta} - \CE(w_{\delta}) & \succeq \0 \\
\exists \, C = C(\delta) > 0, \enspace w_{\delta} - C \{u\} & \succeq \0
\label{Eq: Form Cond 2}
\end{align}
where \eqref{Eq: Form Cond 1} is equivalent to \eqref{Eq: Prop 2} using Proposition \ref{Prop: Equivalence of CE}, and \eqref{Eq: Form Cond 2} is equivalent to \eqref{Eq: Prop 3}. In order to verify \eqref{Eq: Form Cond 1} and \eqref{Eq: Form Cond 2} using a computer program, we need to choose an appropriate value of $r$, and develop an efficient algorithm to test the partial order $\succeq$. We conclude this subsection by doing the former, and leave the development of the latter for ensuing subsections. 

Observe using Definition \ref{Def: Conditional Expectation Operator} that the conditional expectation operator $\CE$ evidently depends on $\delta$ (because the coefficients of the counting form $\CE(\{v\})$ for a basis clause $\{v\}$ depend on the transition probabilities of the coupled NAND 2D regular grid). Thus, it is instructive to consider the noiseless setting where $\delta = 0$. In this case, \eqref{Eq: Form Cond 1} is clearly satisfied if we find a $u$-only counting form $w$ that is a fixed point of $\CE$, i.e. $w \simeq \CE(w)$. Inspired by the ``weight function'' used in \cite[Equation (2.3)]{HolroydMarcoviciMartin2019}, we consider the $u$-only \emph{harmonic counting form} $w_*$ defined in \eqref{Eq: Holroyd potential} (which is an adjusted and symmetrized version of \cite[Equation (2.3)]{HolroydMarcoviciMartin2019}). The next proposition establishes that $w_*$ is a fixed point of $\CE$ when $\delta = 0$.

\begin{proposition}[Harmonic Counting Form]
\label{Prop: Harmonic Counting Form} 
If $\delta = 0$ and $r \geq 3$, then the harmonic counting form $w_*$ in \eqref{Eq: Holroyd potential} satisfies:
$$ w_* \simeq \CE(w_*) \, . $$
\end{proposition}

\begin{proof}
Using \eqref{Eq: Holroyd potential} and Definition \ref{Def: Conditional Expectation Operator}, we have:
\begin{align*}
\CE(w_*) & = 2 \CE(\{u\}) + \CE(\{u 1\}) + \CE(\{1 u\}) + \CE(\{u 1 0\}) \\
& \quad \, + \CE(\{0 1 u\}) - 2\CE(\{0 u 0\}) \\
& = 2 (\{u u\} + \{1 u\} + \{u 1\}) \\
& \quad \, + (\{u 1 0\} + \{u u 0\} + \{1 u 0\}) \\
& \quad \, + (\{0 1 u\} + \{0 u u\} + \{0 u 1\}) + \0 + \0 - 2(\0) \\
& \simeq (\underbrace{\{u u\} + \{1 u\} + \{0 u\}}_{\simeq \, \{u\}}) - \{0 u\} \\
& \quad \, + (\underbrace{\{u u\} + \{u 1\} + \{u 0\}}_{\simeq \, \{u\}}) - \{u 0\} + \{1 u\} + \{u 1\} \\
& \quad \, + \{u 1 0\} + (\underbrace{\{u u 0\} + \{1 u 0\} + \{0 u 0\}}_{\simeq \, \{u 0\}}) + \{0 1 u\} \\
& \quad \, + (\underbrace{\{0 u u\} + \{0 u 1\} + \{0 u 0\}}_{\simeq \, \{0 u\}}) - 2\{0 u 0\} \\
& \simeq 2 \{u\} + \{u 1\} + \{1 u\} + \{u 1 0\} + \{0 1 u\} - 2\{0 u 0\} \\
& = w_*
\end{align*}
where to obtain the second equality, for each basis clause $\{v_1\}$ of $w_*$, we find all basis clauses $\{v_2\}$ with $\rank(\{v_2\}) = \rank(\{v_1\}) + 1$ such that consecutive noiseless NAND gates map $v_2$ to $v_1$ (e.g. $(u 1 0)$, $(u u 0)$, and $(1 u 0)$ form the set of strings that are mapped to $(u 1)$ via \eqref{Eq: Coupled Nand}), and when there are no such $\{v_2\}$'s, $\CE$ outputs the zero counting form $\0$, and to obtain the third equivalence, we utilize Lemma \ref{Lemma: Purification of Counting Forms}. This completes the proof.
\end{proof} 

Propelled by the elegant fixed point property of $w_*$ in Proposition \ref{Prop: Harmonic Counting Form}, since $\rank(w_*) = 3$, we will try to construct $u$-only counting forms $\w$ with rank $3$ such that \eqref{Eq: Form Cond 1} and \eqref{Eq: Form Cond 2} hold for all $\delta \in \big(0,\frac{1}{2}\big)$. Specifically, when performing computer simulations, e.g. for Table \ref{Table: LP Solutions}, we will choose $r = 4$ to ensure all our calculations are correct, because $\CE(\w)$ has rank at most $4$ when $\rank(\w) = 3$. However, we continue our development for general fixed $r$ for now.

\subsection{Cyclic Evaluation and Graph Theoretic Characterizations}
\label{Graph Theoretic Characterization}

In order to efficiently test $\simeq$ and $\succeq$ computationally, we introduce the notion of cyclic evaluation of counting forms in this subsection.

\begin{definition}[Cyclic Evaluation]
\label{Def: Cyclic Counting Forms}
In addition to the standard (acyclic) evaluation over strings presented in Definition \ref{Def: Counting Form}, a basis clause $\{v\}$ with rank $s \in \N\backslash\!\{0\}$ can be also evaluated cyclically over strings, i.e. it can operate on input strings in a periodic fashion:
\begin{align*}
\forall k \in & \, \N\backslash\!\{0\}, \, \forall y = (y_0 \cdots y_{k-1}) \in \Y^k , \\
& \{v\}[y] \triangleq 
\begin{cases}
\displaystyle{\sum_{i = 0}^{k-1}{\I\!\left\{\left(y_{(i)_k} \cdots y_{(i + s - 1)_k}\right) = v\right\}}} \, , & s \leq k \\
0 \, , & s > k 
\end{cases} 
\end{align*}
where $(i)_k \equiv i \Mod{k}$ for every $i \in \N$. In particular, we utilize the square bracket notation $[\cdot]$ to represent such cyclic evaluation, as opposed to the usual parentheses $(\cdot)$ used to represent acyclic evaluation. Furthermore, a counting form $w = \alpha_1 \{v_1\} + \cdots + \alpha_m \{v_m\}$ with rank $s \in \N\backslash\!\{0\}$, basis clauses $\{v_1\},\dots,\{v_m\}$, and coefficients $\alpha_1,\dots,\alpha_m \in \R$ (with $m \in \N\backslash\!\{0\}$) can also be evaluated cyclically over strings, and in this case, we say that $w$ is a \emph{cyclic potential function} $w : \Y^* \rightarrow \R$ which operates on input strings as follows:
$$ \forall y \in \Y^* , \enspace w[y] \triangleq \sum_{j = 1}^{m}{\alpha_j \{v_j\}[y]} \, . $$
\end{definition}

For any counting form $w$ with rank at most $r$, we will assume that the domain of its corresponding cyclic potential function is $\cup_{k \geq r}{\Y^r}$. This is in contrast to standard (acyclic) potential functions, where the domain is $\Y_r^*$. We will see that allowing strings which contain $u$'s in the first and last $r$ positions to remain in the domain will lead to lucid algebraic conditions in the sequel. The next definition presents analogs of $\simeq$ and $\succeq$ for cyclic evaluation of counting forms $w : \cup_{k \geq r}{\Y^r} \rightarrow \R$ with rank at most $r$. (Although we will primarily be concerned with the vector space of $u$-only counting forms with rank at most $r$, unlike Definition \ref{Def: Partial Order}, it is convenient to state the ensuing definition for general counting forms with rank at most $r$.)

\begin{definition}[Cyclic Equivalence Classes and Partial Order]
\label{Def: Cyclic Partial Order}
As before, let $\0$ denote the zero counting form, which outputs $0$ on every input string after cyclic evaluation. We define the \emph{cyclic equivalence relation} $\cyceq$ and the \emph{cyclic partial order} $\cycgeq$ over counting forms as follows. For any pair of counting forms $w_1$ and $w_2$ with $\max\{\rank(w_1),\rank(w_2)\} \leq r$, we write $w_1 \cyceq w_2$ if we have:
$$ \forall y \in \bigcup_{k \geq r}{\Y^k}, \enspace w_1[y] = w_2[y] \, , $$
and we write $w_1 \cycgeq w_2$ if we have:
$$ \forall y \in \bigcup_{k \geq r}{\Y^k}, \enspace w_1[y] \geq w_2[y] \, . $$
Hence, the aforementioned vector space of counting forms with rank at most $r$ is also a partially ordered vector space of cyclic equivalence classes $\big\{\{w + z : \rank(z) \leq r, \, z \cyceq \0\} : \rank(w) \leq r\big\}$. Finally, we say that a counting form $w$ with rank at most $r$ is \emph{cyclically non-negative} if $w \cycgeq \0$, and we will often write $w_1 \cycgeq w_2$ equivalently as $w_1 - w_2 \cycgeq \0$.
\end{definition}

The proposition below illustrates that for the special case of $u$-only counting forms, the relations $\simeq$ and $\succeq$ can be deduced from the (stronger) relations $\cyceq$ and $\cycgeq$, respectively.

\begin{proposition}[Sufficient Conditions for $\simeq$ and $\succeq$]
\label{Prop: Sufficient Conditions for Equivalence and Partial Order}
For any pair of $u$-only counting forms $w_1$ and $w_2$ with rank at most $r$, we have:
\begin{enumerate}
\item $w_1 \cyceq w_2$ implies $w_1 \simeq w_2$,
\item $w_1 \cycgeq w_2$ implies $w_1 \succeq w_2$.
\end{enumerate}
\end{proposition}

\begin{proof}
Using Definitions \ref{Def: Partial Order} and \ref{Def: Cyclic Partial Order}, it suffices to prove this proposition for $w_2 = \0$. So, given a $u$-only counting form $w = \alpha_1 \{v_1\} + \cdots + \alpha_m \{v_m\}$ with rank at most $r$, $m \in \N\backslash\!\{0\}$, basis clauses $\{v_1\},\dots,\{v_m\}$ that contain $u$'s, and coefficients $\alpha_1,\dots,\alpha_m \in \R$, we will show that: 
\begin{enumerate}
\item $w \cyceq \0$ implies $w \simeq \0$,
\item $w \cycgeq \0$ implies $w \succeq \0$.
\end{enumerate}

\textbf{Part 1:} Observe that for any basis clause $\{v\}$ that contains $u$'s and has rank $s \leq r$, and any string $y = (y_0 \cdots y_{k-1}) \in \Y^k \cap \Y_r^*$ with $k \geq r$, we have:
\begin{align*}
\{v\}[y] & = \sum_{i = 0}^{k-1}{\I\!\left\{\left(y_{(i)_k} \cdots y_{(i + s - 1)_k}\right) = v\right\}} \\
& = \sum_{i = 0}^{k-s}{\I\!\left\{\left(y_{i} \cdots y_{i + s - 1}\right) = v\right\}} \\
& \quad \, + \sum_{i = k-s+1}^{k-1}{\underbrace{\I\!\left\{\left(y_{(i)_k} \cdots y_{(i + s - 1)_k}\right) = v\right\}}_{= \, 0}} \\
& = \{v\}(y)
\end{align*}
where the first equality follows from Definition \ref{Def: Cyclic Counting Forms}, and the third equality follows from Definition \ref{Def: Counting Form} and the fact that $v$ contains $u$'s and $y \in \Y_r^*$. Hence, using Definitions \ref{Def: Counting Form} and \ref{Def: Cyclic Counting Forms}, we obtain:
\begin{equation}
\label{Eq: Equivalence of Form Computation}
\forall y \in \Y_r^*, \enspace w(y) = w[y] \, .
\end{equation}
Since $w \cyceq \0$, we have $w[y] = 0$ for all $y \in \cup_{k \geq r}{\Y^k}$ using Definition \ref{Def: Cyclic Partial Order}. This implies that $w(y) = 0$ for all $y \in \Y_r^*$ via \eqref{Eq: Equivalence of Form Computation}. This yields $w \simeq \0$ using Definition \ref{Def: Partial Order}.

\textbf{Part 2:} Since $w \cycgeq \0$, we have $w[y] \geq 0$ for all $y \in \cup_{k \geq r}{\Y^k}$ using Definition \ref{Def: Cyclic Partial Order}. This implies that $w(y) \geq 0$ for all $y \in \Y_r^*$ via \eqref{Eq: Equivalence of Form Computation}. This yields $w \succeq \0$ using Definition \ref{Def: Partial Order}.
\end{proof}

With Proposition \ref{Prop: Sufficient Conditions for Equivalence and Partial Order} at our disposal, we can write down the ensuing \emph{sufficient conditions} for \eqref{Eq: Form Cond 1} and \eqref{Eq: Form Cond 2}; specifically, we seek to construct a $u$-only counting form $w_{\delta}$ such that for all $\delta \in \big(0,\frac{1}{2}\big)$, the following cyclic non-negativity constraints are satisfied:
\begin{align}
\label{Eq: Cyc Form Cond 1}
w_{\delta} - \CE(w_{\delta}) & \cycgeq \0 \\
\exists \, C = C(\delta) > 0, \enspace w_{\delta} - C \{u\} & \cycgeq \0
\label{Eq: Cyc Form Cond 2}
\end{align}
where \eqref{Eq: Cyc Form Cond 1} implies \eqref{Eq: Form Cond 1}, and \eqref{Eq: Cyc Form Cond 2} implies \eqref{Eq: Form Cond 2}. (At this stage, we already have all we need to prove Theorem \ref{Thm: Sufficient Condition for NAND 2D Regular Grid}, but we defer the proof to subsection \ref{Proof of Theorem Sufficient Condition for NAND 2D Regular Grid}.) As we mentioned earlier, these new constraints are easy to test computationally. Indeed, we will develop graph theoretic methods to exactly verify the cyclic equivalence relation and partial order in due course. 

However, before presenting these graph theoretic characterizations, we briefly digress and present an equivalent characterization of the cyclic equivalence relation. Our condition is inspired by the notion of ``purification'' introduced in Lemma \ref{Lemma: Purification of Counting Forms}. Indeed, notice that Lemma \ref{Lemma: Purification of Counting Forms} also holds for $\cyceq$: \emph{For any counting form $w_1$ with rank $s \leq r$, the counting forms $w_2$ and $w_3$, defined in \eqref{Eq: Pure 1} and \eqref{Eq: Pure 2}, respectively, satisfy $w_1 \cyceq w_2 \cyceq w_3$, and have pure rank $t \in [s,r]$}. (We omit an explicit proof of this fact since it follows the argument for Lemma \ref{Lemma: Purification of Counting Forms} mutatis mutandis.) Hence, for any basis clause $\{v\}$ with rank less than $r$, we clearly have $\0 \cyceq \{v\} - \{v\} \cyceq \{v,0\} + \{v,1\} + \{v,u\} - \{0,v\} - \{1,v\} - \{u,v\}$ via purification, where $\{v,z\},\{z,v\}$ denote basis clauses that are obtained by concatenating $v$ and $z \in \Y$. The next theorem shows that linear combinations of counting forms of this kind form a nontrivial linear subspace of counting forms that are equivalent to $\0$ in the cyclic sense.

\begin{theorem}[Equivalent Characterization of $\cyceq$]
\label{Thm: Equivalent Characterization of Cyclic Equivalence}
Consider any counting form $w$ with pure rank $s \in \{2,\dots,r\}$. Then, $w \cyceq \0$ if and only if $w \in \linspan\!\big(\big\{\rho_{v} : v \in \Y^{s-1} \big\}\big)$, where $\linspan(\cdot)$ denotes the linear span of its input counting forms, and for any basis clause $\{v\}$ that has rank $s-1$, $\rho_v$ denotes the following counting form with pure rank $s$:
\begin{equation}
\label{Eq: Def of rho_v}
\rho_v \triangleq \sum_{z \in \Y}{\{v,z\} - \{z,v\}} \, . 
\end{equation}
Furthermore, the linear subspace $\linspan\!\big(\big\{\rho_{v} : v \in \Y^{s-1} \big\}\big)$ has dimension $3^{s-1} - 1$.
\end{theorem}

Theorem \ref{Thm: Equivalent Characterization of Cyclic Equivalence} is established in appendix \ref{Proof of Equivalent Characterization of Cyclic Equivalence} using algebraic topological ideas that are more involved than the graph theoretic notions utilized in this section. (Note that we state this result as a theorem because it demonstrates the significance of the counting forms $\{\rho_{v} : v \in \Y^{s-1}\}$ as elucidated by the development in appendix \ref{Proof of Equivalent Characterization of Cyclic Equivalence}.) Using Theorem \ref{Thm: Equivalent Characterization of Cyclic Equivalence}, we see that to verify whether a counting form $w$ with rank $s \in \{2,\dots,r\}$ satisfies $w \cyceq \0$, we can first purify it to obtain an equivalent counting form $\tilde{w}$ with pure rank $s$, and then test whether $\tilde{w} \cyceq \0$ using the condition in Theorem \ref{Thm: Equivalent Characterization of Cyclic Equivalence}. This latter test simply involves solving a linear programming feasibility problem. Moreover, to verify whether $w_1 \cyceq w_2$ for two counting forms with rank at most $r$, we can simply test for $w_1 - w_2 \cyceq \0$ using the aforementioned procedure. Note that a specialization of the converse direction of Theorem \ref{Thm: Equivalent Characterization of Cyclic Equivalence} states that: \emph{If a $u$-only counting form $w$ with pure rank $s \in \{2,\dots,r\}$ satisfies $w \in \linspan(\{\rho_{v} : v \in \Y^{s-1} \backslash \{0,1\}^{s-1}\})$, then $w \cyceq \0$.} We present some numerical simulations based on this special case of Theorem \ref{Thm: Equivalent Characterization of Cyclic Equivalence} in appendix \ref{Necessary Conditions to Satisfy Cyc Form Cond 1 for Small Delta}, which provide evidence for the existence of $u$-only counting forms satisfying \eqref{Eq: Cyc Form Cond 1} for sufficiently small values of $\delta > 0$. On a related but separate note, we believe that any counting form $w$ with pure rank $s \in \{2,\dots,r\}$ satisfies $w \cycgeq \0$ if and only if $w = w^{\prime} + w_{0}$ for some counting form $w^{\prime}$ with pure rank $s$ and all non-negative coefficients, and some counting form $w_0 \in \linspan(\{\rho_{v} : v \in \Y^{s-1} \})$. The converse direction of this statement holds due to Theorem \ref{Thm: Equivalent Characterization of Cyclic Equivalence}, but the forward direction is an open problem.

Finally, we turn to presenting our graph theoretic characterizations of the cyclic equivalence relation and partial order. To this end, consider any counting form $w$ with pure rank $r$ defined by the formal sum:
\begin{equation}
\label{Eq: Pure Rank Formal Sum}
w = \sum_{v \in \Y^r}{\alpha_v \{v\}}
\end{equation}
with coefficients $\{\alpha_v \in \R : v \in \Y^r\}$. Corresponding to $w$, construct the weighted directed graph $\G_r(w)$ with vertex set $\Y^r$ (of strings with length $r$, or equivalently, basis clauses with rank $r$), directed edge set:\footnote{Here, the notation for an edge $(v,z) \in \Y^r \times \Y^r$ should not be confused with the notation for concatenation of the strings $v$ and $z$; indeed, the correct meaning of $(v,z)$ should be clear from context in the sequel.}
\begin{equation} 
\label{Eq: Constructed edge set}
E \triangleq \left\{(v,z) \in \Y^r \times \Y^r : (v_2 \cdots v_r) = (z_1 \cdots z_{r-1})\right\} ,
\end{equation}
where $v = (v_1 \cdots v_r) \in \Y^r$ and $z = (z_1 \cdots z_r) \in \Y^r$, and weight function $\W: E \rightarrow \R$ given by:
\begin{equation}
\label{Eq: Weight Construction for G}
\forall (v,z) \in E, \enspace \W((v,z)) \triangleq \alpha_{v} \, . 
\end{equation}
This graph encodes all the information required to compute the cyclic potential of any given string. For instance, each vertex of $\G_r(w)$ has outdegree $3$, and transitioning between these vertices via the directed edges can be equivalently construed as sliding an $r$-length window along a string. Indeed, any path $v_1 \rightarrow v_2 \rightarrow \cdots \rightarrow v_{k-1} \rightarrow v_k$ on this graph corresponds to successively visiting all possible $r$-length sub-strings of the associated string $((v_1)_1 \, (v_2)_1 \cdots (v_{k-1})_1 \, (v_k)_1 \cdots (v_k)_r) \in \Y^{r + k-1}$, where $(v_i)_j \in \Y$ denotes the $j$th letter in the $i$th sub-string $v_i \in \Y^r$ for $i \in \{1,\dots,k\}$ and $j \in \{1,\dots,r\}$. In fact, it is straightforward to verify that $\G_r(w)$ is \emph{strongly connected}, because we can reach any vertex (or basis clause) from any other vertex via some sequence of intermediate letters in $\Y$. Furthermore, the coefficients of $w$ inherently assign weights to each vertex of $\G_r(w)$, and we let each directed edge of $\G_r(w)$ inherit the weight associated with its source vertex. So, for example, for the aforementioned path $v_1 \rightarrow v_2 \rightarrow \cdots \rightarrow v_{k-1} \rightarrow v_k$, the sum of the weights along the edges of this path is equal to $\alpha_{v_1} + \cdots + \alpha_{v_{k-1}}$. This equivalence can be used to capture the cyclic potentials of general strings. The proposition below rigorizes these intuitions, and conveys equivalent characterizations of $w \cyceq \0$ and $w \cycgeq \0$.

\begin{proposition}[Graph Theoretic Characterizations of $\cyceq$ and $\cycgeq$]
\label{Prop: Graph Theoretic Characterization of Cyclic Partial Order}
For any counting form $w$ with pure rank $r$, the following are true:
\begin{enumerate}
\item $w \cyceq \0$ if and only if all cycles in the corresponding graph $\G_r(w)$ have zero total weight, where the total weight of any path in $\G_r(w)$ is given by the sum of the weights of the edges on the path.
\item $w \cycgeq \0$ (i.e. $w$ is cyclically non-negative) if and only if the corresponding graph $\G_r(w)$ does not contain any negative cycles (i.e. cycles with negative total weight).\footnote{In this paper, we refer to directed paths that begin and end at the same vertex as (directed) \emph{circuits}, and we refer to circuits with no repeated vertices or edges as (directed) \emph{cycles}.} 
\end{enumerate}
\end{proposition}

\begin{proof} 
We only prove part 2, because the proof of part 1 follows mutatis mutandis. To prove the forward direction of part 2, assume that $w \cycgeq \0$. Suppose further for the sake of contradiction that there exists a negative cycle $v_0 \rightarrow v_1 \rightarrow \cdots \rightarrow v_{m-1} \rightarrow v_0$ in $\G_r(w)$, which has $m \in \N\backslash\!\{0\}$ edges. (Note that the vertices $v_1,\dots,v_{m-1}$ are distinct and not equal to $v_0$, and the case $m = 1$ corresponds to a self-loop $v_0 \rightarrow v_0$). Let $\tilde{y} = \big((v_0)_0 \cdots (v_{m-1})_0\big) \in \Y^{m}$, where $(v_i)_0$ is the first letter of $v_i$ for all $i \in [k]$, and construct the associated string:
$$ y = \underbrace{\left(\tilde{y},\dots,\tilde{y}\right)}_{r \text{ times}} \in \Y^{r m} $$
which concatenates $\tilde{y}$ with itself $r$ times. (This concatenation process ensures that the length of $y$ is at least $r$, so that its cyclic potential can be non-zero.) The non-negative cyclic potential of $y$, $w[y]$, is given by:
\begin{align*}
0 & \leq w[y] = \sum_{v \in \Y^r}{\alpha_v \sum_{i = 0}^{m-1}{r \I\!\left\{v_i = v\right\}}} \\
& = r \sum_{i = 0}^{m-1}{\alpha_{v_i}} \\
& = r \underbrace{\sum_{i = 0}^{m-1}{\W\!\left(\left(v_i,v_{(i+1)_{m}}\right)\right)}}_{\substack{\text{total weight of cycle}\\v_0 \rightarrow v_1 \rightarrow \cdots \rightarrow v_{m-1} \rightarrow v_0}}
\end{align*}
where the first equality uses \eqref{Eq: Pure Rank Formal Sum}, Definition \ref{Def: Cyclic Counting Forms}, and the construction of $y$, the second equality follows from swapping the order of summations, and the third equality uses \eqref{Eq: Weight Construction for G}. This contradicts our assumption that $v_0 \rightarrow v_1 \rightarrow \cdots \rightarrow v_{m-1} \rightarrow v_0$ is a negative cycle. Hence, $\G_r(w)$ does not contain any negative cycles.

To prove the converse direction of part 2, assume that $\G_r(w)$ does not contain any negative cycles. This implies that $\G_r(w)$ does not contain any negative circuits. To prove this, suppose for the sake of contradiction that $\G_r(w)$ contains a circuit starting and ending at the vertex $v_1$ with negative total weight. Then, there exists a vertex $v_2 \neq v_1$ in this circuit that is visited at least twice. (If no such $v_2$ exists, then this circuit must be a cycle and its total weight cannot be negative.) So, we can decompose this circuit into a path from $v_1$ to $v_2$, followed by a circuit from $v_2$ to itself, followed by another path from $v_2$ back to $v_1$. This decomposition yields two shorter circuits, namely, one circuit from $v_1$ to itself via $v_2$, and another circuit from $v_2$ to itself. Moreover, at least one of these shorter circuits must have negative total weight (since the larger circuit has negative total weight). By recursively decomposing negative circuits and obtaining smaller negative circuits, we will eventually obtain a negative cycle. This yields a contradiction, and $\G_r(w)$ must not have had any negative circuits.

Now consider any string $y = (y_0 \cdots y_{k-1}) \in \Y^k$ for any fixed $k \geq r$, and construct the associated circuit $v_0 \rightarrow v_1 \rightarrow \cdots \rightarrow v_{k-1} \rightarrow v_0$, where $v_i = \big(y_{(i)_k} \cdots y_{(i + r-1)_k}\big)$ for all $i \in [k]$. It is straightforward to verify that the edges of this circuit exist in $\G_r(w)$. Then, using essentially the same argument as before, the cyclic potential of $y$, $w[y]$, is non-negative, because:
\begin{align*}
w[y] & = \sum_{v \in \Y^r}{\alpha_v \sum_{i = 0}^{k-1}{\I\!\left\{\left(y_{(i)_k} \cdots y_{(i + r-1)_k}\right) = v\right\}}} \\
& = \sum_{i = 0}^{k-1}{\alpha_{v_i}} \\
& = \underbrace{\sum_{i = 0}^{k-1}{\W\!\left(\left(v_i,v_{(i+1)_k}\right)\right)}}_{\substack{\text{total weight of circuit}\\v_0 \rightarrow v_1 \rightarrow \cdots \rightarrow v_{k-1} \rightarrow v_0}} \\
& \geq 0
\end{align*}
where, as before, the first equality uses \eqref{Eq: Pure Rank Formal Sum} and Definition \ref{Def: Cyclic Counting Forms}, the second equality follows from swapping the order of summations, the third equality uses \eqref{Eq: Weight Construction for G}, and the final inequality holds because $\G_r(w)$ does not contain any negative circuits. Therefore, we obtain that $w \cycgeq \0$ using Definition \ref{Def: Cyclic Partial Order}. This completes the proof.
\end{proof}

Several remarks are in order at this point. Firstly, as explained after Theorem \ref{Thm: Equivalent Characterization of Cyclic Equivalence}, to check whether a general counting form $w$ with rank $s \leq r$ satisfies $w \cycgeq \0$ or $w \cyceq \0$, we can first construct an equivalent purified counting form $\tilde{w}$ with pure rank $r$, and then apply Proposition \ref{Prop: Graph Theoretic Characterization of Cyclic Partial Order} to $\tilde{w}$. Furthermore, to test whether $w_1 \cycgeq w_2$ (respectively $w_1 \cyceq w_2$) for two counting forms with rank at most $r$, we can simply test for $w_1 - w_2 \cycgeq \0$ (respectively $w_1 - w_2 \cyceq \0$).

Secondly, it is worth mentioning that the graph $\G_r(w)$ has three vertices with self-loops: all zeros $(0 \cdots 0)$, all ones $(1 \cdots 1)$, and all $u$'s $(u \cdots u)$. Since we will only test cyclic non-negativity of $u$-only counting forms $w$, the self-loop edges on $(0 \cdots 0)$ and $(1 \cdots 1)$ always get weights of $0$, and can never contribute to negative cycles. On the other hand, part 2 of Proposition \ref{Prop: Graph Theoretic Characterization of Cyclic Partial Order} illustrates that a necessary condition for $w \cycgeq \0$ is that the weight of the self-loop edge on $(u \cdots u)$ is non-negative, i.e. if the weight of the self-loop edge on $(u \cdots u)$ is negative, then $w$ is not cyclically non-negative. Hence, when verifying whether a candidate $u$-only counting form $\w$ with rank $r - 1$ satisfies \eqref{Eq: Cyc Form Cond 1} and \eqref{Eq: Cyc Form Cond 2} for fixed values of $\delta \in \big(0,\frac{1}{2}\big)$, the coefficients associated with the basis clause $\{u \cdots u\}$ in purified versions of both $w_{\delta} - \CE(w_{\delta})$ and $w_{\delta} - C \{u\}$ must be non-negative (where $w_{\delta} - \CE(w_{\delta})$ is also $u$-only\textemdash see the discussion after Definition \ref{Def: Conditional Expectation Operator}).

Lastly, we remark that the condition in part 2 of Proposition \ref{Prop: Graph Theoretic Characterization of Cyclic Partial Order}, which is the more important part of the proposition since we seek to verify \eqref{Eq: Cyc Form Cond 1} and \eqref{Eq: Cyc Form Cond 2}, can be tested using several well-known graph shortest path algorithms. For example, we can utilize the \emph{Bellman-Ford algorithm} to detect negative cycles in $\G_r(w)$ that are reachable from a certain fixed source vertex, cf. \cite[Section 24.1]{Cormenetal2009}. In fact, since $\G_r(w)$ is a strongly connected graph, we can choose to compute shortest paths from any initial vertex of $\G_r(w)$ in order to detect negative cycles in the entire graph.\footnote{For directed graphs that are not strongly connected, one could either run the Bellman-Ford algorithm several times with different choices of source vertices, or run the \emph{Floyd-Warshall algorithm}, which can detect negative cycles reachable from any source vertex in the graph, cf. \cite[Section 25.2]{Cormenetal2009}.} For completeness, we refer readers to Algorithm \ref{Algorithm: Bellman-Ford} in appendix \ref{Algorithm to Detect Non-negative Cyclic Counting Forms}, which provides pseudocode for a Bellman-Ford-based algorithm to determine the cyclic non-negativity of $u$-only counting forms.

\subsection{Proof of Theorem \ref{Thm: Sufficient Condition for NAND 2D Regular Grid}}
\label{Proof of Theorem Sufficient Condition for NAND 2D Regular Grid}

At noted earlier, we are ready to establish Theorem \ref{Thm: Sufficient Condition for NAND 2D Regular Grid} by tracing back through some of the results and arguments in this section. Since most of the arguments have already been carried out rigorously, we will omit many details. 

\renewcommand{\proofname}{Proof of Theorem \ref{Thm: Sufficient Condition for NAND 2D Regular Grid}}

\begin{proof}
Fix any noise level $\delta \in \big(0,\frac{1}{2}\big)$, and assume that there exists $r \geq 2$ (with respect to which both $\succeq$ and $\cycgeq$ are defined) and a $u$-only counting form $\w$ with rank $r-1$ that satisfies \eqref{Eq: Cyc Form Cond 1} and \eqref{Eq: Cyc Form Cond 2}. Then, using Proposition \ref{Prop: Sufficient Conditions for Equivalence and Partial Order}, this $u$-only counting form $\w$ also satisfies \eqref{Eq: Form Cond 1} and \eqref{Eq: Form Cond 2} (where we also utilize the fact that $\CE(\w)$ is a $u$-only counting form with rank $r$). 

Now construct the stopping time:
\begin{align*}
T \triangleq \inf\big\{k \in \N\backslash[r-1] : & \, Y_{k,0},\dots,Y_{k,r-1}, \\
& \, Y_{k-r+1},\dots,Y_{k,k} \in \{0,1\} \big\} 
\end{align*}
with respect to the filtration $\{\F_k : k \in \N\}$. It is straightforward to verify that $\P(T < +\infty) = 1$ because, as discussed after Definition \ref{Def: Counting Form} in subsection \ref{Potential Functions}, the event $A$ in \eqref{Eq: The special event} occurs with probability $1$ due to the second Borel-Cantelli lemma, and $A \subseteq \{T < +\infty\}$. Hence, the events $\{T = m\}$ for $m \geq r-1$ partition the underlying sample space (without loss of generality). Let us condition on any one such event $\{T = m\} \in \F_m$, i.e., if the underlying probability space is $(\Omega,\F,\P)$, let us operate in the smaller conditional probability space $(\{T = m\},\F^{\prime},\P_m)$ with sample space $\{T = m\}$, $\sigma$-algebra $\F^{\prime} = \{A \cap \{T = m\} : A \in \F\}$, and conditional probability measure $\P_m(\cdot) = \P(\cdot \cap \{T = m\})/\P(T = m)$. This means that for all $k \geq m$, $Y_{k,0},\dots,Y_{k,r-1},Y_{k,k-r+1},\dots,Y_{k,k} \in \{0,1\}$ $\P_m$-almost surely (due to \eqref{Eq: Coupled BSC 2} and \eqref{Eq: Coupled Nand}). Furthermore, define the filtration $\{\F^{\prime}_k : k \geq m\}$ with $\sigma$-algebras $\F_k^{\prime} = \{A \cap \{T = m\} : A \in \F_k\} \subseteq \F^{\prime}$. 

Then, as explained after Definition \ref{Def: Counting Form} and Lemma \ref{Lemma: Purification of Counting Forms} in subsection \ref{Potential Functions}, we obtain that $w_{\delta} : \Y^* \rightarrow \R$ is a superharmonic potential function that satisfies the conditions in \eqref{Eq: Prop 1}, \eqref{Eq: Prop 2}, and \eqref{Eq: Prop 3} mutatis mutandis: 
\begin{align*}
\forall k \geq m, & \enspace \E_m\!\left[\left|\w(Y_{k})\right| \right] < +\infty \, , \\
\forall k \geq m, & \, \forall y \in \Y^{k+1} \cap \Y_r^*, \\
& \, \E_m\!\left[\w(Y_{k+1}) \middle| Y_k = y\right] \leq \w(y) \, , \\
\forall y \in \Y_r^*, & \enspace \w(y) \geq C \, \{u\}(y) \, , 
\end{align*}
where $\E_m[\cdot]$ denotes the expectation with respect to $\P_m$, and the second inequality uses Proposition \ref{Prop: Equivalence of CE}, whose proof holds for $\E_m$ in place of $\E$, because $Y_{k+1}$ is conditionally independent of the event $\{T = m\}$ given $Y_k$ for $k \geq m$ (in the original probability space). Moreover, since $\{Y_k : k \geq m\}$ remains a Markov chain adapted to $\{\F^{\prime}_k : k \geq m\}$ in the conditional probability space, we have:
$$ \forall k \geq m, \enspace \E_m\!\left[\w(Y_{k+1}) \middle| \F_k^{\prime}\right] = \E_m\!\left[\w(Y_{k+1}) \middle| Y_k\right] \leq \w(Y_k) $$
$\P_m$-almost surely, i.e. $\{\w(Y_{k}) : k \geq m\}$ is a supermartingale adapted to $\{\F^{\prime}_k : k \geq m\}$. Equivalently, within our conditional probability space, we have shown that $w_{\delta} : \Y^* \rightarrow \R$ is a superharmonic function that fulfills the two requirements delineated in Conjecture \ref{Conj: Martingale Construction} (mutatis mutandis) for the fixed value of $\delta$ under consideration.

Therefore, using the argument in the proof of Proposition \ref{Prop: Martingale Condition} mutatis mutandis, we get:
$$ \P_m\!\left(\lim_{k \rightarrow \infty}{N_{k}} = 0\right) = \P\!\left(\lim_{k \rightarrow \infty}{N_{k}} = 0 \, \middle| \, T = m\right) = 1 $$
where the random variables $N_k$ are defined in \eqref{Eq: N_k rv}. Since $m \geq r-1$ was arbitrary, taking expectations with respect to the law of $T$ yields that $\lim_{k \rightarrow \infty}{N_{k}} = 0$ $\P$-almost surely. Finally, employing Lemma \ref{Lemma: Vanishing u's Condition} shows that Conjecture \ref{Conj: NAND 2D Regular Grid} is true for the noise level $\delta$. This completes the proof. 
\end{proof}

\renewcommand{\proofname}{Proof}
		
\subsection{Linear Programming Criteria}
\label{Linear Programming Criteria}

Returning back to our discussion of counting forms, it is well-known that the graph theoretic task of computing shortest paths from a single source vertex, or detecting negative cycles reachable from this vertex, can also be solved using \emph{linear programming feasibility} problems, cf. \cite[Sections 24.4 and 29]{Cormenetal2009}. Therefore, in this subsection, we will present an efficient linear programming formulation that ensures cyclic non-negativity. This formulation can be employed to numerically compute $u$-only counting forms $w_{\delta}$ (e.g. with rank $r - 1 = 3$) that satisfy \eqref{Eq: Cyc Form Cond 1} and \eqref{Eq: Cyc Form Cond 2} for different values of $\delta \in \big(0,\frac{1}{2}\big)$ as in Table \ref{Table: LP Solutions}. 

Consider any counting form $w$ with pure rank $r$ defined by the formal sum \eqref{Eq: Pure Rank Formal Sum}, and let $\G_r(w)$ be its corresponding weighted directed graph with vertex set $\Y^r$, edge set $E$ given by \eqref{Eq: Constructed edge set}, and weight function $\W: E \rightarrow \R$ given by \eqref{Eq: Weight Construction for G}. Define the input and output incidence matrices, $\mathbf{B}_{\mathsf{in}} \in \{0,1\}^{3^r \times 3^{r+1}}$ and $\mathbf{B}_{\mathsf{out}} \in \{0,1\}^{3^r \times 3^{r+1}}$, respectively, of $\G_r(w)$ entry-wise as follows:
\begin{align}
\forall v \in \Y^r, \, \forall e \in E, \enspace \left[\mathbf{B}_{\mathsf{in}}\right]_{v,e} & \triangleq \I\!\left\{v \text{ has incoming edge } e\right\} \\
\forall v \in \Y^r, \, \forall e \in E, \enspace \left[\mathbf{B}_{\mathsf{out}}\right]_{v,e} & \triangleq \I\!\left\{v \text{ has outgoing edge } e\right\}
\end{align}
where the rows are indexed by vertices in $\Y^r$, and the columns are indexed by edges in $E$. (Note that a self-loop is considered both an incoming and an outgoing edge.) Using $\mathbf{B}_{\mathsf{in}}$ and $\mathbf{B}_{\mathsf{out}}$, we can define the \emph{oriented incidence matrix} $\mathbf{B} \in \{0,1\}^{3^r \times 3^{r+1}}$ of $\G_r(w)$ as:
\begin{equation}
\mathbf{B} \triangleq \mathbf{B}_{\mathsf{out}} - \mathbf{B}_{\mathsf{in}}
\end{equation}  
which is a fundamental matrix utilized in algebraic graph theory to study the homology of graphs.\footnote{Indeed, such matrices define \emph{boundary operators} in the study of abstract simplicial complices in algebraic topology. Moreover, the dual Gramian matrix corresponding to the oriented incidence matrix of a directed simple graph is equal to the unnormalized \emph{Laplacian matrix} of the associated undirected graph obtained by removing edge orientations (see e.g. \cite[Section 8.3]{GodsilRoyle2001}).} The kernel (or nullspace) of $\mathbf{B}$, denoted $\kernel(\mathbf{B}) \subseteq \R^{3^{r+1}}$, is the so called \emph{flow space} of the graph $\G_r(w)$, which is (informally) the linear span of all cycles of $\G_r(w)$, cf. \cite[Section 14.2]{GodsilRoyle2001}. Indeed, let us encode any path in $\G_r(w)$ as a column vector $x \in \N^{3^{r+1}}$ (indexed by edges in $E$) whose $e$th element $x_e \geq 0$ denotes the number of times the edge $e \in E$ was traversed by the path. Then, we have the following lemma from \cite[Section 14.2]{GodsilRoyle2001}.

\begin{lemma}[Circuits and Flow Space {\cite[Theorem 14.2.2]{GodsilRoyle2001}}]
\label{Lemma: Flow Space}
If a vector $x \in \N^{3^{r+1}}$ encodes a circuit in $\G_r(w)$, then $x \in \kernel(\mathbf{B})$, i.e. $x$ is a flow vector.
\end{lemma}

Using Lemma \ref{Lemma: Flow Space} and Proposition \ref{Prop: Graph Theoretic Characterization of Cyclic Partial Order}, we will present sufficient conditions for cyclic equivalence and non-negativity in terms of matrix relations and linear programming. To this end, as shown in appendix \ref{Necessary Conditions to Satisfy Cyc Form Cond 1 for Small Delta}, let us represent the counting form $w$ using its column vector of coefficients $\alpha = (\alpha_v : v \in \Y^r) = \zeta_r(w) \in \R^{3^r}$ akin to \eqref{Eq: The Coeff Map}. Note that $\alpha$ is indexed by the vertices of $\G_r(w)$ consistently with $\mathbf{B}$. The next proposition states our matrix and linear programming based criteria.

\begin{proposition}[Matrix and Linear Programming Criteria for $\cyceq$ and $\cycgeq$]
\label{Prop: Linear Programming Criterion for Cyclic Partial Order}
Given any counting form $w$ with pure rank $r$ and coefficients $\alpha = (\alpha_v : v \in \Y^r) \in \R^{3^r}$, the following are true:
\begin{enumerate}
\item If there exists a vector $z \in \R^{3^r}$ such that $\mathbf{B}^{\T} z = \mathbf{B}_{\mathsf{out}}^{\T} \alpha$, then $w \cyceq \0$. 
\item If there exists a vector $z \in \R^{3^r}$ such that $\mathbf{B}^{\T} z \leq \mathbf{B}_{\mathsf{out}}^{\T} \alpha$ entry-wise, then $w \cycgeq \0$.
\end{enumerate}
\end{proposition}

\begin{proof} ~\newline
\indent
\textbf{Part 1:} First, observe that for any flow vector $x \in \N^{3^{r+1}}$ that encodes a circuit in $\G_r(w)$, the bilinear form $\alpha^{\T} \mathbf{B}_{\mathsf{out}} x$ is equal to the total weight of the circuit corresponding to $x$ due to \eqref{Eq: Weight Construction for G}. (Note that there may be multiple circuits with the same encoding $x$, but these circuits have the same total weight.) Now, suppose the edge weights $\alpha^{\T} \mathbf{B}_{\mathsf{out}}$ of $\G_r(w)$ live in the row space of $\mathbf{B}$, or equivalently, suppose there exists $z \in \R^{3^r}$ such that $\mathbf{B}^{\T} z = \mathbf{B}_{\mathsf{out}}^{\T} \alpha$. Then, for any circuit in $\G_r(w)$ encoded by the flow vector $x \in \N^{3^{r+1}}$, since $x \in \ker(\mathbf{B})$ by Lemma \ref{Lemma: Flow Space}, the total weight of the circuit is $\alpha^{\T} \mathbf{B}_{\mathsf{out}} x = z^{\T} \mathbf{B} x = 0$. So, using part 1 of Proposition \ref{Prop: Graph Theoretic Characterization of Cyclic Partial Order}, we have $w \cyceq \0$. This proves part 1.

\textbf{Part 2:} To prove part 2, consider the feasible integer LP:
$$ L_1 = \min_{\substack{x \in \Z^{3^{r+1}}:\\ \mathbf{B} x = \0, \, x \geq \0}}{\alpha^{\T} \mathbf{B}_{\mathsf{out}} x} $$
where inequalities among vectors hold entry-wise. Note that the constraints of this integer LP ensure that the minimization is over all non-negative integer-valued flow vectors $x$, which includes all flow vectors that encode circuits in $\G_r(w)$ by Lemma \ref{Lemma: Flow Space}. Moreover, when the flow vector $x$ encodes a circuit, the objective function $\alpha^{\T} \mathbf{B}_{\mathsf{out}} x$ is equal to the total weight of this encoded circuit. Thus, using part 2 of Proposition \ref{Prop: Graph Theoretic Characterization of Cyclic Partial Order}, if $L_1 = 0$ then $w \cycgeq \0$ (where we use the fact that $\G_r(w)$ has no negative cycles if and only if it has no negative circuits\textemdash see the proof of Proposition \ref{Prop: Graph Theoretic Characterization of Cyclic Partial Order}).\footnote{Notice further that, since the constraint set is invariant to scaling with non-negative integers, $L_1 < 0$ implies $L_1 = -\infty$, i.e. if there exists $x \in \Z^{3^{r+1}}$ satisfying $\mathbf{B} x = \0$ and $x \geq \0$ such that $\alpha^{\T} \mathbf{B}_{\mathsf{out}} x < 0$, then the integer LP is unbounded. Similarly, $L_1 \geq 0$ implies $L_1 = 0$, because the zero vector is a feasible solution of the integer LP. Thus, there are only two possible values of $L_1$, viz. $L_1 \in \{0,-\infty\}$.} 

We next relax the integer constraints of this integer LP and obtain the following feasible LP:
$$ L_2 = \min_{\substack{x \in \R^{3^{r+1}}:\\ \mathbf{B} x = \0, \, x \geq \0}}{\alpha^{\T} \mathbf{B}_{\mathsf{out}} x} $$
which obviously satisfies $L_2 \leq L_1$. It is straightforward to verify that the dual LP of this (primal) LP is the feasibility problem (cf. \cite[Section 29.4]{Cormenetal2009}, \cite[Chapter 1, Section 2]{Franklin2002}):
$$ L_3 = \max_{\substack{z \in \R^{3^r}: \\ \mathbf{B}^{\T} z \leq \mathbf{B}_{\mathsf{out}}^{\T} \alpha}}{0} $$
where the dual variables $z$ are indexed by the vertices in $\Y^r$. Using the \emph{duality theorem} (see e.g. \cite[Chapter 1, Section 8]{Franklin2002}), there are only two possibilities:
\begin{enumerate}
\item If there exists $z \in \R^{3^r}$ such that $\mathbf{B}^{\T} z \leq \mathbf{B}_{\mathsf{out}}^{\T} \alpha$, then we have $L_1 = L_2 = L_3 = 0$ (via strong duality).
\item Otherwise, we have $L_2 = -\infty$, i.e. the primal LP is unbounded.
\end{enumerate}
This proves that if there exists $z \in \R^{3^r}$ such that $\mathbf{B}^{\T} z \leq \mathbf{B}_{\mathsf{out}}^{\T} \alpha$, then $w \cycgeq \0$ as desired.
\end{proof}

It is worth briefly interpreting the sufficient condition in part 2 of Proposition \ref{Prop: Linear Programming Criterion for Cyclic Partial Order}. The result says that the counting form $w$ is cyclically non-negative, or equivalently, all circuits in $\G_r(w)$ have non-negative total weight, if the set of edge weights $\mathbf{B}_{\mathsf{out}}^{\T} \alpha$ of $\G_r(w)$ can be decomposed as follows:
\begin{equation}
\exists z \in \R^{3^r}, \, \exists \beta \in \R^{3^{r+1}} , \enspace \mathbf{B}_{\mathsf{out}}^{\T} \alpha = \mathbf{B}^{\T} z + \beta \enspace \text{and} \enspace \beta \geq \0
\end{equation}
where the set of edge weights $\mathbf{B}^{\T} z$ lead to all circuits in $\G_r(w)$ having zero total weight (see part 1 of Proposition \ref{Prop: Linear Programming Criterion for Cyclic Partial Order}), and the set of edge weights $\beta$ are entry-wise non-negative. On a separate note, it is also worth comparing part 1 of Proposition \ref{Prop: Linear Programming Criterion for Cyclic Partial Order} with the converse direction of Theorem \ref{Thm: Equivalent Characterization of Cyclic Equivalence}, because the two results are closely related. 

Using part 2 of Proposition \ref{Prop: Linear Programming Criterion for Cyclic Partial Order}, we now illustrate sufficient conditions for \eqref{Eq: Cyc Form Cond 1} and \eqref{Eq: Cyc Form Cond 2}. Let us represent the conditional expectation operator $\CE$ from Definition \ref{Def: Conditional Expectation Operator} as the $3^{r} \times 3^{r-1}$ matrix $\mathbf{C}_{r-1}(\delta)$ (for $\delta \in \big[0,\frac{1}{2}\big]$) defined in \eqref{Eq: Useful representation of CE matrix} in appendix \ref{Necessary Conditions to Satisfy Cyc Form Cond 1 for Small Delta}, and let $P_{r-1}$ be the $3^{r} \times 3^{r-1}$ ``purification matrix'' defined in \eqref{Eq: The Purify Matrix} in appendix \ref{Necessary Conditions to Satisfy Cyc Form Cond 1 for Small Delta}. Note that each $u$-only counting form with pure rank $s \leq r$ can be represented by a column vector of coefficients in $\R^{3^s}$ using the map $\zeta_s(\cdot)$ as shown in \eqref{Eq: The Coeff Map}. So, $\mathbf{C}_{r-1}(\delta)$ and $P_{r-1}$ map coefficient vectors of $u$-only counting forms with pure rank $r-1$ into coefficient vectors of $u$-only counting forms with pure rank $r$ (see e.g. \eqref{Eq: Matrix equation for forms} in appendix \ref{Necessary Conditions to Satisfy Cyc Form Cond 1 for Small Delta}). In fact, the output vectors of $P_{r-1}$ are cyclically equivalent to the corresponding input vectors (as discussed in appendix \ref{Necessary Conditions to Satisfy Cyc Form Cond 1 for Small Delta}). We obviously assume that the entries of different coefficient vectors, and the rows and columns of $\mathbf{C}_{r-1}(\delta)$ and $P_{r-1}$, are all indexed consistently by the sets $\Y^s$ for appropriate choices of $s$. (We remark that the indexing in \eqref{Eq: Index map} in appendix \ref{Necessary Conditions to Satisfy Cyc Form Cond 1 for Small Delta} endows the strings in $\Y^s$ for any $s \leq r$ with a \emph{lexicographic ordering},\footnote{Specifically, we index the entries of vectors in $\R^{3^s}$, or the rows and columns of appropriately sized matrices, using $[3^s]$ and assume the following bijective correspondence between $[3^s]$ and strings (or basis clauses) in $\Y^s$: For each $v^{\prime} \in [3^s]$, we first write $v^{\prime}$ in its $s$-length ternary representation, and then replace each $2$ in this representation with a $u$. The resulting string $v$ belongs to $\Y^s$, and defines the correspondence $[3^s] \ni v^{\prime} \leftrightarrow v \in \Y^s$.} which is useful for computing purposes.) Furthermore, define the vector of coefficients $\psi = (\psi_v : v \in \Y^r) \in \R^{3^r}$ of a purified version of the basis clause $\{u\}$:
\begin{equation}
\forall v = (v_1 \cdots v_r) \in \Y^r, \enspace \psi_v \triangleq \I\!\left\{v_1 = u\right\} . 
\end{equation}
It is straightforward to verify that $\{u\} \cyceq \sum_{v \in \Y^r}{\psi_v \{v\}}$ (see Lemma \ref{Lemma: Purification of Counting Forms} and the discussion preceding Theorem \ref{Thm: Equivalent Characterization of Cyclic Equivalence}). The ensuing proposition portrays sufficient conditions for \eqref{Eq: Cyc Form Cond 1} and \eqref{Eq: Cyc Form Cond 2} that can be tested using an LP.

\begin{proposition}[Linear Programming Criterion for Satisfying \eqref{Eq: Cyc Form Cond 1} and \eqref{Eq: Cyc Form Cond 2}]
\label{Prop: Linear Programming Criterion}
Fix any $\delta \in \big(0,\frac{1}{2}\big)$, and construct the partitioned matrices:
\begin{equation}
\label{Eq: Partitioned Vector}
\xi \triangleq \left[\0^{\T} \, \middle| \, \psi^{\T} \mathbf{B}_{\mathsf{out}} \right]^{\T} \in \R^{2(3^{r+1})} 
\end{equation}
\begin{equation}
\label{Eq: Partitioned Matrix}
\begin{aligned}
\mathbf{A}(\delta) & \triangleq \left[
\begin{array}{c|c|c}
\mathbf{B}_{\mathsf{out}}^{\T} \left(P_{r-1} - \mathbf{C}_{r-1}(\delta)\right) & \mathbf{B}^{\T} & \0 \\
\hline
\mathbf{B}_{\mathsf{out}}^{\T} P_{r-1} \Tstrut & \0 \Tstrut & \mathbf{B}^{\T} \Tstrut
\end{array}
\right] \\
& \quad \, \in \R^{2(3^{r+1}) \times (3^{r-1} + 2(3^r))}
\end{aligned}
\end{equation}
where $\0$ denotes the zero vector or matrix of appropriate dimension (with abuse of notation). Consider the following linear programming feasibility problem:
\begin{equation}
\label{Eq: Final LP Formulation}
\begin{aligned}
\exists \, \phi(\delta) & = \left[\alpha(\delta)^{\T} \, \middle| \, z(\delta)^{\T}\right]^{\T} \in \R^{3^{r-1} + 2(3^r)} \\
\text{with} & \enspace \alpha(\delta) = \left(\alpha_v(\delta) : v \in \Y^{r-1}\right) \in \R^{3^{r-1}} \\
\text{such that} & \enspace \mathbf{A}(\delta) \phi(\delta) \geq \xi \\
\text{and} & \enspace \alpha_v(\delta) = 0 \enspace \text{for all} \enspace v \in \{0,1\}^{r-1} \, . 
\end{aligned}
\end{equation}
If this LP in \eqref{Eq: Final LP Formulation} has a feasible solution $\phi(\delta) = \phi^*(\delta) = \big[\alpha^*(\delta)^{\T} \, \big| \, z^*(\delta)^{\T}\big]^{\T} \in \R^{3^{r-1} + 2(3^r)}$, then the corresponding $u$-only counting form $w_{\delta}^*$ with pure rank $r-1$:
\begin{equation}
\label{Eq: LP Potential}
w_{\delta}^* = \sum_{v \in \Y^{r-1}}{\alpha_v^*(\delta) \{v\}} 
\end{equation}
satisfies both \eqref{Eq: Cyc Form Cond 1} and \eqref{Eq: Cyc Form Cond 2} with $C = C(\delta) = 1$. Furthermore, the $u$-only counting form $w_{\delta}^*$ satisfies both \eqref{Eq: Form Cond 1} and \eqref{Eq: Form Cond 2} with $C = 1$. 
\end{proposition}

\begin{proof}
We seek to construct a $u$-only counting form $w_{\delta}$ with pure rank $r - 1$ such that \eqref{Eq: Cyc Form Cond 1} and \eqref{Eq: Cyc Form Cond 2} hold, i.e. $w_{\delta} - \CE(w_{\delta}) \cycgeq \0$ and $w_{\delta} - C \{u\} \cycgeq \0$ for some constant $C = C(\delta) > 0$. Let $\alpha(\delta) = (\alpha_v(\delta) : v \in \Y^{r-1}) \in \R^{3^{r-1}}$ denote the vector of coefficients of $w_{\delta}$. Using part 2 of Proposition \ref{Prop: Linear Programming Criterion for Cyclic Partial Order}, a sufficient condition for \eqref{Eq: Cyc Form Cond 1} is:
\begin{equation}
\label{Eq: Suff Cond 1 Proxy}
\exists \, z_1(\delta) \in \R^{3^r}, \enspace \mathbf{B}^{\T} z_1(\delta) \leq \mathbf{B}_{\mathsf{out}}^{\T} \left(P_{r-1} - \mathbf{C}_{r-1}(\delta)\right) \alpha(\delta)
\end{equation}
because $\big(P_{r-1} - \mathbf{C}_{r-1}(\delta)\big) \alpha(\delta) \in \R^{3^r}$ is the vector of coefficients corresponding to a purified version of $w_{\delta} - \CE(w_{\delta})$. Similarly, using part 2 of Proposition \ref{Prop: Linear Programming Criterion for Cyclic Partial Order}, a sufficient condition for \eqref{Eq: Cyc Form Cond 2} is:
\begin{equation}
\label{Eq: Suff Cond 2 Proxy}
\begin{aligned}
\exists \, C = C(\delta) > 0, & \, \exists \, z_2(\delta) \in \R^{3^r}, \\
& \mathbf{B}^{\T} z_2(\delta) \leq \mathbf{B}_{\mathsf{out}}^{\T} \left(P_{r-1} \alpha(\delta) - C \psi \right) 
\end{aligned}
\end{equation}
because $P_{r-1} \alpha(\delta) - C \psi \in \R^{3^r}$ is the vector of coefficients corresponding to a purified version of $w_{\delta} - C \{u\}$. Since our goal is to find a $u$-only counting form $w_{\delta}$, or equivalently, a vector of coefficients $\alpha(\delta)$, that satisfies \eqref{Eq: Suff Cond 1 Proxy} and \eqref{Eq: Suff Cond 2 Proxy}, the vector $\alpha(\delta)$ is also unknown in \eqref{Eq: Suff Cond 1 Proxy} and \eqref{Eq: Suff Cond 2 Proxy}. Hence, we can set $C = 1$ without loss of generality in \eqref{Eq: Suff Cond 2 Proxy}, because we can multiply both sides of \eqref{Eq: Suff Cond 1 Proxy} and \eqref{Eq: Suff Cond 2 Proxy} by $C^{-1}$ and absorb this multiplicative constant into the unknown vectors $z_1(\delta)$, $z_2(\delta)$, and $\alpha(\delta)$.

Next, construct the block vector:
$$ \phi(\delta) \triangleq \left[\alpha(\delta)^{\T} \, \middle| \, -z_1(\delta)^{\T} \, \middle| \, -z_2(\delta)^{\T} \right]^{\T} \in \R^{3^{r-1} + 2(3^r)} \, , $$
the block vector $\xi \in \R^{2(3^{r+1})}$ in \eqref{Eq: Partitioned Vector}, and the block matrix $\mathbf{A}(\delta) \in \R^{2(3^{r+1}) \times (3^{r-1} + 2(3^r))}$ in \eqref{Eq: Partitioned Matrix}. Then, it is straightforward to verify that for any fixed vector $\alpha(\delta)$, the sufficient conditions \eqref{Eq: Suff Cond 1 Proxy} and \eqref{Eq: Suff Cond 2 Proxy} with $C = 1$ can be simultaneously expressed as:
$$ \exists \, z_1(\delta),z_2(\delta) \in \R^{3^r}, \enspace \mathbf{A}(\delta) \phi(\delta) \geq \xi \, . $$
Moreover, since we seek to construct an $\alpha(\delta)$ satisfying this condition, our true objective is to verify the following sufficient condition:
$$ \exists \, \alpha(\delta) \in \R^{3^{r-1}}, \, \exists \, z_1(\delta),z_2(\delta) \in \R^{3^r}, \enspace \mathbf{A}(\delta) \phi(\delta) \geq \xi $$
where $\alpha(\delta)$ additionally satisfies the constraints:
$$ \forall v \in \{0,1\}^{r-1}, \enspace \alpha_v(\delta) = 0 $$
because it must correspond to a $u$-only counting form $w_{\delta}$. If these two sets of constraints hold for some feasible $\alpha(\delta) = \alpha^*(\delta) \in \R^{3^{r-1}}$, then the $u$-only counting form $w_{\delta}^*$ defined by the coefficients $\alpha^*(\delta)$ satisfies both \eqref{Eq: Cyc Form Cond 1} and \eqref{Eq: Cyc Form Cond 2} with $C = 1$, which implies, via Proposition \ref{Prop: Sufficient Conditions for Equivalence and Partial Order}, that $w_{\delta}^*$ also satisfies both \eqref{Eq: Form Cond 1} and \eqref{Eq: Form Cond 2} with $C = 1$ (see the discussion following Proposition \ref{Prop: Sufficient Conditions for Equivalence and Partial Order}). This completes the proof.
\end{proof}

Proposition \ref{Prop: Linear Programming Criterion} can be exploited to computationally find counting forms that satisfy the conditions of Theorem \ref{Thm: Sufficient Condition for NAND 2D Regular Grid}. As noted in subsection \ref{Partial Impossibility Result for NAND 2D Regular Grid}, the MATLAB and CVX based simulation results in Table \ref{Table: LP Solutions} solve the LP in \eqref{Eq: Final LP Formulation} of Proposition \ref{Prop: Linear Programming Criterion} to numerically construct $u$-only counting forms $w_{\delta}^*$ with pure rank $3$ that satisfy \eqref{Eq: Cyc Form Cond 1} and \eqref{Eq: Cyc Form Cond 2} with $C = 1$. In particular, for any noise level $\delta \in \big(0,\frac{1}{2}\big)$, the constraint matrices \eqref{Eq: Partitioned Vector} and \eqref{Eq: Partitioned Matrix} can be explicitly written out (as expounded earlier), and the corresponding feasible vectors of coefficients $\alpha^*(\delta) \in \R^{27}$ that solve the LP in \eqref{Eq: Final LP Formulation} (with reasonably small $\ell^1$-norm), depicted in Table \ref{Table: LP Solutions}, define $u$-only counting forms $w_{\delta}^*$ with pure rank $3$ via \eqref{Eq: LP Potential}. Furthermore, inspired by the reasons outlined at the end of subsection \ref{Potential Functions}, we set $r = 4$ for these simulations; indeed, the proof of ergodicity of the 1D PCA with NAND gates also uses counting forms with rank $3$ \cite[Section 2.2]{HolroydMarcoviciMartin2019}. 

To conclude this section, we make a few other noteworthy remarks. As we indicated earlier in subsection \ref{Partial Impossibility Result for NAND 2D Regular Grid}, Theorem \ref{Thm: Sufficient Condition for NAND 2D Regular Grid} and simulations based on the other arguments in this section yield non-rigorous computer-assisted proofs of the impossibility of broadcasting on NAND 2D regular grids. Indeed, for every fixed noise level $\delta \in \big(0,\frac{1}{2}\big)$ where we can compute a feasible solution $\alpha^*(\delta)$ of the LP in \eqref{Eq: Final LP Formulation}, such as the values of $\delta$ that feature in Table \ref{Table: LP Solutions}, Proposition \ref{Prop: Linear Programming Criterion} portrays that we can obtain a $u$-only counting form $w_{\delta}^*$ with pure rank $3$ that satisfies both \eqref{Eq: Cyc Form Cond 1} and \eqref{Eq: Cyc Form Cond 2} with $C = 1$, and Theorem \ref{Thm: Sufficient Condition for NAND 2D Regular Grid} then implies that Conjecture \ref{Conj: NAND 2D Regular Grid} is true for this $\delta$. We emphasize that the computer-assisted part of this argument is not rigorous without further careful analysis, because we use floating-point arithmetic to solve the LP in \eqref{Eq: Final LP Formulation}. If we could instead solve an integer LP akin to \eqref{Eq: Final LP Formulation} (when $\delta$ is rational) or utilize \emph{interval arithmetic}, then our feasible solution $\alpha^*(\delta)$ would be deemed rigorous, and the discussion in section \ref{Martingale Approach for NAND Processing Functions} would truly yield a computer-assisted proof.

Nevertheless, even with a rigorous computational technique, the above approach only demonstrates the impossibility of broadcasting for specific values of $\delta$. So, while this approach lends credence to our development of counting forms as a framework for constructing superharmonic potential functions, it does not enable us to establish Conjecture \ref{Conj: NAND 2D Regular Grid} for all $\delta \in \big(0,\frac{1}{2}\big)$. A promising future direction to remedy this issue is to develop a more sophisticated \emph{semidefinite programming} formulation for the problem of constructing $u$-only counting forms satisfying \eqref{Eq: Cyc Form Cond 1} and \eqref{Eq: Cyc Form Cond 2} for all $\delta \in \big(0,\frac{1}{2}\big)$. For instance, consider the key constraint $\mathbf{A}(\delta) \phi(\delta) \geq \xi$ in the LP in \eqref{Eq: Final LP Formulation} without fixing $\delta$. Let us restrict the vector $\phi(\delta)$ in \eqref{Eq: Final LP Formulation} to be a polynomial vector in $\delta$ with some (sufficiently large) fixed degree and unknown coefficients. Then, since $\mathbf{C}_{r-1}(\delta)$ is a known polynomial matrix in $\delta$ with degree $2(r-1)$ (see \eqref{Eq: Polynomial Matrix} in appendix \ref{Necessary Conditions to Satisfy Cyc Form Cond 1 for Small Delta}), the constraint $\mathbf{A}(\delta) \phi(\delta) \geq \xi$ is equivalent to simultaneously ensuring the non-negativity of a collection of univariate polynomials in $\delta$, whose coefficients are affine functions of the unknown coefficients of $\phi(\delta)$, over the interval $\big[0,\frac{1}{2}\big]$. Since the \emph{Markov-Luk\'{a}cs theorem} provides well-known necessary and sufficient conditions for the non-negativity of univariate polynomials in terms of \emph{sum of squares} (SOS) characterizations \cite{Markov1906,Lukacs1918} (also see \cite[Theorems 1.21.1 and 1.21.2]{Szego1975}), we can represent the aforementioned simultaneous non-negativity of polynomials in $\delta$ over $\big[0,\frac{1}{2}\big]$ using an SOS program. Such SOS programs can be reformulated as semidefinite programs (SDPs) and computationally solved using standard optimization tools, cf. \cite{Parrilo2003,Lasserre2007}. Hence, one could computationally establish the existence of $u$-only counting forms satisfying \eqref{Eq: Cyc Form Cond 1} and \eqref{Eq: Cyc Form Cond 2} for all $\delta \in \big(0,\frac{1}{2}\big)$ using a single SDP. This, in conjunction with rigorous interval arithmetic analysis, could be an avenue for proving Conjecture \ref{Conj: NAND 2D Regular Grid} in its entirety. (Additionally, we remark that this connection between SDPs and the existence of pertinent counting forms is somewhat natural, because SDPs have been very useful in constructing Lyapunov functions in many control theoretic problems, and there are limpid parallels between superharmonic functions and Lyapunov functions as discussed earlier.)

Lastly, we reiterate that while subsection \ref{Partial Impossibility Result for NAND 2D Regular Grid} and section \ref{Martingale Approach for NAND Processing Functions} present a detailed martingale-based approach to prove the impossibility of broadcasting on the NAND 2D regular grid, they also provide a general program for establishing the impossibility of broadcasting on 2D regular grids with other Boolean processing functions. Indeed, the arguments in section \ref{Martingale Approach for NAND Processing Functions} can be easily carried out for other processing functions, such as AND, XOR, or IMP, with minor cosmetic changes, e.g. the elements of the matrix $\mathbf{C}_{r-1}(\delta)$ representing $\CE$ need to be modified accordingly.

\section{Conclusion}
\label{Conclusion}

In closing, we briefly summarize the main contributions of this paper. Propelled by the positive rates conjecture for 1D PCA, we conjectured in subsection \ref{Motivation} that broadcasting is impossible on 2D regular grids when all vertices with two inputs use a common Boolean processing function. We made considerable progress towards establishing this conjecture, and proved impossibility results for 2D regular grids with all AND and all XOR processing functions in Theorems \ref{Thm: Deterministic And Grid} and \ref{Thm: Deterministic Xor Grid}, respectively. Specifically, our proof for the AND 2D regular grid utilized phase transition results concerning oriented bond percolation in 2D lattices, and our proof for the XOR 2D regular grid relied on analyzing the equivalent problem of decoding the first bit of a codeword drawn uniformly from of a linear code. Then, we showed in Theorem \ref{Thm: Sufficient Condition for NAND 2D Regular Grid} that an impossibility result for 2D regular grids with all NAND processing functions follows from the existence of certain structured supermartingales. Moreover, in much of section \ref{Martingale Approach for NAND Processing Functions}, we developed the notion of counting forms and elucidated graph theoretic and linear programming based criteria to help construct the desired superharmonic potential functions of Theorem \ref{Thm: Sufficient Condition for NAND 2D Regular Grid}. These ideas were utilized to depict several numerical examples of such potential functions in Table \ref{Table: LP Solutions}. Finally, we list some fruitful directions of future research below:
\begin{enumerate}
\item The program in subsection \ref{Partial Impossibility Result for NAND 2D Regular Grid} and section \ref{Martingale Approach for NAND Processing Functions} can be rigorously executed to prove the impossibility of broadcasting on the NAND 2D regular grid for all $\delta \in \big(0,\frac{1}{2}\big)$. In particular, our suggestion in subsection \ref{Linear Programming Criteria} of using SDPs to obtain a computer-assisted proof could be pursued.
\item In order to completely resolve the 2D version of our conjecture in subsection \ref{Motivation}, the impossibility of broadcasting on 2D regular grids with IMP processing functions must also be established (see \eqref{Eq: The IMP gate} in subsection \ref{Partial Impossibility Result for NAND 2D Regular Grid}). 
\item In order to prove the 3D part of our broader conjecture in subsection \ref{Motivation} that broadcasting is feasible in $3$ or more dimensions for sufficiently low noise levels $\delta$, broadcasting on the majority 3D regular grid can be examined. As indicated by Proposition \ref{Prop: Majority 3D Regular Grid} and discussed in subsection \ref{3D Regular Grid Model}, it may be possible to draw on ideas from the analysis of Toom's NEC rule in \cite{Toom1980} and \cite{Gacs1995} to establish the possibility of broadcasting on the majority 3D regular grid.
\item Much like the correspondence between our 2D regular grid model with NAND processing functions and the 1D PCA with NAND gates in \cite{HolroydMarcoviciMartin2019}, we can define a 2D $45^{\circ}$ grid model with $3$-input majority processing functions corresponding to Gray's 1D PCA with $3$-input majority gates \cite[Example 5]{Gray1987}. In particular, a 2D $45^{\circ}$ grid model has $L_k = 2k + 1$ vertices at each level $k \in \N$, and every vertex has three outgoing edges that have $45^{\circ}$ separation between them.\footnote{In contrast, every vertex in the 2D regular grid model has two outgoing edges that have $90^{\circ}$ separation between them (see Figure \ref{Figure: Grid}).} Furthermore, all vertices in the 2D $45^{\circ}$ grid model which are at least two positions away from the boundary have three incoming edges. Although Gray's proof sketch of the ergodicity of the 1D PCA with $3$-input majority gates in \cite[Section 3]{Gray1987} shows exponentially fast convergence for sufficiently small noise levels, it includes noise on the vertices (rather than the edges) and obviously does not account for the boundary effects of 2D grid models. It is therefore an interesting future endeavor to study broadcasting in the 2D $45^{\circ}$ grid model with $3$-input majority processing functions by building upon the reasoning in \cite[Section 3]{Gray1987}.
\end{enumerate}

\appendices

\section{Proof of Proposition \ref{Prop: Majority 3D Regular Grid}}
\label{Proof of Proposition Majority 3D Regular Grid}

\begin{proof}
To prove this proposition, we will perform a simple projection argument. Specifically, for any $k \in \N$, we will project any vertex $v = (v_1,v_2,v_3) \in \S_k$ of the majority 3D regular grid onto its first two coordinates to obtain a site $v \mapsto (v_1,v_2) \in \hat{\S}_k$ at time $k$ of Toom's 2D PCA with boundary conditions. (Note that this projection map is bijective, because for any site $x = (x_1,x_2) \in \hat{\S}_k$ at time $k \in \N$ of Toom's 2D PCA with boundary conditions, we can recover the corresponding vertex of the majority 3D regular grid via the affine map $x \mapsto (x_1,x_2,k-x_1 - x_2) \in \S_k$.) 

To this end, let us first couple the i.i.d. $\Ber(\delta)$ random variables $\{Z_{v,i} : v \in \N^3\backslash\!\{\0\}, \, i \in \{1,2,3\}\}$ with the i.i.d. $\Ber(\delta)$ random variables $\{Z_{k,x,v} : k \in \N, \, x \in \hat{\S}_{k+1}, \, v \in \Nbhd\}$ such that:
\begin{equation}
\label{Eq: Coupling of Noise Vars}
\begin{aligned}
\forall v = (v_1,& \, v_2,v_3) \in \N^3\backslash\!\{\0\}, \\
Z_{v,i} & = Z_{v_1 + v_2 + v_3 - 1,(v_1,v_2),-e_i} \enspace \text{for } i \in \{1,2\} \, , \\
\text{and} \enspace Z_{v,3} & = Z_{v_1 + v_2 + v_3 - 1,(v_1,v_2),\0} \, ,
\end{aligned}
\end{equation}
or equivalently:
\begin{equation*}
\begin{aligned}
\forall k \in \N, \, \forall x & = (x_1,x_2) \in \hat{\S}_{k+1}, \\
Z_{k,x,-e_i} & = Z_{(x_1,x_2,k+1-x_1 - x_2),i} \enspace \text{for } i \in \{1,2\} \, , \\
\text{and} \enspace Z_{k,x,\0} & = Z_{(x_1,x_2,k+1-x_1 - x_2),3} \, ,
\end{aligned}
\end{equation*}
almost surely. The rest of the proof proceeds by strong induction. 

Notice that the base case, $X_{\0} = \xi_0(\0)$ almost surely, holds by assumption. Suppose further that the inductive hypothesis:
\begin{equation}
\label{Eq: Ind Hyp}
\forall v = (v_1,v_2,v_3) \in \bigcup_{m = 0}^{k}{\S_m}, \enspace X_v = \xi_{v_1 + v_2 + v_3}((v_1,v_2))  
\end{equation}
almost surely, holds for some $k \in \N$. Now consider any vertex $v = (v_1,v_2,v_3) \in \S_{k+1}$. There are several cases for us to verify:
\begin{enumerate}
\item If $v_1 = v_2 = 0$ and $v_3 = k+1$, then we have almost surely:
\begin{align*}
X_v & = X_{v- e_3} \oplus Z_{v,3} \\
& = \xi_{k}(\0) \oplus Z_{k, \0,\0} \\
& = \xi_{k+1}(\0) 
\end{align*}
using \eqref{Eq: Maj Grid 1}, \eqref{Eq: Ind Hyp}, \eqref{Eq: Coupling of Noise Vars}, and \eqref{Eq: Toom PCA 1}.
\item If there exists $i \in \{1,2\}$ such that $v_i = k+1$, then we have almost surely:
\begin{align*}
X_v & = X_{v- e_i} \oplus Z_{v,i} \\
& = \xi_{k}(k e_i) \oplus Z_{k, (k+1)e_i,-e_i} \\
& = \xi_{k+1}((v_1,v_2)) 
\end{align*}
using \eqref{Eq: Maj Grid 1}, \eqref{Eq: Ind Hyp}, \eqref{Eq: Coupling of Noise Vars}, and \eqref{Eq: Toom PCA 1.5}.
\item If there exists $i \in \{1,2\}$ such that $v_i + v_3 = k+1$, then we have almost surely:
\begin{align*}
X_v & = 
\begin{cases}
X_{v-e_3} \oplus Z_{v,3} \, , & \text{with probability (w.p.) } \frac{1}{2} \\
X_{v-e_i} \oplus Z_{v,i} \, , & \text{w.p. } \frac{1}{2} \\
\end{cases} \\
& = \begin{cases}
\xi_{k}((v_1,v_2)) \oplus Z_{k,(v_1,v_2),\0} \, , & \text{w.p. } \frac{1}{2} \\
\xi_{k}((v_1,v_2) - e_i) \oplus Z_{k,(v_1,v_2),-e_i} \, , & \text{w.p. } \frac{1}{2} \\
\end{cases} \\
& = \xi_{k+1}((v_1,v_2)) 
\end{align*}
using \eqref{Eq: Maj Grid 2}, \eqref{Eq: Ind Hyp}, \eqref{Eq: Coupling of Noise Vars}, and \eqref{Eq: Toom PCA 2}.
\item If $v_1 + v_2 = k+1$ and $v_3 = 0$, then we have almost surely:
\begin{align*}
X_v & = 
\begin{cases}
X_{v-e_1} \oplus Z_{v,1}  \, , & \text{w.p. } \frac{1}{2} \\
X_{v-e_2} \oplus Z_{v,2}  \, , & \text{w.p. } \frac{1}{2} \\
\end{cases} \\
& = 
\begin{cases}
\xi_{k}((v_1 - 1,v_2)) \oplus Z_{k,(v_1,v_2),-e_1} \, , & \text{w.p. } \frac{1}{2} \\
\xi_{k}((v_1,v_2 - 1)) \oplus Z_{k,(v_1,v_2),-e_2} \, , & \text{w.p. } \frac{1}{2} \\
\end{cases} \\
& = \xi_{k+1}((v_1,v_2)) 
\end{align*}
using \eqref{Eq: Maj Grid 2}, \eqref{Eq: Ind Hyp}, \eqref{Eq: Coupling of Noise Vars}, and \eqref{Eq: Toom PCA 2.5}.
\item If $v_1,v_2,v_3 > 0$, then we have almost surely:
\begin{align*}
X_v & = \maj\big(X_{v-e_1} \oplus Z_{v,1}, X_{v-e_2} \oplus Z_{v,2}, \\
& \qquad \qquad \quad \, X_{v-e_3} \oplus Z_{v,3}\big) \\
& = \maj\big(\xi_{k}((v_1-1,v_2)) \oplus Z_{k,(v_1,v_2),-e_1}, \\
& \qquad \qquad \quad \, \xi_{k}((v_1,v_2-1)) \oplus Z_{k,(v_1,v_2),-e_2}, \\
& \qquad \qquad \quad \, \xi_{k}((v_1,v_2)) \oplus Z_{k,(v_1,v_2),\0}\big) \\
& = \xi_{k+1}((v_1,v_2))
\end{align*}
using \eqref{Eq: Maj Grid 3}, \eqref{Eq: Ind Hyp}, \eqref{Eq: Coupling of Noise Vars}, and \eqref{Eq: Toom PCA 3}.
\end{enumerate} 
This shows that for every $v = (v_1,v_2,v_3) \in \S_{k+1}$, we have $X_v = \xi_{k+1}((v_1,v_2))$ almost surely. Therefore, our coupling of the i.i.d. $\Ber(\delta)$ random variables in \eqref{Eq: Coupling of Noise Vars} establishes the proposition statement by induction.
\end{proof}

\section{Proof of Theorem \ref{Thm: Equivalent Characterization of Cyclic Equivalence}}
\label{Proof of Equivalent Characterization of Cyclic Equivalence}

We will prove Theorem \ref{Thm: Equivalent Characterization of Cyclic Equivalence} in this appendix. Throughout section \ref{Martingale Approach for NAND Processing Functions}, we intentionally circumvented an exposition of the algebraic topological perspective of counting forms since it does not help us establish Theorem \ref{Thm: Sufficient Condition for NAND 2D Regular Grid} or perform the simulations illustrated in Table \ref{Table: LP Solutions}. However, in order to derive Theorem \ref{Thm: Equivalent Characterization of Cyclic Equivalence} and elucidate the elegant intuition behind the counting forms $\{\rho_v : v \in \Y^{s-1}\}$ in \eqref{Eq: Def of rho_v}, we have to introduce some basic ideas from algebraic graph theory. Although some of these ideas can be found in or derived from standard expositions, cf. \cite{GodsilRoyle2001,Lim2020}, we develop the details here in a manner that is directly relevant to us. 

Consider a strongly connected finite directed graph $\G = (\V,E)$ with vertex set $\V$ and edge set $E$. (Note that $\G$
has no isolated vertices, but $\G$ may contain self-loop edges.) Given $\G$, we begin with a list of useful definitions:
\begin{enumerate}
\item For any codomain $S \subseteq \R$, we refer to a map $f : E \rightarrow S$ as an \emph{$S$-edge function}.
Moreover, we refer to $\R$-edge functions as simply \emph{edge functions}. 
\item An edge function $f$ is said to be \emph{closed} if for every directed circuit $(e_1,\dots,e_n) \in E^n$ in $\G$,\footnote{Note that the source vertex of $e_{i+1}$ is equal to the destination vertex of $e_i$ for all $i \in \{1,\dots,n-1\}$, and the source vertex of $e_1$ is equal to the destination vertex of $e_n$.} we have:
\begin{equation}
\sum_{k = 1}^{n}{f(e_k)} = 0 \, ,
\end{equation}
where the length $n \in \N$ of the circuit is arbitrary.
\item For the fields $S = \R$ and $S = \Q$, the set of all $S$-edge functions forms an inner product space over $S$ with inner product given by:
\begin{equation}
\left<f,g\right> \triangleq \sum_{e \in E}{f(e) g(e)}
\end{equation}
for all pairs of $S$-edge functions $f,g$.
\item The \emph{divergence} of an edge function $f$ at a vertex $v \in \V$ is defined as:
\begin{equation}
\begin{aligned}
(\divergence f)(v) & \triangleq \sum_{\substack{e \in E:\\v \text{ has incoming edge } e}}{f(e)} \\
& \quad \quad - \sum_{\substack{e^{\prime} \in E:\\v \text{ has outgoing edge } e^{\prime}}}{f(e^{\prime})} \, .
\end{aligned}
\end{equation}
Furthermore, we write $\divergence f = 0$ to denote that $(\divergence f)(v) = 0$ for all vertices $v \in \V$.
\item For every vertex $v \in \V$, we define an associated $\Z$-edge function $\varrho_v$ as follows:
\begin{equation}
\label{Eq: Definition of varrho's}
\begin{aligned}
\forall e \in E, \enspace \varrho_v(e) & \triangleq \I\!\left\{v \text{ has incoming edge } e\right\} \\
& \quad \, - \I\!\left\{v \text{ has outgoing edge } e\right\} .
\end{aligned}
\end{equation}
These edge functions are actually gradients of characteristic (or indicator) functions of vertices in $\V$. They can also be perceived as rows of an oriented incidence matrix of $\G$. 
\item Lastly, for every directed circuit $H = (e_1,\dots,e_n) \in E^n$ in $\G$ (where $n \in \N$ is arbitrary), we define an associated $\N$-edge function $f_H$ as follows:
\begin{equation}
\forall e \in E, \enspace f_H(e) \triangleq \sum_{k = 1}^{n}{\I\!\left\{e_k = e\right\}} \, .
\end{equation}
$f_H$ encodes the circuit $H$ by counting the number of times each edge is traversed. (This parallels the encoding used in Lemma \ref{Lemma: Flow Space}.) Moreover, since the number of incoming edges equals the number of outgoing edges for every vertex in a circuit, it is straightforward to verify that $\divergence f_H = 0$.
\end{enumerate}

We will now state a result that includes Theorem~\ref{Thm: Equivalent Characterization of Cyclic Equivalence} as a
special case; the reduction is given at the end of this appendix.

\begin{theorem}[Closed $\R$-Edge Functions]
\label{Thm: Edge Functions}
Given any strongly connected finite directed graph $\G$, an edge function $f$ is closed if and only if there exist coefficients $\alpha_v \in \R$ for $v \in \V$ such that: 
$$ f = \sum_{v \in\V}{\alpha_v \varrho_v} \, . $$
The dimension of the linear subspace of closed edge functions equals $|\V|-1$. 
\end{theorem}

\begin{proof} 
We are trying to prove that the image of the linear map $\mathsf{G}:\R^{\V} \to \R^E$ given by:
$$ \mathsf{G}(\alpha) \triangleq \sum_{v \in \V} \alpha_v \varrho_v $$
coincides with the kernel of the map $\mathsf{H}:\R^E \to \R^m$ mapping an edge function to its evaluation on each of (finitely many, $m$) circuits in $\G$. Since $\R$ can be seen as an (infinite dimensional) $\Q$-vector space, it is a flat
$\Q$-module over $\Q$. Thus, by tensoring the exact sequence formed by $\mathsf{G},\mathsf{H}$ with $\R$, we see that equality of image of $\mathsf{G}$ with the kernel of $\mathsf{H}$ can be checked over $\Q$. From now on, we restrict our attention to $\Q$-edge functions.

Denote:
$$ V_1 = \linspan(\{\varrho_v : v\in V\})\,, \quad V_2 =\{f: \divergence f = 0\}\,.$$
We claim that: 
\begin{equation}
\label{eq:dpp_1}
	V_1^{\perp} = V_2\,, 
\end{equation}
or $f\in V_1$ if and only if $\left<f,g\right> = 0$ for all $g\in V_2$. Indeed, from the definition: 
$$ (\divergence g)(v) = 0 \quad \iff \quad \left<\varrho_v,g\right> =0\,.$$

In view of~\eqref{eq:dpp_1} we see that the statement of the theorem (over $\Q$) is equivalent to the following claim: \emph{A $\Q$-edge function $g$ satisfies $\divergence g = 0$ if and only if $g$ is a finite linear combination of $\{f_H : H \text{ is a circuit in }\G\}$ with rational coefficients.}
One direction is clear: $\divergence f_H = 0$ for any circuit $H$ in $\G$ since every vertex has the same number of incoming
and outgoing edges in a circuit. 

For the converse direction, first we consider a non-zero $\N$-edge function $g$ satisfying $\divergence g = 0$. Take any edge $e_1 \in E$ such that $g(e_1) > 0$, and let $e_1$ have destination vertex $v \in \V$. Then, since $(\divergence g)(v) = 0$, there exists an edge $e_2 \in E$ with source vertex $v$ such that $g(e_2) > 0$. We can repeatedly generate new edges $e \in E$ with $g(e) > 0$ using this argument until we revisit an already visited vertex. (This must happen because $\G$ is finite.) This yields a directed cycle $H = (e_1^{\prime},\dots,e_n^{\prime}) \in E^n$ with some length $n \in \N$ such
that $g(e_i^{\prime}) > 0$ for all $i \in \{1,\dots,n\}$. Notice that $g - f_H$ is also an $\N$-edge function. So, we
can iterate this procedure. This shows that any $\N$-edge function $g$ satisfies $\divergence g = 0$ if and only if $g$
can be written as:
$$ g = \sum_{\text{directed cycles } H}{\beta_{H} f_H} $$
for some choice of coefficients $\beta_H \in \N$, where the sum is over all directed cycles of $\G$.\footnote{This representation of $g$ can be construed as the converse of Lemma \ref{Lemma: Flow Space}, cf. \cite[Theorem 14.2.2, Corollary 14.2.3]{GodsilRoyle2001}.}

Now suppose that $g$ is a $\Q$-edge function satisfying $\divergence g = 0$ (as in the above claim). Let $J$ be a directed circuit passing through every edge of $\G$, which exists because $\G$ is strongly connected, and let $f_J$ be its corresponding $\N$-edge function. Then, since $\divergence f_J = 0$ and $f_J(e) > 0$ for all $e \in E$, for some sufficiently large integers $k_1,k_2 \in \N$, $g^{\prime} = k_1 g + k_2 f_{J}$ is an $\N$-edge function satisfying $\divergence g^{\prime} = 0$. Thus, $g^{\prime}$ can be represented as: 
$$ k_1 g + k_2 f_{J} = g^{\prime} = \sum_{\text{directed cycles } H}{\beta_{H} f_H} $$
for some choice of coefficients $\beta_H \in \N$, as argued above. This implies that $g$ is a finite linear combination
of $\{f_H : H \text{ is a circuit in } \G\}$ with rational coefficients. This establishes the claim.

The statement about dimension follows from \cite[Theorem 8.3.1]{GodsilRoyle2001}, which shows $\mathrm{dim}\, V_1 = |\V|-1$, after we use the facts that $\G$ is strongly connected and that $\{\varrho_v : v \in \V\}$ define the rows of the oriented incidence matrix of $\G$. This completes the proof.
\end{proof}

\renewcommand{\proofname}{Proof of Theorem \ref{Thm: Equivalent Characterization of Cyclic Equivalence}}

\begin{proof}
Akin to \eqref{Eq: Pure Rank Formal Sum}, consider any counting form $w$ with pure rank $s$ defined by the formal sum:
$$ w = \sum_{e \in \Y^s}{\alpha_e \{e\}} $$
with coefficients $\{\alpha_e \in \R : e \in \Y^s\}$. Let $\G$ be the strongly connected directed graph with vertex set $\Y^{s-1}$ and directed edge set $\Y^{s}$, where any edge $(v_1 \, \cdots \, v_s) \in \Y^s$ has source vertex $(v_1 \, \cdots \, v_{s-1})$ and destination vertex $(v_2 \, \cdots \, v_s)$. Note that $\G$ has the same underlying structure as the graph constructed in subsection \ref{Graph Theoretic Characterization}, but it identifies edges (rather than vertices) with basis clauses of rank $s$.\footnote{We used an alternative graph construction in our development in subsections \ref{Graph Theoretic Characterization} and \ref{Linear Programming Criteria} since it was arguably more natural.} Define the edge function $f : \Y^s \rightarrow \R$ corresponding to $w$ such that:
$$ \forall e \in \Y^s, \enspace f(e) = \alpha_e \, . $$
By treating $f$ as a weight function on the edges of $\G$, the proof of Proposition \ref{Prop: Graph Theoretic Characterization of Cyclic Partial Order} in subsection \ref{Graph Theoretic Characterization}, mutatis mutandis, shows that $w \cyceq \0$ if and only if $f$ is closed. (We omit the details of this argument for brevity.) Hence, applying Theorem \ref{Thm: Edge Functions} yields that $w \cyceq \0$ if and only if $f \in \linspan(\{\varrho_v : v \in \Y^{s-1}\})$. Now, using the structure of $\G$ and the definitions in \eqref{Eq: Def of rho_v} and \eqref{Eq: Definition of varrho's}, it is easy to see that for all $v \in \Y^{s-1}$, the $\Z$-edge function $-\varrho_v$ defines the coefficients of the counting form $\rho_v$:
$$ \rho_v = -\sum_{e \in \Y^{s}}{\varrho_v(e) \{e\}} \, . $$
Therefore, since $f$ encodes the coefficients of $w$, $w \cyceq \0$ if and only if $w \in \linspan(\{\rho_v : v \in
\Y^{s-1}\})$. Lastly, by Theorem~\ref{Thm: Edge Functions}, we get that $\linspan(\{\rho_v : v \in \Y^{s-1}\})$ has dimension $|\Y^{s-1}| - 1 = 3^{s-1} - 1$ (using \eqref{Eq: The coupled NAND alphabet}). This completes the proof.
\end{proof}

\renewcommand{\proofname}{Proof}

\section{Necessary Conditions to Satisfy \eqref{Eq: Cyc Form Cond 1} for Sufficiently Small Noise}
\label{Necessary Conditions to Satisfy Cyc Form Cond 1 for Small Delta}

In this appendix, we consider the setting of sufficiently small noise $\delta > 0$, which is intuitively the most challenging regime to establish the impossibility of broadcasting in. In particular, we will computationally construct a $u$-only counting form with rank $3$ which provides evidence for the existence of $u$-only counting forms that satisfy the supermartingale condition \eqref{Eq: Cyc Form Cond 1} for all sufficiently small $\delta > 0$. (Note that we do not consider the constraint \eqref{Eq: Cyc Form Cond 2} in this appendix.) 

To present our computer-assisted construction, define the function $T : \Y \rightarrow \{0,1,2\}$, $T(y) = 2\I\{y = u\} + y\I\{y \in \{0,1\}\}$, which changes $u$'s into $2$'s. Moreover, for the linear subspaces of counting forms with pure rank $s \in \N\backslash\!\{0\}$, define the ensuing maps that equivalently represent these counting forms as column vectors in $\R^{3^s}$:
\begin{align}
\label{Eq: The Coeff Map}
w & = \!\sum_{v \in \Y^s}{\!\alpha_v \{v\}} \enspace \mapsto \enspace \zeta_s(w) \triangleq \left[\tilde{\alpha}(0) \, \cdots \, \tilde{\alpha}(3^s - 1)\right]^{\T} \in \R^{3^s}, \\
\forall v & = (v_{s-1} \cdots v_{0}) \in \Y^s, \enspace \tilde{\alpha}\!\left(\sum_{i = 0}^{s-1}{T(v_i) \, 3^i}\right) = \alpha_v \, ,
\label{Eq: Index map}
\end{align}
where $\{\alpha_v \in \R : v \in \Y^s\}$ denote the coefficients of the counting form $w$ with pure rank $s$, and we write these coefficients as column vectors whose indices are the lexicographically ordered $s$-length strings interpreted as ternary representations of non-negative integers. 

Now consider the conditional expectation operator $\CE$ in Definition \ref{Def: Conditional Expectation Operator} acting on the linear subspace of all counting forms with pure rank $s$. Then, since the codomain of $\CE$ is the linear subspace of all counting forms with pure rank $s+1$, we can equivalently represent $\CE$ using a $3^{s+1} \times 3^s$ matrix $\mathbf{C}_s(\delta)$ such that for every counting form $w$ with pure rank $s$, we have:
\begin{equation}
\label{Eq: Matrix equation for forms}
\zeta_{s+1}(\CE(w)) = \mathbf{C}_s(\delta) \, \zeta_s(w) \, .
\end{equation}
Note that when $s = 1$, the matrix $\mathbf{C}_1(\delta)$ is given by:
\begin{equation}
\label{Eq: Basic CE}
\mathbf{C}_{1}(\delta) \triangleq \!
\bbordermatrix{
	       & 0 & 1 & u \cr
   (0,0) & \delta^2 & 1 - \delta^2 & 0 \cr 
   (0,1) & \delta(1-\delta) & \delta + (1-\delta)^2 & 0 \cr
   (0,u) & \delta^2 & \delta + (1-\delta)^2 & \delta(1 - 2\delta) \cr
   (1,0) & \delta(1-\delta) & \delta + (1-\delta)^2 & 0 \cr
   (1,1) & (1-\delta)^2 & 2\delta - \delta^2 & 0 \cr
   (1,u) & \delta(1-\delta) & 2\delta - \delta^2 & (1-\delta)(1 - 2\delta) \cr
   (u,0) & \delta^2 & \delta + (1-\delta)^2 & \delta(1 - 2\delta) \cr
   (u,1) & \delta(1-\delta) & 2\delta - \delta^2 & (1-\delta)(1 - 2\delta) \cr
   (u,u) & \delta^2 & 2\delta - \delta^2 & 1-2\delta \cr}
\end{equation}
for all $\delta \in \big[0,\frac{1}{2}\big]$, where we label the indices of the rows and columns using $\Y^2$ and $\Y$, respectively, for illustrative purposes. (Note, for example, that the rows are actually indexed by $[9]$ since each $(v_1 \, v_0) \in \Y^2$ corresponds to a non-negative integer $T(v_0) + 3 T(v_1)$ using \eqref{Eq: Index map}.) Furthermore, when $s \geq 2$, for every $v = (v_{s} \cdots v_0) \in \Y^{s+1}$ and every $y = (y_{s-1} \cdots y_0) \in \Y^{s}$, the $(a,b)$th element of the matrix $\mathbf{C}_s(\delta)$, with $a = \sum_{i = 0}^{s}{T(v_i) 3^i}$ and $b = \sum_{j = 0}^{s-1}{T(y_j) 3^j}$, is given by:
\begin{align}
\left[\mathbf{C}_s(\delta)\right]_{a,b} & \triangleq \prod_{k = 0}^{s-1}{\left[\mathbf{C}_1(\delta)\right]_{T(v_k) + 3 T(v_{k+1}),T(y_k)}} \nonumber \\
& = \P\!\left((Y_{s+1,1},\dots,Y_{s+1,s}) = v \, \middle| \, Y_s = y\right)
\label{Eq: Useful representation of CE matrix}
\end{align}
for all $\delta \in \big[0,\frac{1}{2}\big]$, according to Definition \ref{Def: Conditional Expectation Operator}. Since $\mathbf{C}_1(\delta)$ is a quadratic \emph{polynomial matrix} in $\delta$ (see \eqref{Eq: Basic CE}), it follows that $\mathbf{C}_s(\delta)$ is a polynomial matrix in $\delta$ with degree $2s$, and we may write it as:
\begin{equation}
\label{Eq: Polynomial Matrix}
\forall \delta \in \left[0,\frac{1}{2}\right], \enspace \mathbf{C}_s(\delta) = \sum_{k = 0}^{2s}{\mathbf{C}_s^{(k)} \, \delta^k}
\end{equation}
where each $3^{s+1} \times 3^s$ coefficient matrix $\mathbf{C}_s^{(k)}$ does not depend on $\delta$ (and is explicitly known).

Next, suppose we seek to find a $u$-only counting form $w_{\delta}$ with pure rank $s$ satisfying \eqref{Eq: Cyc Form Cond 1} that can be represented as a linear polynomial vector:
\begin{equation}
\label{Eq: First order form}
\forall \delta \in \left[0,\frac{1}{2}\right], \enspace \zeta_{s}(w_{\delta}) = \hat{w}_0 + \hat{w}_1 \delta
\end{equation} 
where $\hat{w}_0,\hat{w}_1 \in \R^{3^s}$ do not depend on $\delta$. Define the $3^{s+1} \times 3^s$ ``purification matrix'' $P_s$ via (see Lemma \ref{Lemma: Purification of Counting Forms} and the discussion preceding Theorem \ref{Thm: Equivalent Characterization of Cyclic Equivalence}):
\begin{equation}
\label{Eq: The Purify Matrix}
P_s \triangleq I_{s} \otimes [1 \, 1 \, 1]^{\T}
\end{equation}
where $I_s$ denotes the $3^s \times 3^s$ identity matrix, and $\otimes$ denotes the \emph{Kronecker product} of matrices. Note that for any $u$-only counting form $w$ with pure rank $s$, $P_s \zeta_s(w)$ produces the coefficients of a $u$-only counting form with pure rank $s+1$ that is cyclically equivalent to $w$. To satisfy the supermartingale condition \eqref{Eq: Cyc Form Cond 1}, $w_{\delta} - \CE(w_{\delta}) \cycgeq \0$, observe that it suffices to ensure that:
\begin{equation}
\label{Eq: Feasibility Problem Form 1}
\begin{aligned}
& \exists \, z_{\delta} \in \linspan\!\left(\left\{\rho_{v} : v \in \Y^{s} \backslash \{0,1\}^{s}\right\}\right) , \\
& \qquad \quad \underbrace{P_s \zeta_s(w_{\delta})}_{\cyceq w_{\delta}} - \underbrace{\mathbf{C}_s(\delta) \zeta_s(w_{\delta})}_{= \, \zeta_{s+1}(\CE(w_{\delta}))} + \underbrace{\zeta_{s+1}(z_{\delta})}_{\substack{\text{does not change}\\w_{\delta} - \CE(w_{\delta})}} \geq \0
\end{aligned}
\end{equation}
for all $\delta \in \big[0,\frac{1}{2}\big]$, where we utilize the special case of the converse direction of Theorem \ref{Thm: Equivalent Characterization of Cyclic Equivalence} for $u$-only counting forms (and \eqref{Eq: Matrix equation for forms}), and the $\geq$ relation holds entry-wise. Indeed, if all the coefficients of a counting form that is cyclically equivalent to $w_{\delta} - \CE(w_{\delta})$ are non-negative, then $w_{\delta} - \CE(w_{\delta}) \cycgeq \0$ using Definitions \ref{Def: Cyclic Counting Forms} and \ref{Def: Cyclic Partial Order}. To simplify \eqref{Eq: Feasibility Problem Form 1} further, observe using \eqref{Eq: Matrix equation for forms}, \eqref{Eq: Polynomial Matrix}, and \eqref{Eq: First order form} that:
\begin{align}
\zeta_{s+1}(\CE(w_{\delta})) & = \left(\sum_{k = 0}^{2s}{\mathbf{C}_s^{(k)} \, \delta^k}\right) \! (\hat{w}_0 + \hat{w}_1 \delta) \nonumber \\
& = \mathbf{C}_s^{(0)} \hat{w}_0 + \left(\mathbf{C}_s^{(0)} \hat{w}_1 + \mathbf{C}_s^{(1)} \hat{w}_0\right) \! \delta + O\!\left(\delta^2\right) , \\
P_s \zeta_s(w_{\delta}) & = P_s \hat{w}_0 + P_s \hat{w}_1 \delta \, ,
\end{align}
where $O(\delta^2)$ is the standard big-$O$ asymptotic notation (as $\delta \rightarrow 0^+$). This implies that:
\begin{equation}
\label{Eq: Simplify Feasibility Problem Form 1}
\begin{aligned}
& P_s \zeta_s(w_{\delta}) - \zeta_{s+1}(\CE(w_{\delta})) \\
& = \left(P_s - \mathbf{C}_s^{(0)}\right) \! \hat{w}_0 + \left(\left(P_s - \mathbf{C}_s^{(0)}\right) \! \hat{w}_1 - \mathbf{C}_s^{(1)} \hat{w}_0\right) \! \delta \\
& \quad \, + O\!\left(\delta^2\right) .
\end{aligned}
\end{equation}
Furthermore, every $u$-only counting form $z_{\delta} \in \linspan\!\big(\big\{\rho_{v} : v \in \Y^{s} \backslash \{0,1\}^{s}\big\}\big)$ can be represented as $\zeta_{s+1}(z_{\delta}) = \mathbf{D}_s \tilde{x}_{\delta}$ for some $\delta$-dependent vector $\tilde{x}_{\delta} \in \R^{3^s - 2^s}$, where $\mathbf{D}_s$ denotes the (explicitly known and $\delta$-independent) $3^{s+1} \times (3^s - 2^s)$ matrix whose columns are given by:
\begin{equation}
\label{Eq: Simplify Feasibility Problem Form 2}
\mathbf{D}_s = \left[\zeta_{s+1}(\rho_v) : v \in \Y^{s} \backslash \{0,1\}^{s}\right] .
\end{equation}
Suppose $\tilde{x}_{\delta} = x_0 + x_1 \delta$ for $x_0,x_1 \in \R^{3^s - 2^s}$ that do not depend on $\delta$. Then, in the limit when $\delta \rightarrow 0^+$, using \eqref{Eq: Simplify Feasibility Problem Form 1} and \eqref{Eq: Simplify Feasibility Problem Form 2}, \eqref{Eq: Feasibility Problem Form 1} can be written as:
\begin{equation}
\label{Eq: Feasibility Problem Form 2}
\begin{aligned}
& \exists \, x_0,x_1 \in \R^{3^s - 2^s} , \\
& \qquad \left(\left(P_s - \mathbf{C}_s^{(0)}\right) \! \hat{w}_0 + \mathbf{D}_s x_0\right) \\
& \qquad + \left(\left(P_s - \mathbf{C}_s^{(0)}\right) \! \hat{w}_1 - \mathbf{C}_s^{(1)} \hat{w}_0 + \mathbf{D}_s x_1\right) \! \delta \geq \0
\end{aligned}
\end{equation}
where we neglect the $O(\delta^2)$ terms. 

\begin{algorithm*}[t] 
\caption{Algorithm to determine whether a $u$-only counting form with pure rank is cyclically non-negative.}
\label{Algorithm: Bellman-Ford}
\begin{algorithmic}[1]
\renewcommand{\algorithmicrequire}{\textbf{Input 1:}}
\Require A rank $r \in \N\backslash\!\{0\}$
\renewcommand{\algorithmicrequire}{\textbf{Input 2:}}
\Require An array of coefficients $(\alpha[v] \in \R : v \in \Y^r)$ corresponding to a $u$-only counting form $w$ with pure rank $r$, where $\alpha[v] = 0$ if $v \in \{0,1\}^r$ \Comment{$\alpha[v]$ is the coefficient of the basis clause $\{v\}$ for $w$}
\renewcommand{\algorithmicensure}{\textbf{Boolean Output:}}
\Ensure \True (if $w \cycgeq \0$) or \False (otherwise)
\Statex \textbf{\emph{Stage 1}} \Comment{Construct edge set $E$ of the directed graph $\G_r(w)$} 
\State \textbf{Initialization:} $E = \{ \}$
\ForAll{source vertices $v = (v_1 \cdots v_r) \in \Y^r$} 
\ForAll{letters $y \in \Y$} 
\State $v^{(y)} = (v_2 \cdots v_r \, y)$ \Comment{Construct destination vertex}
\State $E.$\Call{insert}{}$\big\{\big(v,v^{(y)}\big)\big\}$ \Comment{Insert directed edge into $E$}
\EndFor
\EndFor
\Statex \textbf{\emph{Stage 2}} \Comment{\emph{Bellman-Ford algorithm} to detect negative cycles reachable from $(u \cdots u)$} 
\State \textbf{Initialization:} distances[$v$] $\gets +\infty$ for all $v \in \Y^r$ \Comment{Shortest path lengths from $(u \cdots u)$ to $v$}
\State \textbf{Initialization:} distances[$(u \cdots u)$] $\gets 0$ \Comment{$(u \cdots u)$ is the source vertex}
\For{$i = 1$ \textbf{to} $3^r - 1$} \Comment{Loop to \emph{relax} edges}
\ForAll{edges $(v_1,v_2) \in E$}
\If{distances[$v_1$] $+$ $\alpha[v_1]$ $<$ distances[$v_2$]} \Comment{Edge weights are given by the relation \eqref{Eq: Weight Construction for G}}
\State distances[$v_2$] $\gets$ distances[$v_1$] $+$ $\alpha[v_1]$
\EndIf
\EndFor
\EndFor
\ForAll{edges $(v_1,v_2) \in E$} \Comment{Loop to detect negative cycles reachable from $(u \cdots u)$}
\If{distances[$v_1$] $+$ $\alpha[v_1]$ $<$ distances[$v_2$]}
\State \Return \False \Comment{Negative cycle found!}
\EndIf
\EndFor
\State \Return \True \Comment{No negative cycles found}
\end{algorithmic}
\end{algorithm*}

To simplify \eqref{Eq: Feasibility Problem Form 2} further, let us set $s = 3$. So, we seek to find a $u$-only counting form $w_{\delta}$ of the form \eqref{Eq: First order form} with pure rank $3$, that satisfies \eqref{Eq: Cyc Form Cond 1} up to first order in the limit as $\delta \rightarrow 0^+$. Recall the $u$-only harmonic counting form $w_*$ with rank $3$ defined in \eqref{Eq: Holroyd potential}, and notice that the proof of Proposition \ref{Prop: Harmonic Counting Form} can be easily modified (using purification with respect to $\cyceq$\textemdash see the discussion preceding Theorem \ref{Thm: Equivalent Characterization of Cyclic Equivalence}) to yield $w_* \cyceq \CE(w_*)$ when $\delta = 0$. Hence, let us restrict ourselves to $u$-only counting forms $w_{\delta}$ with some fixed $w_0 = w_0^* \cyceq w_*$ so that $w_0^* \cyceq \CE(w_0^*)$, where $w_0^*$ is $u$-only and has pure rank $3$. Specifically, we use:
\begin{align}
\hat{w}_0 & = \hat{w}_0^* \triangleq \zeta_3(w_0^*) \nonumber \\
& = \big[0 \enspace 0 \enspace 0 \enspace 0 \enspace 0 \enspace 1 \, -\!2 \enspace 0 \enspace 0 \enspace 0 \enspace 0 \enspace 0 \enspace 0 \enspace 0 \enspace 0 \enspace 1 \enspace 1 \enspace 1 \enspace 2 \enspace 2 \enspace 2 \nonumber \\
& \quad \enspace \,\, 4 \enspace 3 \enspace 3 \enspace 2 \enspace 2 \enspace 2\big]^{\T} \in \R^{27} \, .
\end{align}
With this choice of $\hat{w}_0 = \hat{w}_0^*$, since $w_0^* \cyceq \CE(w_0^*)$, we know that the counting form corresponding to $\big(P_3 - \mathbf{C}_3^{(0)}\big) \hat{w}_0^* + \mathbf{D}_3 x_0$ is $\cyceq \0$ in \eqref{Eq: Feasibility Problem Form 2}. Therefore, in the limit when $\delta \rightarrow 0^+$, the problem of finding $\hat{w}_1$ such that \eqref{Eq: Feasibility Problem Form 2} is satisfied can be recast as:
\begin{equation}
\label{Eq: Feasibility Problem Form 3}
\begin{aligned}
& \text{Find } \hat{w}_1 \in \R^{27} \text{ and } x_1 \in \R^{19} \text{ such that:} \\
& \left(P_3 - \mathbf{C}_3^{(0)}\right) \! \hat{w}_1 - \mathbf{C}_3^{(1)} \hat{w}_0^* + \mathbf{D}_3 x_1 \geq \0 \, , \\
& \, \, \forall (v_2 \, v_1 \, v_0) \in \{0,1\}^3, \enspace \big(\hat{w}_1\big)_{v_0 + 3 v_1 + 9 v_2} = 0
\end{aligned}
\end{equation}
where $(\hat{w}_1)_{i}$ denotes the $i$th entry of $\hat{w}_1$ for $i \in [27]$, the matrices $P_3 - \mathbf{C}_3^{(0)} \in \R^{81 \times 27}$, $\mathbf{D}_3 \in \R^{81 \times 19}$ and the vector $\mathbf{C}_3^{(1)} \hat{w}_0^* \in \R^{81}$ are known, and the equality constraints in \eqref{Eq: Feasibility Problem Form 3} ensure that $\hat{w}_1$ corresponds to a $u$-only counting form. 

The problem in \eqref{Eq: Feasibility Problem Form 3} is a \emph{linear programming feasibility} problem. Using MATLAB, we can solve \eqref{Eq: Feasibility Problem Form 3} (with additional constraints) to obtain the following integer-valued solution:
\begin{equation}
\label{Eq: Solution 1}
\begin{aligned}
\hat{w}_1 = \hat{w}_1^* & \triangleq \big[0 \enspace 0 \enspace 2 \enspace 0 \enspace 0 \enspace 4 \enspace 4 \enspace 4 \enspace 3 \enspace 0 \enspace 0 \enspace 4 \enspace 0 \enspace 0 \enspace 4 \enspace 4 \enspace 4 \\
& \quad \enspace \,\, 4 \enspace 2 \enspace 4 \enspace 4 \enspace 4 \enspace 4 \enspace 4 \enspace 3 \enspace 4 \enspace 2 \big]^{\T} \in \N^{27} 
\end{aligned}
\end{equation}
which has a corresponding integer-valued solution $x_1 = x_1^* \in \Z^{19}$ that we do not reproduce here. Therefore, the $u$-only counting form $w_{\delta}^* \triangleq \zeta_3^{-1}(\hat{w}_0^* + \hat{w}_1^* \delta)$ with pure rank $3$ has the property that it satisfies \eqref{Eq: Cyc Form Cond 1} up to first order as $\delta \rightarrow 0^+$. This provides numerical evidence for the existence of $u$-only counting forms with rank $3$ satisfying \eqref{Eq: Cyc Form Cond 1} for all values of $\delta$ in a sufficiently small neighborhood of $0$. Indeed, if we could construct solutions $\hat{w}_1^*$ and $x_1^*$ that satisfied the entry-wise strict inequality:
\begin{equation}
\label{Eq: Local Dream}
\left(P_3 - \mathbf{C}_3^{(0)}\right) \! \hat{w}_1^* - \mathbf{C}_3^{(1)} \hat{w}_0^* + \mathbf{D}_3 x_1^* > \0 \, , 
\end{equation} 
then we would obtain that for all sufficiently small $\delta > 0$:
\begin{equation}
\left(\left(P_3 - \mathbf{C}_3^{(0)}\right) \! \hat{w}_1^* - \mathbf{C}_3^{(1)} \hat{w}_0^* + \mathbf{D}_3 x_1^*\right) \!\delta + O\!\left(\delta^2\right) \geq \0 \, , 
\end{equation}
which includes all the higher order terms in \eqref{Eq: Simplify Feasibility Problem Form 1}. This would in turn prove that $w_{\delta}^*$ satisfies \eqref{Eq: Cyc Form Cond 1} for all sufficiently small $\delta > 0$. Since our solution \eqref{Eq: Solution 1} does not satisfy \eqref{Eq: Local Dream}, it only \emph{suggests} the existence of $u$-only counting forms with rank $3$ that satisfy \eqref{Eq: Cyc Form Cond 1} for all sufficiently small $\delta > 0$.

\section{Algorithm to Detect Cyclically Non-negative Counting Forms}
\label{Algorithm to Detect Non-negative Cyclic Counting Forms}

In this appendix, we present some pseudocode in Algorithm \ref{Algorithm: Bellman-Ford} for an algorithm that computes whether a given $u$-only counting form with pure rank $r$ is cyclically non-negative in the sense of Definition \ref{Def: Cyclic Partial Order}. Specifically, given the coefficients of a $u$-only counting form $w$ with pure rank $r$ as input, Algorithm \ref{Algorithm: Bellman-Ford} first constructs the directed graph $\G_r(w)$ by adding three outgoing edges to each vertex in $\Y^r$, even if the coefficient corresponding to the vertex (or basis clause) is zero, i.e. even if the outgoing edges are assigned zero weight. (Indeed, edges with zero weight are not equivalent to non-existent edges, because the latter kind of edges can be construed as having infinite weight.) Then, Algorithm \ref{Algorithm: Bellman-Ford} executes the Bellman-Ford algorithm, cf. \cite[Section 24.1]{Cormenetal2009}, to find negative cycles in $\G_r(w)$ with edge weights given by the relation \eqref{Eq: Weight Construction for G} and source vertex chosen as $(u \cdots u)$ (all $u$'s), which is a reasonable choice since we only consider $u$-only counting forms $w$. The correctness of this algorithm follows from part 2 of Proposition \ref{Prop: Graph Theoretic Characterization of Cyclic Partial Order} and the correctness of the Bellman-Ford algorithm (in detecting the existence of negative cycles). We also remark that it is straightforward to modify Algorithm \ref{Algorithm: Bellman-Ford} to return an explicit negative cycle in $\G_r(w)$ when such cycles exist.

Since the correctness of the Bellman-Ford algorithm is typically proved for directed \emph{simple} graphs, it is worth explicitly mentioning that the self-loops on the vertices $(0 \cdots 0)$, $(1 \cdots 1)$, and $(u \cdots u)$ in $\G_r(w)$ do not affect the correctness of this algorithm. As mentioned earlier in subsection \ref{Graph Theoretic Characterization}, the self-loop edges on $(0 \cdots 0)$ and $(1 \cdots 1)$ have weight $0$, and never satisfy the \textbf{if} clause conditions in steps 12 and 18 of Algorithm \ref{Algorithm: Bellman-Ford}. Likewise, when the self-loop edge on $(u \cdots u)$ has non-negative weight, it also never satisfies the \textbf{if} clause conditions in steps 12 and 18. Hence, the algorithm runs correctly in this scenario. On the other hand, if the self-loop edge on $(u \cdots u)$ has negative weight, then it does satisfy the \textbf{if} clause condition in step 18, and the algorithm correctly discovers a negative cycle.

\balance 

\bibliographystyle{IEEEtran}
\bibliography{BroadcastRefsGrid}

\end{document}